\documentclass[11pt]{amsart}
\usepackage{amsmath}
\usepackage{bbm}
\usepackage{mathrsfs}
\usepackage[english]{babel}
\usepackage[utf8]{inputenc}
\usepackage[T1]{fontenc}
\usepackage{wrapfig}

\usepackage{ dsfont }

\usepackage[top=3cm,bottom=2cm,left=3cm,right=3cm,marginparwidth=1.75cm]{geometry}
\usepackage{amsmath,amssymb}
\usepackage{graphicx}
\usepackage[dvipsnames]{xcolor}
\usepackage[colorlinks=true, allcolors=BlueViolet, pagebackref=true]{hyperref}
\renewcommand*\backref[1]{\ifx#1\relax \else (Cited on p.#1) \fi}
\usepackage{cite}
\usepackage[shortlabels]{enumitem}
\usepackage[all,cmtip]{xy}
\usepackage{tikz-cd}
\usepackage{lmodern}
\usepackage{pgfplots}
\pgfplotsset{compat=1.17}
\usetikzlibrary{intersections}
\usepackage{relsize}
\usepackage{subcaption}
\usepackage{dsfont}

\usepackage{xcolor, import}
\graphicspath{{Figures/}}

\usepackage{cleveref}
\usepackage{autonum}
\usepackage{xspace}

\usepackage{cancel}
\usepackage{ulem}
\usepackage{fancyvrb}
\newtheorem{theorem}{Theorem}[section]
\newtheorem{proposition}[theorem]{Proposition}
\newtheorem{corollary}[theorem]{Corollary}
\newtheorem{lemma}[theorem]{Lemma}

\newtheorem{definition}[theorem]{Definition}

\newtheorem{innercustomthm}{Theorem}
\newenvironment{customthm}[1]
  {\renewcommand\theinnercustomthm{\ref{#1}}\innercustomthm}
  {\endinnercustomthm}

\newtheorem{mainthm}{Theorem}

\theoremstyle{remark} 
\newtheorem{remark}[theorem]{Remark}
\newtheorem{example}[theorem]{Example}

\newtheorem{question}[]{Question}

\crefname{equation}{Equation}{}
\crefname{figure}{Figure}{Figures}
\crefname{question}{Question}{Question}
\crefname{section}{Section}{Sections}
\crefname{subsection}{Subsection}{Subsections}
\crefname{lemma}{Lemma}{Lemmas}
\crefname{proposition}{Proposition}{Propositions}
\crefname{theorem}{Theorem}{Theorems}
\crefname{innercustomthm}{Theorem}{Theorems}
\crefname{mainthm}{Theorem}{Theorems}
\crefname{corollary}{Corollary}{Corollaries}
\crefname{definition}{Definition}{Definitions}
\crefname{remark}{Remark}{Remarks}
\crefname{proposition}{Proposition}{Proposition}
\crefname{corollary}{Corollary}{Corollaries}
\crefname{example}{Example}{Examples}
\crefname{conjecture}{Conjecture}{Conjectures}
\crefname{enumi}{}{}

\numberwithin{equation}{section}
\usepackage{amsmath}

\newcommand{\leonew}[1]{{\textcolor{WildStrawberry}{#1}}}

\newcommand{\miknew}[1]{{\textcolor{Green}{#1}}}

\newcommand{\R}{\mathbb{R}}

\newcommand{\N}{\mathbb{N}}
\newcommand{\PP}{\mathbb{P}}
\newcommand{\ZZ}{\mathscr{Z}}
\newcommand{\ZZo}{\mathscr{Z}_0}

\newcommand{\dd}{\mathrm{d}}

\newcommand{\EE}{\mathbb{E}}
\newcommand{\F}{\mathcal{F}}

\newcommand{\f}{\varphi}
\renewcommand{\a}{\alpha}

\newcommand{\uball}{b}
\newcommand{\usph}{s}

\newcommand{\Vint}{\mathcal{V}}

\def\randin{%
  \mathchoice%
    {\raisebox{-.35ex}{$\displaystyle{^\subset}$}\mkern-11.5mu\raisebox{+.45ex}{$\displaystyle{_\subset}$}}
    {\mkern+1mu\raisebox{-.27ex}{$\textstyle{^\subset}$}\mkern-11.7mu\raisebox{+.45ex}{$\textstyle{_\subset}$}}
    {\raisebox{.35ex}{$\scriptstyle\subset$}\mkern-14mu\raisebox{-.15ex}{$\scriptstyle\subset$}}
    {\raisebox{.3ex}{$\scriptscriptstyle\subset$}\mkern-13.5mu\raisebox{-.10ex}{$\scriptscriptstyle\subset$}}
}

\newcommand{\maschera}{\textcolor{white}{\scalebox{0.3}{$\blacktriangle$}}}
\def\FlatOmega{%
  \mathchoice{%
    \displaystyle{\Omega}\mkern-14mu\raisebox{+.166ex}{$\displaystyle{\maschera}$}
    \mkern+7mu\raisebox{+.166ex}{$\displaystyle{\maschera}$}}{
    \hbox{$\textstyle{\Omega}$}\mkern-14mu\raisebox{+.166ex}{\hbox{$\textstyle{\maschera}$}}
    \mkern+7mu\raisebox{+.166ex}{\hbox{$\textstyle{\maschera}$}}
    }
    {
   \scriptstyle{\Omega}
    \mkern-14mu\raisebox{+.13ex}{$\scriptstyle{\maschera}$}
    \mkern+5mu\raisebox{+.13ex}{$\scriptstyle{\maschera}$}
    }
    {
    \scriptscriptstyle{\Omega}
    \mkern-14mu\raisebox{+.13ex}{$\scriptscriptstyle{\maschera}$}
    \mkern+5mu\raisebox{+.13ex}{$\scriptscriptstyle{\maschera}$}
    }
}
\newcommand{\scaledFlatOmega}{{\scalebox{0.8}{$_{\FlatOmega}$}}}
\def\cleanrandto{%
  \mathchoice%
    {
    \raisebox{-.101ex}{$\displaystyle{-}$}\mkern-4.4mu
    \raisebox{.729ex}{$\displaystyle{\scaledFlatOmega}$}\mkern-5.2mu
    \raisebox{-.101ex}{$\displaystyle\to$}
    }
    {
    \raisebox{-.101ex}{\hbox{$\textstyle{-}$}}\mkern-4.4mu
    \raisebox{.729ex}{\hbox{$\textstyle{\scaledFlatOmega}$}}\mkern-5.2mu
    \raisebox{-.101ex}{\hbox{$\textstyle\to$}}
    }
    {
    \raisebox{-.101ex}{$\scriptstyle{-}$}\mkern-4.4mu
    \raisebox{.729ex}{$\scriptstyle{\scaledFlatOmega}$}\mkern-5.2mu
    \raisebox{-.101ex}{$\scriptstyle\to$}
    }
    {
    \raisebox{-.101ex}{$\scriptscriptstyle{-}$}\mkern-4.4mu
    \raisebox{.729ex}{$\scriptscriptstyle{\scaledFlatOmega}$}\mkern-5.2mu
    \raisebox{-.101ex}{$\scriptscriptstyle\to$}
    }
}
\newcommand{\randto}{\cleanrandto}

\newcommand{\eps}{\varepsilon}

\newcommand{\set}[2]{\left\{ #1 \ | \ #2\right\}}

\newcommand{\Me}{\mathcal{M}}
\newcommand{\Mev}{\mathcal{M}_{even}}
\newcommand{\one}{\mathds{1}}
\newcommand{\GZ}{\mathcal{G}}
\newcommand{\GZo}{\mathcal{G}_0}

\newcommand{\zkrok}{\emph{z-KROK}\xspace}

\newcommand{\simto}{\overset{\sim}{\longrightarrow}}

\newcommand{\y}{\mathbbm{y}}

\newcommand\restr[2]{{
  \left.\kern-\nulldelimiterspace 
  #1 
  \vphantom{\big|} 
  \right|_{#2} 
  }}

\DeclareMathOperator{\vol}{vol}
\DeclareMathOperator{\Span}{Span}

\DeclareMathOperator{\MV}{MV}

\DeclareMathOperator{\Id}{Id}

\newcommand{\seg}{\underline}

\newcommand{\transv}{\mathrel{\text{\tpitchfork}}}
\makeatletter
\newcommand{\tpitchfork}{%
  \raise-0.1ex\vbox{
    \baselineskip\z@skip
    \lineskip-.52ex
    \lineskiplimit\maxdimen
    \m@th
    \ialign{##\crcr\hidewidth\smash{$-$}\hidewidth\crcr$\pitchfork$\crcr}
  }%
}
\newcommand{\mC}{\mathcal{C}}
\newcommand{\be}{\begin{equation}}
\newcommand{\ee}{\end{equation}}

\newcommand{\bega}{\begin{equation}\begin{aligned}}
\newcommand{\eega}{\end{aligned}\end{equation}}
\newcommand{\kop}{\left\{}
\newcommand{\pok}{\right\}}
\newcommand{\tyu}{\left(}
\newcommand{\uyt}{\right)}
\newcommand{\qwe}{\left[}
\newcommand{\ewq}{\right]}
\makeatother

\everypar{\looseness=-1}
\title{Expectation of a random submanifold: the zonoid section}
\author{Léo Mathis \and Michele Stecconi}
\date{1/2/2023}
\begin{document}
\maketitle
\begin{abstract}
We develop a calculus based on zonoids -- a special class of convex bodies -- for the expectation of functionals related to a random submanifold $Z$ defined as the zero set of a smooth vector valued random field on a Riemannian manifold. We identify a convenient set of hypotheses on the random field under which we define its \emph{zonoid section}, an assignment of a zonoid $\zeta(p)$ in the exterior algebra of the cotangent space at each point $p$ of the manifold. We prove that the first intrinsic volume of $\zeta(p)$ is the Kac-Rice density of the expected volume of $Z$, while its center computes the expected current of integration over $Z$. We show that the intersection of random submanifolds corresponds to the wedge product of the zonoid sections and that the preimage corresponds to the pull-back.

Combining this with the recently developed \emph{zonoid algebra}, it allows to give a multiplication structure to the Kac-Rice formulas, resembling that of the cohomology ring of a manifold. Moreover, it establishes a connection with the theory of convex bodies and valuations, which includes deep results such as the Alexandrov-Fenchel inequality and the Brunn-Minkowsky inequality. We export them to this context to prove two analogous new inequalities for random submanifolds. Applying our results in the context of \emph{Finsler geometry}, we prove some new Crofton formulas for the length of curves and the Holmes-Thompson volumes of submanifolds in a Finsler manifold.
\end{abstract}
\tableofcontents
\section{Introduction}\label{sec:intro}
\subsection{Overview}

Let $X\colon M\to \mathbb{R}^k$ be a random smooth function on a smooth Riemannian manifold $M$. Under the hypothesis that the random subset $Z:=X^{-1}(0)$ is almost surely a submanifold, we study the following functionals:
\begin{equation}\label{eq:rhoalpha1}
    A\mapsto\EE\left\{\vol_{(m-k)}(Z\cap A)\right\}, \qquad \omega\mapsto \EE\left\{\int_{Z} \omega\right\},
\end{equation} 
where $A\subset M$ is any Borel subset and $\omega$ is any smooth differential $(m-k)$-form with compact support, that is, $\omega\in \Omega_c^{m-k}(M)$.
In more fancy words, the former is the measure obtained by taking the expectation of the random measure "$(m-k)$-volume of the intersection with $Z$"; while the latter, which is defined whenever $Z$ is oriented, is the current obtained by taking the expectation of the random current $\int_Z\in \Omega_c^{m-k}(M)^*$.
Our aim is not just to find formulas for them two, but to establish a framework to understand the relations among them for multiple instances of $Z$.

\subsubsection{The examples that we have in mind}\label{subsub:exmind}

There is a vast literature dedicated to the study of nodal sets of random fields \cite{AdlerTaylor, AzaisWscheborbook, marinucci_peccati_2011, bogachev}. The first example in our mind is Kostlan polynomials \cite{kostlan:93}, studied in relation with real algebraic geometry \cite{ShSm1,shsm,ShSm3}, \cite{GaWe1, GaWe2,GaWe3},  \cite{NazarovSodin2009,NazarovSodin2016}, \cite{Lerariolemniscate, LeLu:gap,FyLeLu,basu2019random,kozhasov2020number,breiding2018geometry}, \cite{stec2019MaxTyp};
then, random submanifolds in homogenous spaces and integral geometry \cite{Brgisser2016ProbabilisticSC,lerario2021probabilistic, bernig2014integral};
random eigenfunctions and Riemannian random waves \cite{Zelditch2009RealAC,Berry_1977}, a topic that in the current years is at the center of a lot of attention, see \cite{Wigman_2010, MaWi2011defect, MaWi2011excarea, Wig2013arith, MaWi2014, CammarotaM2015, MaPeRoWi2015, WigRud2016,MAFFUCCI2017, CamMari2018Berry, SarnakWigman2019, NourdinPeccatiRossi2019, MaRoWi2020Berry, CaHa2020,Gass2020,WigYesha2021,MarRossVidotto2021}
and the surveys \cite{WigSurvey2010,WigSurvey2022,canzani2019probabilistic, Ma2021surveyECS}.
 The vast majority of these works deals with Gaussian random fields \cite{NazarovSodin2016,dtgrf, Nicolaescu, NicSavale, MaxNotarnickHermite,bogachev}.  The methods and the results proposed in this paper are aimed to a general study of random fields including non-Gaussian situations, see for instance \cite{KabluSartoWig2022,stecconi2021isotropic}.

Our results are also to be compared with the work of Akhiezer and Kazarnovskii \cite{KazaAveragezeroes}. Their \emph{average number of zeros}, corresponds, in our case, to the average number of zeroes of a system of independent scalar Gaussian random fields in finite dimensional function spaces. In \cite{Kazarnovskii2020}, a more general distribution than Gaussian is covered although it remains in the setting of scalar fields in finite dimensional function spaces.
It is yet unclear for us if Kazarnovskii's ``B-bodies'' correspond to our zonoid section.
\subsection{Main results}

\subsubsection{Expected length and currents}

We propose to study the functionals in \cref{eq:rhoalpha1} using zonoids - a special family of convex bodies (see \cref{sec:zonoids}). A convex body is a zonoid if it can be approximated, in the Hausdorff topology, by a finite Minkowski sums of segments. To any regular enough random function $X\colon M\to \mathbb{R}^k$ we associate a field of convex bodies in the exterior algebra of the cotangent space:
\begin{equation}
M\ni p\mapsto \zeta_X(p)\subset \Lambda^kT^*_pM.
\end{equation}
For any $p\in M$, the convex body $\zeta_X(p)$ is a zonoid defined as the expectation of a random segment, via the following formula (\cref{def:zonoidsec}):
\begin{equation}\label{eq:introzonsec}
    \zeta_X(p):=\EE\kop\qwe0,\dd_pX^1\wedge\cdots\wedge\dd_p X^k\ewq\Big|X(p)=0\pok\rho_{X(p)}(0),
\end{equation}
where $\rho_{X(p)}:\R^k\to [0,+\infty]$ is the density of the random vector $X(p)$.
Every convex body $K$ has a well defined \emph{length} $\ell(K)$, that is, the first intrinsic volume (\cref{def:zlength}) of $K$, also called the first Lipschitz-Killing curvature \cite{AdlerTaylor}. Moreover, a zonoid $K$ always has a center of symmetry $c(K)$. For technical reasons we will have to consider the point $e(K):=2c(K)$, which we named \emph{nigiro}, see \cref{def:nigiro}.
Finally, we identify a set of desired condition on the random field $X$ under which we can apply a Kac-Rice formula. We call those the \zkrok conditions, see below after \cref{thm:1}.
The first main result of the paper is the following theorem.
\begin{mainthm}\label{thm:1}
Let $X\colon M \to \R^k$ be a \zkrok random field and let $Z:=X^{-1}(0)$. Then there is a continuous section of zonoids $\zeta_X$ as in \eqref{eq:introzonsec} such that:
\begin{equation}\label{eq:rhoalpha}
    \EE\left\{\vol_{(m-k)}(Z\cap A)\right\}=\int_A \delta_Z dM, \qquad \EE\left\{\int_{Z} \omega\right\}=\int_M e_Z\wedge \omega,
\end{equation}
where $\delta_Z(p)=\ell(\zeta_X(p))\in \R$ and $e_Z(p)=e(\zeta_X(p))\in \Lambda^kT^*_pM$ are a continuous function and a continuous $k$-form, respectively. We call $\zeta_X$ the \emph{zonoid section} of $X$.
\end{mainthm}

In the main body of the paper, \cref{thm:1} is divided into \cref{thm:Evol} and \cref{thm:Ecurrent}.

The description of the {\zkrok} hypotheses (\cref{def:zkrok}) is an important part of this work (see \cref{sec:zkrokhypotheses}) in that they are the conditions that are required to employ our version of the Kac-Rice formula (\cref{thm:Alphaca}), on which \cref{thm:1} is ultimately based.
Roughly speaking, a random field $X\colon M\to \R^k$ is \zkrok if (Compare with \cite[2.1]{KRStec}):
\begin{enumerate}
    \item $X$ is almost surely of class $\mC^1$.
    \item $0$ is a regular value of $X$, almost surely. This is to guarantee that $Z=X^{-1}(0)$ is almost surely a submanifold.
    \item The law of $X(p)$ on $\R^k$ is absolutely continuous and \dots
    \item \dots its density $\rho_{X(p)}(x)$ is continuous in both variables at $(p,0)$.
    \item The conditional expectation $\EE\{J_pX|X(p)=0\}$ makes sense and it is regular enough, where for every $f=(f^1,\ldots,f^k)\in C^1(M,\R^k),$ we write $J_pf:=\|d_pf^1\wedge\cdots\wedge d_pf^k\|$.
\end{enumerate}
 If $X$ is Gaussian, then it is very easy to check the \zkrok conditions (see \cref{prop:gauskrok1} and \cref{prop:Gausmthkrok}) and in this case the zonoids $\zeta_X(p)$ are ellipsoids. 
 
We can express the length and the nigiro of the zonoid section as follows.
\bega\label{eq:krmain}
\ell(\zeta_X(p))&=\EE\kop J_pX\big|X(p)=0\pok\rho_{X(p)}(0), \\
e(\zeta_X(p))&=\EE\kop d_pX^1\wedge \dots \wedge d_pX^k\big|X(p)=0\pok\rho_{X(p)}(0),
\eega
where $X=(X^1,\ldots,X^k)$ and $J_pX$ denotes the Jacobian determinant of $X$, that is, $J_pX=\|d_pX^1\wedge \dots \wedge d_pX^k\|$. From the first equation in \eqref{eq:krmain}, the reader that is familiar with Kac-Rice formulas, can recognize that the first identity in \eqref{eq:rhoalpha} is in fact a translation of the most common version of it (see \cite{AzaisWscheborbook}). On the contrary, the formula obtained by combining the second identities in \eqref{eq:rhoalpha} and \eqref{eq:krmain} is new.
\be\label{eq:newEcurrent}
\EE\left\{\int_{Z} \omega\right\}=\int_M\tyu\EE\kop d_pX^1\wedge \dots \wedge d_pX^k\big|X(p)=0\pok\rho_{X(p)}(0)\uyt\wedge \omega,
\ee
Although it is based on Kac-Rice formula, to the authors' knowledge such a general result for the expected current was not available in the literature. In particular, under our hypotheses, the resulting current is represented by a continuous differential form. Other works which study the expected current of a random submanifold are  \cite{Nicolaescu2016, NicSavale,DangRiv2018,LETENDRE2016}.
\begin{remark}
If $X(p) $ and $d_pX$ are stochastically independent, then the conditioning disappears: $\zeta_X(p)=\EE\kop\qwe0,\dd_pX^1\wedge\cdots\wedge\dd_p X^k\ewq \pok\rho_{X(p)}(0),$ see \cref{rem:indcond}.
\end{remark}

\subsubsection{The wedge and pull-back properties}

Given two independent random fields $X_1,X_2$, with zero sets $Z_i:=X_i^{-1}(0)$, $i=1,2$, one can study the intersection $Z_0:=Z_1\cap Z_2$ as the zero set of the random field $X_0:=(X_1,X_2)$. The idea behind this paper is to answer to the following questions:
\begin{question}\label{q:tomorrow}
    Suppose that you are given $X_1$ and you know that \emph{tomorrow} you will have to compute $\delta_{Z_1\cap Z_2}$ or $e_{Z_1\cap Z_2}$ for some yet unknown $X_2$. What can you do \emph{today} to start simplifying \emph{tomorrow}'s work?
\end{question}
In more formal terms, we want to identify some objects associated to $X_1$ and $X_2$ that are sufficient to determine the density $\delta_{Z_1\cap Z_2}$ and the form $e_{Z_1\cap Z_2}$ and a set of rules to compute them.

In the case of the expected current the answer is pretty simple since, by linearity, we have $e_{Z_1\cap Z_2}=e_{Z_1}\wedge e_{Z_2}$, so the answer to \cref{q:tomorrow} is that one needs to compute the form $e_{Z_1}$ in this case.  

In the volume case things are more subtle in that the couple $(\delta_{Z_1},\delta_{Z_2})$ is not a sufficient data to determine $\delta_{Z_1\cap Z_2}$. This is where the zonoid section really comes into play as an elegant answer to \cref{q:tomorrow}.

For example, if $S\subset M$ is a submanifold and the field $Y=X|_S$ is {\zkrok}, then $e_{Y}=e_X|_{S}$, but the density of expected volume $\delta_Y$ is not determined by $\delta_X$. However, 
the zonoid section of $Y$ is determined by that of $X$, via pull-back.
\begin{customthm}{thm:pullback}[Pull-back property]\label{thm:3}
Let $X\colon M \to\R^k$ be \zkrok. Let $S$ be a smooth manifold and let $\f\colon S\to M$ be a smooth map such that $\f\transv X^{-1}(0)$ almost surely. Then $X\circ \f\colon S\to \R^k$ is {\zkrok} and 
\be\label{eq:intropullback} 
\zeta_{X\circ \f}(q)=d_q\f^*\tyu\zeta_X\tyu\f(q)\uyt\uyt, \quad \forall q\in S.
\ee
\end{customthm}

Recently in \cite{ZA} a framework was developed by the first author together with Breiding, Bürgisser and Lerario to build multilinear maps on zonoids from multilinear maps on the underlying vector spaces, see \cref{thm:FTZC} or \cite[Theorem~4.1]{ZA} 
In particular, the \emph{wedge product} of two zonoids $\zeta_1\subset \Lambda^{k_1} T^*_pM$ and $\zeta_2\subset \Lambda^{k_2} T^*_pM$ is defined and lives in $\Lambda^{k_1+k_2}T^*_pM$.
\begin{customthm}{thm:wedge}[Wedge property]\label{thm:2}
Let $X_i\colon M\to \R^{k_i}$ be independent {\zkrok} random fields. Then $X_0:=(X_1,X_2)\colon M\to \R^{k_1+k_2}$ is {\zkrok} and
\be 
\zeta_{X_0}=\zeta_{X_1}\wedge \zeta_{X_2}.
\ee
\end{customthm}
In other words, an answer to \cref{q:tomorrow} above is to compute the zonoid section of $X_1$, so that \emph{tomorrow} it will be sufficient to apply \cref{thm:1} and \cref{thm:2} to get $\delta_{Z_1\cap Z_2}=\ell(\zeta_{X_1}\wedge \zeta_{X_2})$. The passage from $X$, a probability law on $\mC^1(M,\R^k)$, to $\zeta_{X}$ is a big reduction of data since the zonoid $\zeta_X(p)$ is defined pointwise (\cref{def:zonoidsec}) and depends only on the law of
\be 
(X(p),d_pX^1\wedge \dots\wedge d_pX^k)
\text{ random vector in }\R^k\times \Lambda^k T^*_pM,
\ee
hence the zonoid section does not remember the whole correlation structure of the field $X$. This is the same spirit as that of Kac-Rice formula. 

\begin{remark}
It is important that the {\zkrok} hypotheses are stable enough to allow the operations in both \cref{thm:2} and  \cref{thm:3}, while keeping \cref{thm:1} true.
\end{remark}

\subsubsection{Alexandrov-Fenchel and Brunn-Minkowsky}
The results just discussed create a bridge between random fields and the very rich theory of convex bodies. Such connection allows to draw on deep results such as the Alexandrov-Fenchel inequality (\cref{prop:AFi} and \cite[Theorem 7.3.1]{bible}) and the Brunn-Minkowsky inequality (\cref{prop:BMi} and \cite[p 372 (e)]{bible}) to obtain relations between different instances of $\delta_Z$. The former allows to deduce \cref{cor:KRAF} which, in the case $M$ is a surface, says the following. 
\begin{customthm}{cor:KRAF}[KRAF for surfaces]\label{thm:D}
Let $\dim M=2$ and let $Z_1,Z_2$ be random curves defined by independent {\zkrok} fields, then, for all $p\in M$,
\begin{equation}\label{eq:KRAFintro}
    \delta_{ Z_1\cap Z_2}(p)\geq \sqrt{\delta_{ Z_1\cap Z'_1}(p)\cdot \delta_{ Z_2\cap Z_2'}(p)},
\end{equation}
where $Z_i'$ is an independent copy of $Z_i$.
\end{customthm}
Similarly, from the Brunn-Minkowsky inequality we deduce \cref{cor:KRBM}.
\begin{customthm}{cor:KRBM}[KRBM for surfaces]\label{thm:E}
Let $\dim M=2$ and let $Z_1,Z_2$ be random curves defined by independent {\zkrok} fields. For $t\in [0,1]$, let $Z_t$ be the random curve such that $Z_t=Z_2$ with probability $t$ and $Z_t=Z_1$ otherwise. Then, for all $p\in M$,
\be\label{eq:KRBMintro} 
\delta_{Z_t\cap Z_t'}(p)\ge \delta_{Z_1\cap Z_1'}^{(1-t)}(p)\delta_{Z_2\cap Z_2'}^{t}(p)
\ee 
where $Z_i'$ is an independent copy of $Z_i$.
\end{customthm}
This result is based on the observation that $Z_t$ is the zero set of another field $X_t$ that, if \zkrok, has for zonoid section the \emph{Minkowsky sum} of the other two: $\zeta_{X_t}=(1-t)\zeta_{X_1}+t\zeta_{X_2}$, see \cref{prop:bern}.
\begin{remark}
The inequality \eqref{eq:KRBMintro} actually involves the same three terms as \eqref{eq:KRAFintro}. Indeed from the definition of $Z_t$ it is immediate to deduce that: \be 
\delta_{Z_t\cap Z_t'}=(1-t)^2\delta_{Z_1\cap Z_1'}+t^2\delta_{Z_2\cap Z_2'}+2t(1-t)\delta_{Z_1\cap Z_2}.
\ee
\end{remark}

In the full statements of \cref{thm:D} and \cref{thm:E} (see \cref{sub:AFBM}) there is no assumption on the dimension of $M$.

\subsubsection{Comment on the proof of \cref{thm:1}}
The main technical result that we need and that is the content of \cref{thm:Alphaca} is the following version of Kac-Rice formula expressing the expectation of the integral of some functional $\a\colon \mC^1(M,\R^k)\times M\to \R$ over the submanifold $Z=X^{-1}(0)$ defined by a random field $X\randin \mC^1(M,\R^k)$:
\be\label{eq:ourKR}
\EE\kop\int_{ Z}\alpha(X,p)dZ(p)\pok=\int_M
\EE\kop \a(X,p)J_pX \Big| X(p)=0\pok \rho_{X(p)}(0)dM(p).
\ee 
We don't consider this an original result, since this formula is essentially known as one of the many variations of Kac-Rice. Nevertheless, we remark that we couldn't find any reference in the literature for a statement equivalent to \cref{thm:Alphaca}, which is crucial for us since it shows the validity of \eqref{eq:ourKR} under the hypothesis that $X$ is a {\zkrok} random field, except for the case when $k=\dim M$, that is \cref{thm:alphaKR} and for which we refer to \cite{KRStec} (see also \cref{apx:comp}). 

We also remark that to obtain \cref{thm:Alphaca} we use an argument that is new in this context and which shows that the validity of Formula \eqref{eq:ourKR} just in the case $k=\dim M$, when $Z$ is discrete, implies its validity for all cases. For this we exploit the properties of a class of Gaussian random fields on a Riemannian manifold $(M,g)$, that we call \emph{normal}, defined as those for which $g$ is the associated metric in the sense of \cite{AdlerTaylor}, see \cref{sec:ATfield}.
This strategy reflects the philosophy of this paper in that it exploits the interplay between different instances of the Kac-Rice formula.

\subsection{Other results}\label{sub:intro:secondary}
\subsubsection{Density of intersection in terms of mixed volumes}
To a convex body $K\subset \R^d$, one can associate $d+1$ numbers $\Vint_0(K),\ldots,\Vint_d(K)$ called the \emph{intrinsic volumes} of $K$ (also called \emph{Lipschitz-Killing curvatures} in more general contexts \cite{AdlerTaylor}). They are the coefficients in Steiner's formula \cite{bible}:
$ 
\vol_d(K+tB_d)=\sum_{i=0}^d\Vint_{d-i}(K)\vol_{i}(tB_{i}),
$ where $B_i\subset \R^i$ is the unit ball.
The \emph{length} $\Vint_1(K)=\ell(K)$ is the one appearing in \cref{thm:1}. Then, the Euler characteristic $\Vint_0(K)=\chi(K)\in\{0,1\}$ only tells if $K$ is empty or not and $\Vint_d(K)=\vol_d(K)$ is the usual volume.

 The role of the intrinsic volumes in our picture is clarified by the wedge product of zonoids \cite{ZA}. In particular, if $K=\zeta$ is a zonoid, we have
$ 
i!\Vint_i(\zeta)=\ell(\zeta^{\wedge i})
$, see \cref{prop:kintvolpow}. Combining it with \cref{thm:1} and \cref{thm:2}, this yields \cref{cor:Evolvol}:
\be 
\EE\kop\vol_d(Z_1\cap \dots \cap Z_k)\pok=k!\int_M\Vint_k(\zeta_{X})dM,
\ee
whenever $Z_i$ are i.i.d. zero sets of a scalar {\zkrok} random field $X\colon M\to \R$. The notion of intrinsic volume for zonoids is related to that of \emph{mixed volume}. The mixed volume of $m$ convex bodies $K_1,\ldots,K_m\subset \R^m$, denoted $\MV(K_1,\dots,K_m)$, is defined as the coefficient of $t_1\cdots t_m$ in the polynomial $\vol_d(t_1K_1+\dots t_mK_m)$, see \cite[Theorem 5.1.7]{bible}. If $Z_1,\ldots,Z_m$ are random level sets of $m$ independent scalar {\zkrok} field $X_1,\ldots,X_m$ respectively, on a $m$ dimensional manifold $M$, then \cref{cor:Evolvol} states also that 
\be 
\EE\kop\#(Z_1\cap \dots \cap Z_m)\pok=m!\int_M\MV(\zeta_{X_1}, \dots, \zeta_{X_m})dM.
\ee
\subsubsection{What does the zonoid section know?}

The zonoid section can be separated into two parts as follows, see \cref{def:nigiro}.
\be\label{eq:separation}
\zeta_X(p)=\frac12 e(\zeta_X(p))+\underline{\zeta_X(p)}
\ee
where $\underline{\zeta_X(p)}$ has its center of symmetry at the origin.
The length, and thus the density of expected volume, depends only the centered zonoid, that is, on  $\underline{\zeta_X(p)}$.
In general, the centered zonoid is a sufficient data to compute the expectation of all quantities of the form $\int_ZF(T_pZ)dZ$. 
More precisely, given a measurable function $F:G(m-k,TM)\to \R$, we have
   \be\label{eq:sexy}
    \EE\kop\int_Z F\tyu T_pZ\uyt dZ(p)\pok=\int_{G(m-k,TM)}F\, d V_{\seg{\zeta_X}},
    \ee
where $V_{\seg{\zeta_X}}$ is a measure on $G(m-k,TM)$ associated to the centered zonoid section $\seg{\zeta_X}$ via the \emph{cosine transform}, see \cref{sec:zonandmeas}. The function $\seg{\zeta_X}\mapsto V_{\seg{\zeta_X}}$ is, in fact, injective

    We will discuss this in more details in \cref{sub:whatdoesitknow}. 
    In particular, we will show that the centered zonoid section $\seg{\zeta_X}$ depends only on the law of the random submanifold $Z=X^{-1}(0)$, see \cref{prop:onlydeponrandsubm}.

\subsubsection{The zonoid section as the expectation of a random varifold}
A $d$-Varifold in $M$ is a positive Borel measure on the total space of the Grassmann bundle 
\be G(d,TM)=\kop V\subset T_pM: p\in M, V\text{ is a linear subspace of dimension $d$}\pok.
\ee
We thus can think of a $d$-varifold $V$ as a linear continuous  functional $F\mapsto  V(F)$, defined for every bounded continuous function $F:G(d,TM)\to\R$ and such that $V(F)\le \sup|F|.$
Traditionally, varifold are introduced as a non-oriented variant of the concept o currents. Indeed, any non-necessarily-oriented $d$ dimensional compact submanifold $Z\subset M$ of a Riemannian manifold $M$ canonically defines a varifold $V_Z(F):=\int_M F(T_pZ) dM(p).$ 

On the other hand, a classical result in the theory of zonoids (see \cite{bible}) is that centered zonoids in a Euclidean space $V$ are in $1-1$ correspondence with even measures on the sphere $S(V).$ 
In our case, the zonoid $\zeta_X(p)$ of a \zkrok field $X:M\randto \R^k$, lives in $V=\Lambda^kT^*_pM$ and it is special in that the associated measure is supported on the space of simple vectors, which can be indentified with $G(k,T_p^*M)\cong G(d,T_pM),$ where we set $d=m-k.$ Because of this observation, a zonoid section $\zeta=\kop\zeta(p)\pok_{p\in M}$, is  uniquely associated to a section of measures $\{ \mu_{\zeta(p)} \}_{p\in M}$ and we can use this data to construct a $d$-varifold $V_{\zeta}$ via the formula (see \cref{eq:defimuX})
\be 
V_{\zeta}(F)=\int_M \int_{G(d,T_pM)}F(V)d\mu_{\zeta(p)}(V) dM(p).
\ee
We have the following.
\begin{customthm}{thm:title}[Expectation of a random varifold]\label{thm:titlethm}
Let $X\randin \mC^1(M,\R^k)$ be a \zkrok random field, and let $d=m-k$ be the dimension of the random submanifold $Z:=X^{-1}(0).$ Then
\be 
\EE V_Z=V_{\zeta_X}.
\ee
\end{customthm}
We will prove that (see \cref{lem:zonoidvarifold}), in the case in which $\zeta=\zeta_X$ is the zonoid section of a \zkrok field, one can recover the zonoid section $\zeta_X$ from the varifold $V_{\zeta_X}$ and viceversa. In this sense, \cref{thm:title} explains the title of the paper.

\subsubsection{Many representatives of the Euler class}

All the previous results extend naturally to random sections of vector bundles (\cref{thm:megathmvb}); if $\pi\colon E\to M$ is a smooth vector bundle of rank $k$ and $X\colon M\to E$ is a random section that is {\zkrok} in any local trivialization (in this case we say that it is \emph{locally} \zkrok, see \cref{de:loczkrok}) then the zonoid section is defined (\cref{def:zonoidsecvb}) as a function of the form:
\be 
M\ni p\mapsto \zeta_X(p)\subset \Lambda^k T^*_pM\otimes \det E_p,
\ee
where we recall that $\det E:=\Lambda^k E$ is a real line bundle, trivial if and only if $E$ is orientable. 
The reader who is familiar with algebraic topology will recognize a strong analogy between such extensions of \cref{thm:2} and \cref{thm:3} with the axiomatic properties of characteristic classes of vector bundles. Indeed, in the case in which both $M$ and $E$ are orientable the expected current $e(\zeta_X)=\EE\int_Z$, if smooth, is in fact a closed $k$-form representing the De Rham Euler class of $E$:
\be\label{eq:GaussBonnet} 
\qwe e(\zeta_X)\ewq=e(E)\in H_{DR}^k(M),
\ee
see \cref{thm:megathmvb}.(5). A more subtle version of this fact holds without any orientability assumption, see \cref{cor:nonorientEcurrent} and \cref{rem:nonorient}.
Equation \eqref{eq:GaussBonnet} can be regarded as a generalized Gauss-Bonnet-Chern theorem (see \cite{spivak,nicolaescu2020lectures}) in that on the left there is a \emph{local} object that depends on the structure of the random field, while on the right hand side we have a \emph{global} topological quantity depending only on the bundle. In other words, a random section specifies a way to distribute the Euler class of $E$ over the manifold $M$. For instance in the case when $k=m$ the Euler class becomes a number: the Euler characteristic $\chi(E)\in\mathbb{Z}$ and Equation \eqref{eq:GaussBonnet} reads
\be\label{eq:pointsGaussBonnet}
\int_M e(\zeta_X)=\chi(E).
\ee
The classical statement of Gauss-Bonnet-Chern Theorem for a vector bundle $E$ endowed with a metric $h$ and a connection $\nabla$ can be recovered from \cref{eq:GaussBonnet} by taking $X$ to be a suitable Gaussian random section. This has been proved, by direct computations, in \cite{Nicolaescu2016}. 

\subsubsection{Finsler Crofton formula}
In \cref{sec:Finsler} we give an interpretation of our results in the context of \emph{Finsler Geometry} \cite{introRiemannFinsler}. Given a scalar \zkrok random field $X\randin \mC^1(M)$ on $M$, the convex body $\zeta(p):=\seg{\zeta_X(p)}$, if full dimensional, defines a norm $F_p:=h_{\zeta(p)}:T_pM\to\R$, that is continuous with respect to $p\in M$. This norm is such that the convex body $\zeta(p)$ is the dual of the unit ball, see \cref{def:finsler}. Such an assignment is called a \emph{Finsler structure}\footnote{In general the norm of a Finsler structure is also assumed to have some $C^2$ regularity that we won't assume here.}. In our case the convex body $\zeta(p)$ always contains the origin and depends continuously on $p$ but may not be full dimensional , thus $h_{\zeta(p)}$ only defines a \emph{semi} norm. We will call a \emph{semi Finsler} structure the choice of a semi norm $F_p:T_pM\to\R$ that depends continuously on $p\in M$. Then we have that a scalar \zkrok random field $X\randin C^1(M,\R)$ defines a semi Finsler structure $F^X$, see \cref{def:finsler}.

Given a (semi) Finsler structure $F$ on $M$, the usual definition of the length of a curve as the integral of the norm of the velocity still makes sense, see \cref{eq:lengthFinsler}. Combining the pull-back property (\cref{thm:3}) with \cref{thm:1} we are able to produce a Crofton formula, that is, to relate the length of a curve with the expectation of the number of points of intersection with an hypersurface. More precisely, if $X:M\to\R$ is \zkrok, $Z=X^{-1}(0)$ and $\gamma$ is a $C^1$ curve in $M$ almost surely transversal to $Z$, then we have, see \cref{propcroftoniftrs}:
\begin{equation}
    \EE\#(\gamma\cap Z)=2\,\ell^{F^X}(\gamma).
\end{equation}
Unlike for the length, there are several notions of the volume  of a $k$ dimensional submanifold $S\subset M$ in Finsler geometry, see \cite{volumesFinsler}. One of the most common is the \emph{Holmes-Thompson volume}, which is still defined in the semi Finsler case and we denote it as $\vol^F_k(S)$. It turns out that in the case in which the semi Finsler structure $F^X$ is defined by a scalar \zkrok field $X$ we can also prove a Crofton formula for the Holmes-Thompson volume (\cref{prop:FinsCroftonVol}):
\begin{equation}
    \EE\kop\#(S\cap Z_1\cap \dots \cap Z_k)\pok=k! \uball_k \vol_k^{F^{X}}(S),
\end{equation}
where $Z_i$ are independent copies of $Z=X^{-1}(0)$ and $S\subset M$ is any $k$ dimensional submanifold almost surely transversal to $Z$. Constructions of Finsler structures that admit a Crofton formula are known for random hyperplanes in projective space, see \cite{PaivaGelfCrof,BernigCrofton,SchneiderCrofton}. Moreover, a more general result very similar to \cref{propcroftoniftrs} can be found in \cite[Theorem~A]{paivageod} although the {\zkrok} hypotheses are significantly less restrictive and the construction of the metric $F^X$ is explicit (see \cref{eqdefFX}).
\subsubsection{Examples}
With \cref{thm:abundance} we show that any random field ${Y\randin}\mC^\infty(M,\R^k)$ can be \emph{approximated} by a {\zkrok} random field, with the only condition being that $\EE\{J_pY\}$ should be finite and continuous with respect to $p\in M$. Such operation is obtained by means of what can be described as a convolution with a constant field, that is, a random vector $\lambda \randin \R^k$, provided that the latter has a continuous, bounded and non vanishing density. In this case,
\be\label{eq:perturb} 
X:= Y-\lambda \text{ is \zkrok.}
\ee
This result, while demonstrating the abundance of {\zkrok} fields, suggests that they could be used to study more wild random fields via perturbative techniques. The study of the behavior of the results obtained in this paper when $\lambda\to 0$ in \eqref{eq:perturb} will be object of future work by the authors. 

A particular case of \eqref{eq:perturb} is when $Y=f$ is a deterministic smooth function, so that $Z=Y^{-1}(\lambda)$ is a random level set of $f$. We discuss this example in \cref{sub:abu}

In \cref{sub:finifields} we discuss the case when the law of the random field $X$ is supported on a finite dimensional linear subspace $\mathcal{F}\subset \mC^\infty(M,\R^k)$ and has a density $\rho_X\colon \mathcal{F}\to [0,+\infty)$. 
This is the most typical situation in the existing literature (see \cref{subsub:exmind}). It includes especially the case of random eigenfunctions of elliptic operators, Riemannian random waves and random band limited functions, not necessarily Gaussian. It also naturally applies to random polynomials.

We show (see \cref{prop:finitezkrokcond} and \cref{prop:zonfinitetype}) that such $X$ is always {\zkrok} as long as $\mathcal{F}$ is ample, meaning that for any $p\in M$ the set $\{f(p)\colon f\in\mathcal{F}\}$ spans the whole $\R^k$ (i.e., $\mathcal{F}$ generates $\mC^\infty(M,\R^k)$ as a $\mC^\infty(M)$-module), and if the density satisfies the integrability condition $\rho_X(f)=O(\|f\|^{-\dim \mathcal{F}})$ as $\f\to\infty$.  

\subsection{Structure of the paper}
\cref{sec:zonoids} contains a brief survey on the theory of convex bodies and zonoids, with emphasis on the formulas and the notations that are needed in the following sections. This section is essentially based on the monograph \cite{bible} and on the recent paper \cite{ZA}.
In \cref{sec:zkrokhypotheses} we define the {\zkrok} hypotheses in details, discussing alternative formulations and special cases. We give the definition of the zonoid section in \cref{sec:zonoidsection} and the proof of \cref{thm:2} and \cref{thm:3}. In \ref{sec:proofAlpha} we establish the Kac-Rice formula (\cref{thm:Alphaca}) that we need to prove \cref{thm:1}. The latter is divided into two statements, \cref{thm:Evol} and \cref{thm:Ecurrent}, both proved in \cref{sec:mainres}. In \cref{sub:AFBM} we report the full statements of \cref{cor:KRAF} and \cref{cor:KRBM}, which are obtained as corollaries of \cref{thm:Evol}. 
The subsequent sections cover the material discussed in \cref{sub:intro:secondary} above, in particular, the proof of \cref{thm:title} is given in \cref{sub:whatdoesitknow}.
\subsection{Aknowledgements}
This  work is partially supported by the grant TROPICOUNT of Région Pays de la Loire, and the ANR project ENUMGEOM NR-18-CE40-0009-02.
\section{Notations}\label{sec:notations}
\newcommand{\Prob}{(\Omega, \mathfrak{S},\PP)}

Here below, a list of the main notations used in this paper, for the reader's convenience.
\begin{enumerate}[$\bullet$]
\item We say that $X$ is a \emph{random element} (see \cite{Billingsley}) of the topological space $T$ if $X$ is a measurable map $X\colon \Omega\to T$, defined on some probability space $\Prob$. In this case we will write
    \be 
    X\randin T
    \ee 
    and we denote by $[X]=\PP X^{-1}$ the Borel probability measure on $T$ induced by pushforward. We will use the following notation:
\be 
\PP\{X\in U\}:=
\PP X^{-1}(U)
\ee 
to denote the probability that $X\in U$, for some measurable subset $U\subset T$, and
\be 
\EE\{f(X)\}:=\int_{T}f(t)d[X](t),
\ee
to denote the integral of a measurable vector-valued function $f\colon T\to \R^k$.
We call $X$ a \emph{random variable}, \emph{random vector} or \emph{random map} if $T$ is the real line, a vector space or a space of continuous functions $\mC(M,N)$, 
respectively.
\item Given topological spaces $M$ and $N$, we write
\be 
X\colon M \randto N,
\ee 
to say that $X$ is a random map, i.e., a random element of $\mC(M,N)$. The symbol winks at the fact that $X$ can be seen as a function $X\colon M\times \Omega\to N$.
\item The sentence: ``$X$ has the property $\mathcal{P}$ almost surely'' (abbreviated ``a.s.'') means that the set  $S=\{t\in T| t \text{ has the property }\mathcal{P}\}$ contains a Borel set of $[X]$-measure $1$. It follows, in particular, that the set $S$ is $[X]$-measurable, i.e. it belongs to the $\sigma$-algebra obtained from the completion of the measure space $(T,\mathcal{B}(T),[X])$.
 \item We write $\#(S)$ for the cardinality of the set $S$.
\item We use the symbol $A\transv B$ to say that objects $A$ and $B$ are in transverse position, in the usual sense of differential topology (as in \cite{Hirsch}).
    \item The space of $\mC^r$ functions between two manifolds $M$ and $N$ is denoted by $\mC^r(M,N)$. We just write $\mC^r(M)$ in the case $N=\R$.
    If $E\to M$ is a vector bundle, we denote the space of its $\mC^r$ sections by $\mC^r(M|E)$. In both cases, we consider it to be a topological space endowed with the weak Whitney's topology (see \cite{Hirsch}).
    
\item We use $\Gamma(\mathcal{Z})$ for the space of continuous sections of a continuous fiber bundle $\mathcal{Z}\to M$.
\item Given a topological space $T$, we denote by $\Me(T)$ the topological vector space of finite signed Borel measures, endowed with the weak-$*$ topology induced by the inclusion $\Me(T)\subset \mC_b(T)^*$. We write $\Me^+(T)$ for the subset of positive finite measures and $\mathscr{P}(T)$ for that of probability measures, both considered with the subspace topology, if not otherwise specified.
    \item If $V$ is a vector space and $x,y\in V$, we write $[x,y]:=\set{(1-t)x+ty}{t\in[0,1]}$. Moreover, we abbreviate
    \begin{equation}\label{eq:notseg}
        \seg{x}:=\frac {1}{2} [-x,x].
    \end{equation}
    \item We use $\uball_k$ for the $k$ dimensional volume of the unit ball in $\R^{k}$ and $\usph_{k}=\frac{2^{k+1}\pi^{k}}{k!b_k}$ for the $k$ dimensional volume of the unit sphere in $\R^{k+1}$. 
\end{enumerate}

\section{Zonoids}\label{sec:zonoids}

Throughout this section $(V, \langle \cdot, \cdot \rangle)$ is a (real) Euclidean space of dimension $m$, $V^*$ its dual and $S(V)$ is the unit sphere of $V$.

\subsection{Basic definitions}
 A subset $K$ of $V$ is \emph{convex} if for every $x,y\in K$, the segment $[x,y]=\set{(1-t)x+ty}{t\in[0,1]}$ is contained in $K$. A \emph{convex body} is a non empty compact convex subset. If $K\subset V$ is a convex body, its \emph{support function} is the positively homogeneous function $h_K: V^*\to \R$ given by
\begin{equation}
  h_K(u):=\sup\set{\langle u, x\rangle}{x\in K}.
\end{equation}
A function $h:V^*\to \R$ is the support function of a convex body in $V$ if and only if it is \emph{sublinear}, that is if $h(\lambda u)=\lambda h(u)$ for all $u\in V^*, \lambda \geq 0$ and $h(u+v)\leq h(u)+h(v)$ for all $u,v\in V^*$; see \cite[Theorem~1.7.1]{bible}.

The norm on $V^*$ induces a complete distance on the space of convex bodies of $V$ called the \emph{Hausdorff distance} \cite{hausdorff1914grundzüge}. This is equivalent to the supremum distance of the support functions, given for all $K_1,K_2\subset V$ convex bodies by (see\cite[ Lemma 1.8.14]{bible}):
\begin{equation}\label{eq:Hausdist}
  \dd(K_1,K_2)=\sup\set{|h_{K_1}(u)-h_{K_2}(u)|}{\|u\|=1}.
\end{equation}
The \emph{Minkowski sum} of two convex bodies $K_1,K_2\subset V$ is the convex body defined as:
\begin{equation}
  K_1+K_2:=\set{x_1+x_2}{x_1\in K_1, \, x_2\in K_2}.
\end{equation}
Finally we define for every $\lambda\in \R$ and convex body $K$, the convex body $\lambda K:=\set{\lambda x}{x\in K}$.

The support function satisfy some useful properties that we summarize in the next proposition. Those are direct consequences of the definition and for this reason we omit the proof.
\begin{proposition}\label{prop:suppfctprop}
  Let $K, L$ be convex bodies in a vector space $V$ and let $h_K$, respectively $h_L$ be their support functions. We have the following.
  \begin{enumerate}
      \item For all $t,s\geq0$ we have $h_{tK+sL}=th_K+sh_L$.
     \item \label{itm:linh} If $W$ is a vector space and $T:V\to W$ is a linear map then $h_{T(K)}=h_K\circ T^t$.
  \end{enumerate}
\end{proposition}

We are interested in a particular class of convex bodies.

\begin{definition}
  A \emph{zonotope} is a finite Minkowski sum of segments. A \emph{zonoid} is a limit, in the Hausdorff distance, of zonotopes.
\end{definition}
 Segments are always centrally symmetric and we can write $[x,y]=\seg{x-y}+\tfrac{1}{2}\{x+y\}$ where we recall the notation defined in~\cref{eq:notseg}. It follows that zonotopes, and thus zonoids are centrally symmetric. Moreover $K$ is a zonotope if and only if there exist $x_1,\ldots,x_N,e\in V$ such that $K=\seg{x_1}+\cdots+\seg{x_N}+\tfrac{1}{2}\{e\}$. This implies that for every zonoid $K$ there is a zonoid $\seg{K}$ with $(-1)\seg{K}=\seg{K}$ and a vector $c$ such that
 \begin{equation}
   K=\seg{K}+\frac{1}{2}\{e\}.
 \end{equation}
 
 \begin{definition}\label{def:nigiro}
   The point $e$ will be called the \emph{nigiro}\footnote{The nigiro $e(K)$ is symmetric to the \emph{origin} with respect to the center of $K$. In other words, as a vector, it is twice the center of $K$.} of $K$ and denoted $e(K)$. Moreover, for every zonoid $K$, we write $\seg{K}$ for the unique zonoid such that $K=\seg{K}+\tfrac{1}{2}\{e(K)\}.$ 
 \end{definition}

 We write $\ZZ(V)$ for the space of zonoids of $V$ and $\ZZo(V)$ for the space of \emph{centered} zonoids, i.e. $\ZZo(V):=\set{K\in \ZZ(V)}{(-1)K=K}$. By the discussion above we have
 \begin{equation}
   \ZZ(V)=\ZZo(V)\oplus V
 \end{equation}
 In the sense of the monoid structure given by the Minkowski sum. Elements of $\ZZo(V)$ are called \emph{centered zonoids}.

\subsection{Zonoids and random vectors.}\label{sec:zonandrv}
If $\Lambda$ is a random zonoid in $V$, that is a map from some probability space to $\ZZ(V)$, such that $\EE |d(0,\Lambda)|<\infty$ then we define the \emph{expected zonoid} $\EE\Lambda$ to be the convex body with support function given for all $u\in V^*$ by
\begin{equation}
  h_{\EE\Lambda}(u):=\EE\left\{ h_\Lambda(u)\right\}.
\end{equation}
It follows from a strong law of large number for compact sets from \cite{ArsVit} that if $\Lambda_1,\ldots,\Lambda_n$ are i.i.d. copies of $\Lambda$, then the random zonoid $\tfrac{1}{n}\left(\Lambda_1+\cdots+\Lambda_n\right)$ converges almost surely as $n\to\infty$ to $\EE\Lambda$. In particular the expected zonoid $\EE\Lambda$ is indeed a zonoid.

We will, in the following, consider mostly two examples. Let $X\randin V$ be a random vector such that $\EE\|X\|<\infty$. We say that $X$ is \emph{integrable} and we consider $\EE[0,X]$ and $\EE\seg{X}$. Their support function is given for all $u\in V^*$ by
\begin{align}\label{eq:suppofexpseg}
  h_{\EE[0,X]}(u)&=\EE\max\{0,\langle u, X\rangle\};& h_{\EE\seg{X}}(u)&=\frac{1}{2}\EE|\langle u, X\rangle|.
\end{align}
Next, we show that they are translate of one another.
\begin{lemma}\label{lem:centnotcent}
  Let $X\randin V$ be integrable. We have
  \begin{equation}
    \EE[0,X]=\EE\seg{X}+\frac{1}{2}\left\{\EE X\right\}.
  \end{equation}
 With the notation introduced above, this means that $e(\EE[0,X])=\EE X$. In particular $\EE[0,X]=\EE\seg{X}$ if and only if $\EE X=0$.
\end{lemma}
\begin{proof}
  It is enough to see that for every $t\in \R$ we have $\max\{0,t\}=\tfrac{1}{2}\left(|t|+t\right)$.Then use the expressions in~\cref{eq:suppofexpseg} and the fact that $h_{\{c\}}=\langle\cdot,c\rangle$.
\end{proof}
These constructions behave well under linear mappings.

\begin{lemma}\label{lem:lineartransfofVitale}
  Let $X\randin V$ be integrable, let $W$ be a finite dimensional Euclidean space and let $T:V\to W$ be a linear map. Then $T(X)\randin W$ is integrable and we have
  \begin{align}
      \EE[0,T(X)]=T\EE[0,X] && \EE\seg{T(X)}=T\EE\seg{X}
  \end{align}
\end{lemma}
\begin{proof}
By~\cref{eq:suppofexpseg} we have $ h_{\EE[0,T(X)]}(u)=\EE\max\{0,\langle u, T(X)\rangle\}=h_{\EE[0,X]}(T^t(u))$. By \cref{prop:suppfctprop}-\ref{itm:linh} this is the support function of $T\EE[0,X].$ The other case is done similarly.
\end{proof}

\begin{example}
  Let $x_1,\ldots,x_N\in\R^m$ and let $X\randin \R^m$ be the random vector that is equal to $Nx_i$ with probability $1/N$ for $i=1,\ldots,N$. Then computing the expression in~\cref{eq:suppofexpseg}, we find,
  \begin{align}
    \EE[0,X]=\sum_{i=1}^N[0,x_i]; && \EE\seg{X}=\sum_{i=1}^N\seg{x_i}.
  \end{align}
\end{example}

\begin{example}\label{eg:gausszonisball}
  Let $\xi\randin\R^m$ be a standard Gaussian vector and let $B^m:=B(\R^m)$. Then we have
  \begin{equation}
    \EE\seg{\xi}=\frac{1}{\sqrt{2\pi}}B^m.
  \end{equation}
  Indeed, since $\xi$ is $O(m)$--invariant, by \cref{lem:lineartransfofVitale}, $\EE\seg{\xi}$ must also be $O(m)$--invariant and thus is a ball. To compute its radius, it is enough to compute the support function at $e_1$, the first vector of the standard basis of $\R^m$. Since $\langle \xi, e_1\rangle\randin\R$ is a standard Gaussian variable, we obtain
  \begin{equation}
    h_{\EE\seg{\xi}}(e_1)=\frac{1}{2}\EE|\langle \xi, e_1\rangle|=\frac{1}{2}\sqrt{\frac{2}{\pi}}=\frac{1}{\sqrt{2\pi}}.
  \end{equation}
\end{example}

Vitale in \cite[Theorem~3.1]{Vitale} shows that every zonoid can be obtained via the above construction, i.e. for every $K\in\ZZ(V)$ there is an integrable $X\randin V$ and a vector $e\in V$ such that $K=\EE\seg{X}+\tfrac{1}{2}\{e\}$. However, the integrable random vector $X$ defining the zonoid $K:=\EE\seg{X}$ is not unique. This defines an equivalence relation on the integrable random vectors of a vector space known as the \emph{zonoid equivalence}, see \cite{MSSzoneq}. The following is \cite[Corollary~3]{MSSzoneq}.
\begin{proposition}\label{prop:zoneq}
  Let $X,Y\randin V$ be integrable. Then $\EE\seg{X}=\EE\seg{Y}$ if and only if for every one-homogeneous even measurable function $f: V\to \R_+$, we have:
  \begin{equation}
    \EE \left[f(X)\right]=\EE \left[f(Y)\right].
  \end{equation}
\end{proposition}

This shows that the following is well defined.
\begin{definition}\label{def:zlength}
  Let $X\in V$ be an integrable random vector and let $K:=\EE\seg{X}$. Then the \emph{length} of $K$ is defined to be
  \begin{equation}
    \ell(K):=\EE\|X\|.
  \end{equation}
\end{definition}
This functional is actually something very well known, see \cite[Theorem 5.2]{ZA}.
\begin{lemma}
  The length of a zonoid is equal to its first intrinsic volume (see \cref{eq:defvint} below).
\end{lemma}
Despite this result, we will continue to use the name \emph{length} and the notation $\ell$ to emphasize that we are thinking of \cref{def:zlength}. Since the first intrinsic volume is Minkowski linear and vanishes on zero dimensional bodies we also have, by \cref{lem:centnotcent},
\begin{equation}
  \ell(\EE[0,X])=\EE\|X\|.
\end{equation}

Finally, there is a simple trick to express the Minkowski sum of two zonoids in terms of random vectors. The proof is straightforward and thus omitted.

\begin{lemma}[Bernoulli trick] Let $X_0,X_1\randin\R^m$ be integrable and let $\epsilon\randin\{0,1\}$ be a Bernoulli random variable of parameter $t\in[0,1]$ independent of $X_0$ and $X_1$, that is $\epsilon=0$ with probability $t$ and $\epsilon=1$ with probability $1-t$. Let $X_t:=\epsilon X_0 + (1-\epsilon)X_1$. Then we have 
\begin{align}
    \EE[0,X_t]=(1-t)\EE[0,X_0]+t\EE[0,X_1]; && \EE\seg{X_t}=(1-t)\EE\seg{X_0}+t\EE\seg{X_1}.
\end{align}
  
\end{lemma}

\subsection{Zonoids and measures: the classical viewpoint}\label{sec:zonandmeas}
It is most common to approach centered zonoids with even measures on the sphere. We recall here this point of view and describe how this approach relates to Vitale's construction. The space of even signed measures on the unit sphere $S(V)$ is denoted by $\Mev(S(V))$ and the cone of non negative even measures by $\Mev^+(S(V))$.

It is a classical result (see~\cite[Theorem~3.5.3]{bible}) that for every centered zonoid $K\in\ZZo(V)$ there is a unique $\mu_K\in\Mev^+(S(V))$ such that
 \begin{equation}\label{eq:costrans}
   h_K(u)=\frac{1}{2}\int_{S(V)}|\langle u, x \rangle| \ \dd\mu_K(x).
 \end{equation}
The function $h_K$ is also called the \emph{cosine transform} of $\mu_K$. We also denote by $\mu_K$ the measure on $S(V^*)$ defined by \eqref{eq:costrans} with the scalar product replaced by the duality pairing.
If a centered zonoid is given by a random vector, it is possible to retrieve the corresponding measure on the sphere.
\begin{proposition}\label{prop:fromrvtomeas}
  Let $X\randin V$ be integrable and let $K:=\EE\seg{X}$. Then $\mu_K$ is the measure such that for every continuous function $f:S(V)\to\R$ we have
  \begin{equation}\label{eq:forMUla}
    \int_{S(V)}f \dd\mu_K:=\EE\left\{\|X\|f\left(\frac{X}{\|X\|}\right)\one_{X\neq 0}\right\}
  \end{equation}
\end{proposition}
\begin{proof}
  The function $x\mapsto\|x\|f\left(\frac{x}{\|x\|}\right)\one_{x\neq 0}$ is a one homogeneous continuous function on $V$. Thus by \cref{prop:zoneq} the term on the right only depends on $K$.
  To see that it satisfies~\cref{eq:costrans} apply it to $f=|\langle u, \cdot\rangle|$ for any $u\in V^*$.
\end{proof}
In particular, note that we have $\mu_K(S(V))=\ell(K)$. More generally, if $f:V\to\R_+$ is measurable and one homogeneous, we get
\begin{equation}\label{eq:intwrtgenmeas}
    \EE f(X)=\int_{S(V)}f\, d\mu_K
\end{equation}
where $X\randin V$ is integrable and $K:=\EE\seg{X}$.

\subsection{Zonoid calculus}\label{sub:zonoidcalc} In the recent paper \cite{ZA} the first author together with P. Breiding P. Bürgisser and A. Lerario proved that multilinear maps between vector spaces give rise to multilinear maps on the corresponding spaces of centered zonoids. The following is \cite[Theorem~4.1]{ZA}.

\begin{proposition}\label{thm:FTZC}
  Let $M:V_1\times\cdots\times V_k\to W$ be a multilinear map between finite dimensional vector spaces. There is a unique Minkowski multilinear continuous map
  \begin{equation}
    \widehat{M}:\ZZo(V_1)\times\cdots\times\ZZo(V_k)\to \ZZo(V)
  \end{equation}
  such that for all $v_1\in V_1,\ldots,v_k\in V_k$ we have
  \begin{equation}
    \widehat{M}\left(\seg{v_1},\ldots,\seg{v_k}\right)=\seg{M(v_1,\ldots,v_k)}.
  \end{equation}
\end{proposition}
We extend the map $\widehat{M}$ to general zonoids by setting for all $K_1\in\ZZo(V_1)$, $\ldots$, $K_k\in\ZZo(V_k)$ and every $c_1\in V_1,\cdots, c_k\in V_k$:
\begin{equation}\label{eq:Mnotcent}
    \widehat{M}\left(K_1+\tfrac{1}{2}\{c_1\},\ldots, K_k+\tfrac{1}{2}\{c_k\}\right):=\widehat{M}\left(K_1,\ldots, K_k\right)+\frac{1}{2}\left\{M(c_1,\ldots,c_k)\right\}.
  \end{equation}

One can check that this map is still Minkowski multilinear. Moreover, it behaves well under the Vitale construction.

\begin{proposition}\label{prop:multandVit}
  Let $M:V_1\times\cdots\times V_k\to W$ be a multilinear map between finite dimensional vector spaces and let $X_1\randin V_1,\ldots, X_k\randin V_k$ be integrable and independents. We have
  \begin{align}
    \widehat{M}\left(\EE\seg{X_1},\ldots,\EE\seg{X_k}\right)&=\seg{\EE M(X_1,\ldots,X_k)};& \widehat{M}\left(\EE[0,X_1],\ldots,\EE[0,X_k]\right)&=\EE[0,M(X_1,\ldots,X_k)].
  \end{align}
\end{proposition}
\begin{proof}
  The first statement about centered zonoids is~\cite[Corollary~4.3]{ZA}. The second one follows from it, \cref{lem:centnotcent} and~\cref{eq:Mnotcent}.
\end{proof}

Consider the exterior powers $\Lambda^k V$, $0\leq k \leq m$, where we recall that $m=\dim V$. There is a collection of bilinear maps $wedge_{k,l}:\Lambda^k V\times \Lambda^l V\to \Lambda^{k+l} V$
given for all $w\in \Lambda^k V$, $w'\in\Lambda^l V$ by $wedge_{k,l}(w,w'):=w\wedge w'$.
We consider the bilinear map induced on zonoids and if $A\in\ZZ(\Lambda^k V),A'\in\ZZ(\Lambda^l V)$ we write
\begin{equation}
  A\wedge A':=\widehat{wedge_{k,l}}(A,A').
\end{equation}
We will call this operation the \emph{wedge product of zonoids}. Using \cref{prop:multandVit} we have for $X$ and $Y$ \emph{independent} integrable random vectors:
\begin{align}\label{eq:wedgeandVit}
  \EE\seg{X}\wedge\EE\seg{Y}&=\EE\seg{X\wedge Y};& \EE[0,X]\wedge \EE[0,Y]=\EE[0,X\wedge Y].
\end{align}
\begin{remark}
  Note that the wedge product on centered zonoids is commutative, this follows from~\cref{eq:wedgeandVit} and the fact that $\seg{x}=\seg{-x}$.
\end{remark}
Finally, in the notation introduced in \cref{def:nigiro}, and using \cref{eq:Mnotcent}, we get that for every zonoids $K\in\ZZ(\Lambda^kV),L\in\ZZ(\Lambda^lV),$ we have
\begin{equation}
    \seg{K\wedge L}=\seg{K}\wedge\seg{L}\in \ZZo(\Lambda^{k+l}V)\quad\text{ and }\quad e(K\wedge L)=e(K)\wedge e(L)\in\Lambda^{k+l}V .
\end{equation}

\subsection{Mixed volume and inequalities}\label{sec:MV}
A fundamental result by Minkowski \cite[Theorem~5.1.7]{bible} states that, given convex bodies $K_1,\ldots,K_m\subset \R^m$, the function $(t_1,\ldots,t_m)\mapsto \vol_m(t_1K_1+\cdots+t_mK_m)$ is a polynomial in $t_1,\ldots t_m\geq0$. The coefficient of $t_1\cdots t_m$ is called the \emph{mixed volume} of $K_1,\ldots, K_m$ and will be denoted here by $\MV(K_1,\ldots, K_m).$ It relates to the wedge product of zonoids as follows.
\begin{proposition}[{\cite[Theorem~5.1]{ZA}}]
\label{prop:MVandwedge}
  Let $K_1,\ldots,K_m\in\ZZ(\R^m)$. We have the following.
  \begin{equation}
    \frac{1}{m!}\ell(K_1\wedge\cdots\wedge K_m )=\MV(K_1,\ldots,K_m).
  \end{equation}
\end{proposition}

From Minkowski's result, one can also build the \emph{intrinsic volumes} of a convex body $K\subset \R^m$ which are the coefficient (suitably normalized) of the Steiner polynomial $t\mapsto \vol_m(K+tB_m)$ where $B_m\subset \R^m$ is the unit ball. In our context we define the $k$--th intrinsic volume to be 
\begin{equation}\label{eq:defvint}
    \Vint_k(K):=\frac{\binom{m}{k}}{\uball_{m-k}}\MV(K[k],B_m[m-k])
\end{equation}
where $K[k]$ denotes the convex body $K$ repeated $k$ times in the argument.

From the previous Lemma one can deduce the following, which is \cite[Theorem~5.2]{ZA} and will be used later in the proof of \cref{cor:Evolvol}.

\begin{proposition}\label{prop:kintvolpow}
  Let $K\in\ZZ(\R^m)$. We have the following.
  \begin{equation}
    \frac{1}{k!}\ell(K^{\wedge k} )=\Vint_k(K)
  \end{equation}
  Moreover for all $k>\dim(K)$, $K^{\wedge k} =0$.
\end{proposition}

Moreover the support function on simple vectors takes the following form which will be used in \cref{lem:HTandwedge} to link zonoid calculus to the notion of \emph{Holmes-Thompson volume}.

\begin{lemma}\label{lem:supofsimplek}
  Let $K\in\ZZo(\R^m)$ be a centered zonoid and let $u=u_1\wedge\cdots \wedge u_k\in\Lambda^k \R^m$. We have
  \begin{equation}
      h_{K^{\wedge k}}(u_1\wedge\cdots\wedge u_k)=\frac{\|u_1\wedge\cdots\wedge u_k\Vert}{2}k!\vol_k(\pi_u(K)) 
  \end{equation}
  where $\pi_u:\R^m\to \Span(u_1,\ldots,u_k)$ denotes the orthogonal projection.
\end{lemma}

\begin{proof}
  Let $X\randin\R^m$ be such that $K=\EE\seg{X}$ and let $X_1,\ldots,X_k$ be iid copies of $X$. Then we have 
  \begin{align}
      h_{K^{\wedge k}}(u)&=\frac{1}{2}\EE|\langle X_1\wedge \cdots \wedge X_k , u_1\wedge \cdots\wedge u_k\rangle|   \\
      &=\frac{\Vert u_1\wedge \cdots \wedge u_k\Vert}{2}\EE\Vert \pi_u(X_1)\wedge \cdots \wedge \pi_u(X_k)\Vert \\
      &=\frac{\Vert u_1\wedge \cdots \wedge u_k\Vert}{2}\ell(\pi_u(K)^{\wedge k}).
  \end{align}
  Finally, by \cref{prop:MVandwedge}, we have $\ell(\pi_u(K)^{\wedge k})=k!\vol_k(\pi_u(K))$ which concludes the proof.
\end{proof}

\subsubsection{Alexandrov-Fenchel and Brunn-Minkowsky inequalities}
One of the most important inequality of convex geometry (if not the most important) involves the mixed volume and is known as the \emph{Alexandrov--Fenchel inequality} (AF), see~\cite[Theorem~7.3.1]{bible}.
\begin{proposition}[AF]\label{prop:AFi}
  Let $K_3,\ldots,K_m\subset\R^m$ be convex bodies and let us denote by $\mathfrak{K}$, the tuple $(K_3,\ldots,K_m)$. For all convex bodies $K,L\subset \R^m$ we have 
  \begin{equation}
      \MV(K,L,\mathfrak{K})\geq \sqrt{\MV(K,K,\mathfrak{K})\MV(L,L,\mathfrak{K})}.
  \end{equation}
\end{proposition}

Another inequality bounds from below the volume of the Minkowski sum of two convex bodies and is known as the \emph{Brunn--Minkowski inequality} (BM). It has many equivalent form and we chose to present here the multiplicative one, see~\cite[p.372~(e)]{bible}.

\begin{proposition}[BM]\label{prop:BMi}
  Let $K_0,K_1\subset \R^m$ be convex bodies. For all $t\in[0,1]$, we have
  \begin{equation}
      \vol_m((1-t)K_0+tK_1)\geq \vol_m(K_0)^{1-t}\vol_m(K_1)^t.
  \end{equation}
\end{proposition}

\subsection{Grassmannian zonoids}\label{sec:grasszon}
The zonoids that will appear in the construction of the zonoid section below (see \cref{def:zonoidsec}) belong to a particular subset of $\ZZ(\Lambda^k V)$. Recall that if $V$ is Euclidean then $\Lambda^k V$ inherits an Euclidean structure given for all $v_1\wedge\cdots\wedge v_k, w_1\wedge\cdots\wedge w_k\in\Lambda^k V$ by
\begin{equation}
    \langle v_1\wedge\cdots\wedge v_k, w_1\wedge\cdots\wedge w_k\rangle:=\det\left(\langle v_i, w_j\rangle\right)_{1\leq i,j\leq k}.
\end{equation}
Vectors of the form $v_1\wedge\cdots\wedge v_k\in\Lambda^k V$ are said to be \emph{simple}.

We write $G(k,V)$ for the Grassmannian of $k$--dimensional subspaces of $V$. Recall that the Grassmannian embeds in the projective space of $\Lambda^kV$ via the Plücker embedding that sends
$E\in G(k,V)$ to $[e_1\wedge\cdots\wedge e_k]\in\PP(\Lambda^kV)$ where $e_1,\ldots,e_k$ is a basis of $E.$ In particular the set of simple vectors in $\Lambda^kV$ can be viewed as the cone over the Grassmannian and a measure on $G(k,V)$ can be identified with an even measure on $S(V)$ supported on the simple vectors.

For every $E\in G(k,V)$ we define the segment
\begin{equation}
    \seg{E}:=\seg{e_1\wedge\cdots\wedge e_k}\subset\Lambda^k V
\end{equation}
where $e_1,\ldots,e_k$ is an orthonormal basis of $E.$

\begin{definition}
  A zonoid $K\in\ZZ(\Lambda^k V)$ is a \emph{Grassmannian zonotope} if there exists subspaces $E_1,\ldots,E_n\in G(k,V)$ scalars $\lambda_1,\ldots,\lambda_n\geq 0$ and a simple vector $c=c_1\wedge\cdots\wedge c_k\in\Lambda^k V$ such that $K=\lambda_1\seg{E_1}+\cdots+\lambda_n\seg{E_n}+\tfrac{1}{2}\{c\}$. A \emph{Grassmannian zonoid} is a limit of Grassmannian zonotopes. We denote the set of Grassmannian zonoids in $\Lambda^k V$ by $\GZ(k,V)\subset \ZZ(\Lambda^k V)$ and centered Grassmannian zonoids by $\GZo(k,V):=\GZ(k,V)\cap \ZZo(\Lambda^k V)$.
\end{definition}

\begin{remark} For $k\in\{0,1,m-1,m\}$ where $m:=\dim V$, all zonoids are Grassmannian.
\end{remark}


The following lemma clarifies how to recognize Grassmannian zonoids when represented by random vectors or by measures. In particular, centered Grassmannian zonoids in $\Lambda^k V$ correspond to positive measures on $G(k,V)$.

\begin{lemma}\label{lem:GZchar}
  Let $K\in \ZZo(\Lambda^k V)$. The following are equivalent.
  \begin{itemize}
    \item[(i)] $K\in\GZo(k,V)$;
    \item[(ii)] There is an integrable random vector $X\in\Lambda^k V$ that is almost surely simple, i.e. such that almost surely $X=X_1\wedge\cdots\wedge X_k$ (the vectors $X_1,\ldots, X_k$ can be \emph{dependent}), such that $K=\EE\seg{X}$
    \item[(iii)] The support of the measure $\mu_K\in\Mev^+(S(\Lambda^kV))$ is contained in the intersection of $S(\Lambda^kV)$ with the set of simple vectors, i.e. $\mu_K\in\Me^+(G(k,V))$.
  \end{itemize}
\end{lemma}

\begin{proof}
  The equivalence $(ii)\iff(iii)$ follows from \cref{prop:fromrvtomeas}. The equivalence $(i)\iff(iii)$ follows from the fact that Hausdorff convergence of zonoids corresponds to weak--$*$ convergence of measures \cite[Theorem~2.26(5)]{ZA}.
\end{proof}

\begin{remark}
As it will be clear from \cref{def:zonoidsec}, \cref{lem:GZchar}.(ii) implies that the value at $p\in M$ of the zonoid section $\zeta_X$ of a \zkrok field $X\randin \mC^1(M,\R^k)$ is a Grassmannian zonoids: $\zeta_X(p)\in \GZ(k,T_pM)$ for all $p\in M$.
\end{remark}

\begin{remark}
From (iii) we see that $\GZo(k,V) \cong \Me^+(G(k,V))$.
\end{remark}

It is not difficult, using (iii), to see that the Grassmannian zonoids are closed under the Minkowski sum. Similarly, one can see using (ii) that they are also closed under the wedge product.

\begin{lemma}
  The wedge product, respectively the Minkowski sum, of two Grassmannian zonoids is a Grassmannian zonoid.
\end{lemma}

The next lemma makes computations easier for Grassmannian zonoids and, for instance, it can be used to compute directly the constant in the proof of \cref{thm:Alphaca}. We will use it in the proof of \cref{lem:hate}.
\begin{lemma}\label{lem:lengthwithballs}
  Let $C\in\GZ(k,\R^m)$ and let $B_m:=B_{\R^m}$ be the unit ball of $\R^m$. Then we have
  \begin{equation}
    \ell(C)=\frac{1}{(m-k)!\uball_{m-k}}\ell\left(C\wedge B_m^{\wedge (m-k) }\right)
  \end{equation}
  where $\uball_d:=\vol_d (B_d)$.
\end{lemma}

\begin{proof}
  Since the length is translation invariant, we can assume $C$ is centered. Let $C=\EE\seg{X_1\wedge\cdots\wedge X_k}$, let $Y\randin\R^m$ be a Gaussian vector of mean $0$ and variance $\sqrt{2\pi}$ in such a way that $B_m=\EE\seg{Y}$ and let $Y_1,\ldots,Y_{m-k}$ be iid copies of $Y$ independents of $X_1\wedge\cdots\wedge X_k$. Then using the independence of the random variables and the fact that $Y_1\wedge\cdots\wedge Y_{m-k}$ is orthogonal invariant we have
  \begin{align}
    \ell\left(C\wedge B_d^{\wedge (d-k) }\right) &=\EE\|X_1\wedge\cdots\wedge X_k\wedge Y_1\wedge\cdots\wedge Y_{m-k}\|  \\
                                      &=\EE\|X_1\wedge\cdots\wedge X_k\|\cdot\EE\|e_1\wedge\cdots\wedge e_k\wedge Y_1\wedge\cdots\wedge Y_{m-k}\|
  \end{align}
  where $e_1,\ldots,e_m$ denotes the standard basis of $\R^m$. We obtain
  \begin{equation}
    \ell\left(C\wedge B_m^{\wedge (m-k) }\right)=\ell(C)\cdot\EE\| \pi(Y_1)\wedge\cdots\wedge \pi(Y_{m-k})\|
  \end{equation}
  where $\pi:\R^m\to\R^{m-k}$ is the orthogonal projection onto $\Span(e_{k+1},\ldots,e_m)$. Then it remains only to see, using \cref{prop:kintvolpow}, that $\EE\| \pi(Y_1)\wedge\cdots\wedge \pi(Y_{m-k})\|=\ell\left(\pi(B_m)^{\wedge{(m-k)}}\right)=\ell\left((B_{m-k})^{\wedge{(m-k)}}\right)= (m-k)!\uball_{m-k}$.
\end{proof}

Finally, we observe the following. Let $f:G(k,V)\to \R$ be a measurable function and denote also by $f$ its (even and) homogeneous extension on the cone of simple vectors. Then if $K=\EE\seg{X_1\wedge\cdots\wedge X_k}$ is a Grassmannian zonoid with generating measure $\mu_K\in\Me^+(G(k,V))$, we get that \cref{eq:intwrtgenmeas} becomes:
\begin{equation}\label{eq:intwrtgenmeasgrass}
    \EE f(X_1\wedge\cdots\wedge X_k)=\int_{G(k,V)}f\, d\mu_K.
\end{equation}

\subsection{Topology of zonoids}\label{sub:zonoidbundles}
We conclude this introduction to zonoids with a short comment on zonoid bundles. It will be useful to keep in mind this section in what follows, to understand the continuity of the zonoid section (\cref{def:zonoidsec}). Let $M$ be a manifold of dimension $m$ and let $\pi:E\to M$ be a topological vector bundle of rank $k$. The structure of vector bundle is given by the \emph{trivialization maps} $\chi_U: E|_U\simto U\times \R^k$ which are homeomorphisms that are linear isomorphism on the fibers. 

We can define the zonoid bundle $\ZZ(E)$ whose fiber at a point $p\in M$ is defined to be $\ZZ(E)_p:=\ZZ(E_p)$ where $E_p$ is the fiber of $E$ at $p$, and whose bundle structure is given by the collection of maps $\widehat{\chi_U}:\ZZ(E)|_U\simto U\times \ZZ(\R^k)$ in particular the topology on $\ZZ(E)$ is the smallest topology that makes all $\widehat{\chi_U}$ homeomorphisms. Recall that the space of zonoids $\ZZ(\R^k)$ is topologized by the Hausdorff distance, see \cref{eq:Hausdist}. Similarly one can define $\ZZo(E)$, $\GZ(k,E)$, $\GZo(k,E)$.

Given a fiber bundle $\pi:F\to M$ we denote by $\Gamma(F)$ the space of continuous sections of $F$, that is $\gamma\in\Gamma(F)$ if and only if $\gamma:M\to F$ is a continuous map such that for every $p\in M$, $\pi(\gamma(p))=p$. In particular a section $\zeta\in\Gamma(\ZZ(E))$ is the choice of a zonoid at each point $p$ of the manifold $M$ in the vector space $E_p$ such that this zonoid depends \emph{continuously} on the point $p$. We will call $\zeta$ a \emph{zonoid section}.

We observe then that a section $\zeta$ of the bundle $\ZZ(E)\to M$ defines at each point $p\in M$ a continuous positively homogeneous sublinear function $h_{\zeta(p)}:E_p^*\to \R$.
\begin{lemma}\label{lem:continuih}
$\zeta$ is continuous if and only if the map $h_{\zeta}:E^*\to\R, (p,u)\mapsto h_{\zeta(p)}(u)$ is a continuous function on $E^*$.
\end{lemma}
\begin{proof}
It is sufficient to prove the statement locally, thus we assume $E=\R^m\times \R^k$. 
Consider the space $\mC(\R^k)$ endowed with the compact-open topology. This has the property that: $h\in \mC(\R^m\times \R^k)$ if and only if $h_1\in \mC(\R^m,\mC(\R^k))$, where $h_1:p\mapsto h(p,\cdot)$.
Therefore, the statement translates into proving that a sequence of zonoids $\zeta_n\subset \R^k$ converges to a limit $\zeta$ if and only if the corresponding sequence of support functions $h_n\colon \R^k\to \R$ converges to $h:=h_{\zeta}$ in $\mC(\R^k)$ with respect to the compact-open topology. Now, we recall that $h_n$ and $h$ are positively homogeneous functions, which implies that $h_n\to h$ if and only if the same convergence holds for the restrictions to the sphere $S^{k-1}$. The compact-open topology of $\mC(S^{k-1})$ coincides with the one induced by the supremum norm, hence we conclude by \cref{eq:Hausdist}.
\end{proof}

\cref{lem:continuih} will be used in \cref{sec:zonoidsection} to show the continuity of the zonoid section.

We conclude this section with some observations regarding the space of zonoid sections, with the only scope of giving a more complete picture.
In fact, it is easy to turn the latter proof into a proof of the following statement. Linearity is meant with respect to the Minkowsky sum on the left and follows from \cref{prop:suppfctprop}.
\begin{proposition}
The assignment $\zeta\mapsto h_\zeta$ defines a linear topological embedding 
\begin{equation}
    h_{\cdot}:\Gamma(\ZZ(E))\hookrightarrow \mC(E^*),
\end{equation}
\end{proposition}
\begin{remark}
The exact image of $h_\cdot$ is not easy to determine, but it is certainly contained in the subset of functions that are sublinear on fibers, see \cref{sec:zonoids}.
\end{remark}


A further observation is that, as fiber bundles, we have $\ZZ(E)\cong \ZZo(E)\oplus E$ and thus
\begin{equation}\label{eq:splitzonbundle}
    \Gamma(\ZZ(E))\cong \Gamma(\ZZo(E))\oplus \Gamma(E).
\end{equation}
Therefore we can, as before, treat the \emph{nigiro} (see \cref{def:nigiro}) of a zonoid and the centered zonoid as separate continuous sections.
\section{{\zkrok} hypotheses}\label{sec:zkrokhypotheses}
Let $(M,g)$ be a smooth Riemannian manifold of dimension $m\in\N$, possibly non-compact. In this section we are going to describe a class of random functions $X:M\randto \R^k$ for which Kac-Rice formula works well and it can be written in terms of a \emph{field of zonoids} as explained in \cref{sec:intro}.
\begin{definition}[{\zkrok} hypotheses]\label{def:zkrok}

Let $X\colon M \randto \R^k$ be a random map. We say that $X$ is {\zkrok} if the following properties hold.
\begin{enumerate}
    \item\label{itm:krok1} $X\randin \mC^1(M,\R^k)$.
    \item\label{itm:krok2} Almost surely, $0$ is a regular value of $X$.
    \item\label{itm:krok:3} For any $p\in M$ the probability $[X(p)]$ on $\R^k$ is absolutely continuous with density denoted as $\rho_{X(p)}\colon \R^k\to [0,+\infty)$.
    \item\label{itm:krok4} The function $\rho_X\colon M\times \R\to \R$ given by $\rho_X(p,x)=\rho_{X(p)}(x)$ is continuous at $(p,0)$ for all $p\in M$.
    \item\label{itm:krok5} There exists a regular conditional probability $\mu(p,x)\in \mathscr{P}(\mC^1(M,\R^k))$ of $X$ given $X(p)$ (see \cref{sub:remarkrok} below) such that the following holds. Let $J_p\cdot \mu(p,x)\in \Me^+(\mC^1(M,\R^k))$ be the measure defined by
    \be 
    J_p\cdot \mu(p,x)(B)=\int_{B}J_pf \cdot d\tyu\mu(p,x)\uyt(f).
    \footnote{In the distributional sense, it is the multiplication of the measure $\mu(p,x)$ with the function $J_p\colon f\mapsto J_pf$.}
    \ee
    Then we ask that $J_p\cdot \mu(p,x)$  is a finite measure and that the function 
    \be 
    J_M \cdot \mu:M\times \R^k\to \Me^+\tyu\mC^1(M,\R^k)\uyt
    \ee
    \be 
    (p,x)\mapsto J_p\cdot \mu(p,x)
    \ee
    is continuous at $(p,0)$ for all $p\in M$.
\end{enumerate}
\end{definition}
These hypotheses are exactly what we need to apply the Kac-Rice formula to express the expectation of quantities of the form: 
\be 
I_\a(X):=\int_{X^{-1}(0)}\a(p,X)dM(p),
\ee
where $\a\colon M\times \mC^1(M,\R^k)\to \R$ is a measurable function, see Theorem \ref{thm:Alphaca}.
They are a variation of the KROK hypotheses introduced in \cite{KRStec}: a series of hypotheses on pairs $(X,W)$, where $X\colon M\to N$ is a random map and $W\subset N$ is a submanifold of codimension $m=\dim M$. If $(X,W)$ is KROK, then the measure $\mu(A):=\EE\#(X^{-1}(W)\cap A)$ is computed by a generalized Kac-Rice formula, see \cite[Theorem 2.2]{KRStec}.
In this paper, we only consider the case when $W=\{0\}\subset N=\R^k$ but we do not impose conditions on its codimension $k$.

The precise relation between the KROK hypotheses of \cite{KRStec} and the {\zkrok} hypotheses of Definition \ref{def:zkrok} is that $X$ is {\zkrok} if and only if the pair $(X,\{0\})$ satisfies all conditions KROK.$(\ell)$ for all $\ell\in \{\mathrm{i},\dots,\mathrm{vii}\}\smallsetminus\{\mathrm{v}\}$ in \cite[Definition 2.1]{KRStec}. Indeed KROK.$\mathrm{(v)}$ is a codimension assumption and it translates to our setting as the condition: $k=m$, which is not required for $X$ to be \zkrok. The hypothesis KROK.(vii) is equivalent to \zkrok.\ref{itm:krok5} by point (3) of Proposition \ref{prop:tecnicalzkrok} below, that is a more precise version of \cite[Prop. 2.4]{KRStec}. See also \cref{apx:comp} to compare with the hypotheses that appear in the more standard statements of Kac-Rice formulas, \cite{AdlerTaylor,AzaisWscheborbook}.

\begin{remark}
Although having a Riemannian metric $g$ on $M$ is useful to state \zkrok.\ref{itm:krok5}, the notion does not depend on $g$: If $X$ is {\zkrok} on $(M,g)$ then it is {\zkrok} on $(M,\Tilde{g})$ for any Riemannian metric $\tilde{g}$. 
This is easily seen by the fact that the functions $J_p$ and $\tilde{J}_p$ corresponding to the two metrics are related by an identity: $J_p=\f(p)\tilde{J}_p$ for some smooth function $\f\in \mC^\infty(M,(0,+\infty))$. 
\end{remark}
\begin{remark}\label{rem:bulinskaya}
The hypothesis \zkrok.\ref{itm:krok2} can be verified in some cases using the generalization of Bulinskaya Lemma proved in \cite[Prop. 6.12]{AzaisWscheborbook}. This says that if $X\randin \mC^2(M,\R^k)$ and the triple $(p,X(p),d_pX)$ has a joint density $\rho\colon J^1(M,\R^k) \to \R$, where $J^1(M,\R^k)$ is the first jet bundle, that is bounded on a compact neighborhood of each point $(p,0,A)\in J^1(M,\R^k)$, then \zkrok.\ref{itm:krok2} holds.
\end{remark}

\subsubsection{A comment about the notation}
The notation \emph{KROK}, introduced in \cite{KRStec}, stands for \emph{Kac-Rice OK}. Here, we add the letter \emph{z} for two reasons: to remind that we only care about the \emph{zeroes} and to indicate that some \emph{zonoid} will appear. {\zkrok} is pronounced ``skrok'', ``zkrok'' or ``zee krok''.

\subsection{Remarks on \zkrok-\cref{itm:krok5}}\label{sub:remarkrok}
Given a random element $X\randin \mC^1(M,\R^k)$ and a point $p \in M$,
a \emph{regular conditional probability}\footnote{
See \cite{dudley} or \cite{Erhan}. In the latter the same object is called a \emph{regular version of the conditional probability}.
} of $X$ given $X(p)$ is a function 
\be
\mu(p,\cdot)(\cdot)\colon \mathcal{B}(\mC^1(M,\R^k))\times \R^k\to [0,1],
\ee
\be 
(x,B)\mapsto \mu(p,x)(B)
\ee
that satisfies the following two properties, see \cite{dudley} (The definition for any fixed $p$ as it depends only on the pair of random variables $X$ and $X(p)$).
\begin{enumerate}[label=\alph*)]
\item\label{itm:propA} For every $B\in\mathcal{B}(\mC^1(M,\R^k))$, the function $\mu(p,\cdot)(B)\colon N\to [0,1]$ is Borel and for every $V\in \mathcal{B}(\R^k)$, we have
\be \label{eq:condproba}
\PP\{X\in B ; X(p)\in V\}=\int_V\mu(p,x)(B)d[X(p)](x)
\ee
where recall that $[X(p)]$ denotes the probability measure that is the law of the random vector $X(p)\randin\R^k.$
\item\label{itm:propB} For all $x\in N$, $\mu(p,x)$ is a Borel probability measure on $\mC^1(M,\R^k)$.
\end{enumerate} 
The fact that the space $\mC^1(M,\R^k)$ is Polish ensures that, for every $p\in M$, a regular conditional probability measure $\mu(p,\cdot)(\cdot)$ of $X$ given $X(p)$ exists (see \cite[Theorem 10.2.2]{dudley}) and it is unique up to $[X(p)]$-a.e. equivalence on $\R^k$. However, strictly speaking, it is not a well defined function of $p$, although the notation can mislead to think that. 

According to the above definition, there are many different choices of measures $\mu(p,x)\in \mathscr{P}(\mC^1(M,\R^k))$ with the property that 
$\mu(p,\cdot)(\cdot)$ is a regular conditional probability of $X$ given $X(p)$, for all fixed $p\in M$.
In our case such ambiguity may be traumatic, since we will be interested in the value of $\mu(p,x)$ at $x=0$ which, by \zkrok-\cref{itm:krok:3}, is negligible for the measure $[X(p)]$, i.e. $\PP\{X(p)=0\}=0$. Therefore, it is essential to choose a family of regular conditional probabilities $\{\mu_p\}_{p\in M}$ that has at least some continuity property at  $(p,x)\to (p_0,0)$. This is the motivation for the hypothesis \zkrok-\cref{itm:krok5}. 

\subsection{Notation for conditioned random maps}\label{sub:notakrok} We will use the notation of random elements, in the following sense. If $X\randin \mC^1(M,\R^k)$ is \zkrok, then for any $(p,x)\in M\times \R^k$, we write 
\be 
(X|X(p)=x)\randin\mC^1(M,\R^k)
\ee
for any random element representing the measure $\mu(p,x)$, i.e. such that ${[X|X(p)=x]}=\mu(p,x)$. Hence $(X|X(p)=x)$ is not a well defined random element but since in the sequel everything will only depend on the \emph{law} this will not be a problem.
Moreover, we will write 
\be 
\PP\kop X\in B |X(p)=x \pok:= \PP\kop (X|X(p)=x)\in B\pok =\mu(p,x)(B),
\ee
for every $B\subset \mC^1(M,\R^k)$ and 
\be 
\EE\kop \alpha(X) |X(p)=x \pok:= \EE\kop \alpha((X|X(p)=x))\pok=\int_{\mC^1(M,\R^k)}\alpha(f)d\mu(p,x)(f).
\ee
for every $\alpha\colon \mC^1(M,\R^k)\to \R$ measurable, whenever the integral, called \emph{expectation} in this context, makes sense. If $X$ is {\zkrok} then the probability $\mu(p,0)$ is unique, so the notation $[X|X(p)=x]$ is not ambiguous at $x=0$. More precisely, if $\mu(p,x)$ and $\mu'(p,x)$ are two regular conditional probabilities of $X$ given $X(p)$ satisfying \zkrok-\cref{itm:krok5} then $\mu(p,0)=\mu'(p,0).$  For all the other $x\in\R^k$, we will abuse the notation.

The following observation is often useful in computations.
\begin{remark}\label{rem:indcond}
Let $X\randin \mC^1(M,\R^k)$ and let $p\in M$. If $d_pX$ and $X(p)$ are stochastically independent, then the law of the random vector $d_pX$ is a regular conditional probability of $d_pX$ given $X(p)$, therefore  we have that the two laws are equivalent:
\be 
[d_pX]=[d_pX|X(p)=x], \quad \text{ for $[X(p)]$-almost every $x\in \R^k$. }
\ee
In particular, if $X$ is \zkrok,  the continuity of $\mu(p,x)$ at $x=0$ yields
\be 
[d_pX]=[d_pX|X(p)=0].
\ee
Therefore, in this case the zonoid section at $p$ is computed by:
\be 
\zeta_X(p)=\EE\kop [0,d_pX^1\wedge \dots \wedge d_pX^k] \pok \rho_{X(p)}(0).
\ee
\end{remark}
\subsubsection{The notation makes sense}
The Lemma below has the scope to clarify some doubts that often arise when using the notation explained above.
\begin{lemma}\label{lem:stupidconditioning}
Let $X\randin \mC^1(M,\R^k)$ and fix $p\in M$. Let $\mu(p,\cdot)(\cdot)$ be a regular conditional probability for $X$ given $X(p)$. Then $\mu(p,x)$ is supported on $\{f\in \mC^1(M,\R^k)\colon f(p)=x\}$ for $[X(p)]$-a.e.  $x\in \R^k$, that is, in the above notation, 
\be 
\PP\kop X(p)=x \ \Big|\ X(p)=x\pok=1,\quad \text{for $[X(p)]$-a.e. $x\in \R^k$.}
\ee
\end{lemma}
\begin{proof}
Let us fix $p\in M$. Let $V\subset \R^k$ be a Borel subset and define $B_V:=\{f\in\mC^1(M,\R^k)|f(p)\in V\}$. Then, by \cref{eq:condproba}, we have that 
\be 
\int_{V}d[X(p)](x)=
\PP\kop X(p)\in V\pok=\PP\kop X\in B_V\pok =\int_{\R^k}\mu(p,x)(B_V)d[X(p)](x).
\ee
It follows that there is a Borel subset $N_V\subset \R^k$, with $\PP\{X(p)\in N_V\}=1$ such that for every $x\in N_V$, we have \be
1_V(x)=\mu(p,x)(B_V)=\PP\kop X(p)\in V|X(p)=x\pok.
\ee
Let $\{V_n\}_{n\in\N}$ be a  countable basis of the topology of $N$. Let $B_n=B_{V_n}\subset \mC^1(M,\R^k)$ be defined as above. Then $\cap_n N_{V_n}:=N'\subset \R^k$ is still a full measure set for $[X(p)]$.  Clearly, we have that every singleton $x\in N$, can be written as a countable intersection
\be 
\{x\}=\bigcap_{\kop n\in \N:\ x\in V_n\pok}V_n.
\ee
Moreover, for every $x\in N'$ and every $n\in\N$, we have that $\mu(p,x)(B_n)=1_{V_n}(x)$. Therefore, if $x\in N'$, then we conclude by the continuity from above of the measure $\mu(p,x)$:
\be 
\PP\kop X(p)=x|X(p)= x \pok=\mu(p,x)\tyu B_{\kop x\pok}\uyt=\inf_{\kop n\in \N:\ x\in V_n\pok} 1_{V_n}(x)=1.
\ee
\end{proof}


\subsection{Equivalent formulations of \zkrok-\cref{itm:krok5}}
We derive a more technical version of the hypothesis \zkrok-\cref{itm:krok5}. 
See also \cref{apx:comp}. 
\begin{proposition}\label{prop:tecnicalzkrok}
Let $X:M\randto\R^k$ be a random map satisfying \zkrok-\cref{itm:krok1}-\cref{itm:krok4} and let $\mu(p,\cdot)(\cdot)$ $=:[X|X(p)=\cdot](\cdot)$ be a regular conditional probability of $X$ given $X(p)$ (See \cref{sub:remarkrok}). The following statements are equivalent:
\begin{enumerate}
    \item ({\zkrok-\cref{itm:krok5}}) The function 
    $ J_M\cdot \mu:M\times \R^k\to \Me^+(\mC^1(M,\R^k))$ is continuous at $(p,0)$ for all $p\in M$.
    \item\label{itm:tecnokrok2} For any bounded continuous function $\a\in  \mC_b(\mC^1(M,\R^k); \R)$ and any convergent sequence $(p_n,x_n)\to (p,0)$ in $M\times \R^k$ we have
\be 
    \EE\kop\a(X)J_{p_n}X\Big| X(p_n)=x_n\pok \to \EE\kop\a(X)J_pX\Big| X(p)=0\pok.
    \ee
    \item\label{itm:KROK8} For any bounded continuous function $\a\in  \mC_b(\mC^1(M,\R^k)\times M; \R)$, the function
    \be 
    M\times \R^k\ni (p,x)\mapsto \EE\kop\alpha(X,p)J_pX\Big| X(p)=0\pok
    \ee
    is finite and continuous at $(p,0)$ for every $p\in M$.
    \item\label{itm:tecnokrok4} For any sequence of continuous functions $\beta_n\to \beta_0\in \mC(\mC^1(M,\R^k); \R)$ that converges in the compact-open topology and any sequence $(p_n,x_n)\to (p_0,0)$ converging in $M\times \R^k$ such that $\beta_n (f)\le C J_{p_n}f$ for some $C>0$, we have that
    \be\label{eq:limbeta}
    \EE\kop\beta_n(X)\Big| X(p_n)=x_n\pok \to \EE\kop\beta_0(X)\Big| X(p_0)=0\pok.
    \ee
\end{enumerate}
\end{proposition}
\begin{proof}
$(1)\iff (2)$ by definition. Moreover, it is clear that $(4)\implies (3)\implies (2)$, so that it will be sufficient to show that $(1)\implies (4)$. In \cite[Proposition 2.4]{KRStec} it was proven that $(1)\implies (3)$, but a slight modification of the same argument allows to obtain the (apparently) stronger statement $(4)$. We are going to repeat it here, with some extra care, to prove the Proposition.

Assume $(1)$ and let $\beta_n,p_n,x_n\to \beta_0,p_0,0$ as in the statement of $(4)$. Observe that for all $\beta=\beta_n$ and $p=p_n$, if $J_{p}f=0$, then $\beta(f)=0$, so that
\bega
    \EE\kop\beta(X)\Big| X(p)=x\pok 
    =
    \int_{\mC^1(M,\R^k)}\beta(f)d\mu(p,x) (f)
    =
    \\
    =
    \int_{\mC^1(M,\R^k)\smallsetminus \{J_p=0\}}\beta(f)\frac{J_pf}{J_pf}d\mu(p,x) (f)
    +
    \int_{\mC^1(M,\R^k)\cap  \{J_p=0\}}\beta(f)d\mu(p,x) (f)
    \\=
    \int_{\mC^1(M,\R^k)}\frac{\beta(f)}{J_pf}d\tyu J_p\cdot \mu(p,x)\uyt (f).
\eega
Notice that the last term makes sense because $J_p\cdot \mu(p,x)\tyu\{J_p=0\}\uyt=0$.

Let $E(p,x):=\EE\kop J_pX|X(p)=x\pok$ be the total mass of the measure $J_p\cdot \mu(p,x)$. By \zkrok-\cref{itm:krok5}, the number $E(p,0)\ge 0$ is finite, though notice that it could be zero (See \cref{ex:parabolazza}). The hypothesis $(1)$ implies that $E(p_n,x_n)\to E(p,0)$. If $E(p_0,0)=0$, then the limit \eqref{eq:limbeta} holds since
\be 
    \left|\EE\kop\beta_n(X)\Big| X(p_n)=x_n\pok\right| \le CE(p_n,x_n)
    \to 0=
    \EE\kop\beta_0(X)\Big| X(p_0)=0\pok.
\ee
Assume that $E(p_0,0)>0$, then we can assume that $E(p_n,x_n)>0$ for all $n\in\N$. In this case, the next sequence of probabilities converges:
\be 
P_n:=E(p_n,x_n)^{-1}J_{p_n}\cdot\mu(p_n,x_n)\to P_0:=E(p_0,0)^{-1}J_{p_0}\cdot\mu(p_0,0).
\ee
Thus by Skorohod's Theorem (See \cite{Billingsley, Parth}) there exists a sequence of random functions $Y_n, Y_0\randin \mC^1(M,\R^k)$ defined on a common probability space such that $Y_n\to Y_0$ in $\mC^1(M,\R^k)$ almost surely. Then 
\bega 
\EE\kop\beta_n(X)\Big| X(p_n)=x_n\pok
&= E(p_n,x_n)\int_{\mC^1(M,\R^k)}\frac{\beta_n(f)}{J_{p_n}f}dP_n(f)
\\
&=E(p_n,x_n)\EE\kop \frac{\beta_n(Y_n)}{J_{p_n}f}\pok
\to 
E(p_0,0)\EE\kop \frac{\beta_0(Y)}{J_{p_0}f}\pok
\\
&=E(p_0,0)\int_{\mC^1(M,\R^k)}\frac{\beta_0(f)}{J_{p_0}f}dP_0(f)
\\
&=\EE\kop\beta(X)\Big| X(p)=0\pok.
\eega
Here the limit holds by dominated convergence, since $\frac{\beta_n(Y_n)}{J_{p_n}f}\le C$ and $\frac{\beta_n(Y_n)}{J_{p_n}f}\to \frac{\beta_0(Y_0)}{J_{p_0}f}$ almost surely.
\end{proof}
To show that a given random field verifies \zkrok-\ref{itm:krok5}, it is often convenient to check directly that it satisfies  point \ref{itm:tecnokrok2} of \cref{prop:tecnicalzkrok} above, which is equivalent to \zkrok-\ref{itm:krok5} by definition.
On the other hand, the apparently stronger formulation given in point \cref{itm:tecnokrok4} is the one that we will refer to in the subsequent proofs, in order to deduce other properties of \zkrok fields. We also note that \zkrok-\ref{itm:krok5} is an equivalent formulation of property KROK.(vii) of \cite[Definition 2.1]{KRStec}, that is point \ref{itm:KROK8}.
\begin{remark}
\cref{prop:tecnicalzkrok} is based on the same principle as the theorem of Banach-Steinhaus \cite[Chapter 2]{brezis2010functional}.
\end{remark}
\begin{example}\label{ex:parabolazza}
There are examples of random maps $X\randin C^1(M,\R)$ that are \zkrok, thus in particular 
\be 
\PP\kop X(p)=0\implies  J_pX>0, \forall p\in M\pok=1,
\ee 
but for which there are points $p\in M$ with  $\EE\kop J_pX|X(p)=0\pok=0$. It is possible to build such examples on any manifold $M$ by generalizing the following construction.

Let $\gamma_1,\gamma_2\sim N(0,1)$ be independent normal Gaussians. Define $X\randin \mC^\infty(\R,\R)$ as
\be 
X(u):=u^2\gamma_1 +\gamma_2.
\ee
By \cref{prop:Gausmthkrok}, in the next subsection, the field $X$ is {\zkrok} and the probability $\mu(u_0,0)$ is represented by the random field such that $(X(u)|X(u_0)=0)=(u^2-u_0^2)\gamma_1$. Thus, $\EE\kop J_0X|X(0)=0\pok=0$. See also \cref{ex:randlev}.
\end{example}

\subsection{The Gaussian case}
Assume that the random map $X\colon M \randto \R^k$ is Gaussian, see \cite{dtgrf,AdlerTaylor}. As it should be expected, in this case the {\zkrok}  hypotheses are much simpler, in particular \zkrok.\ref{itm:krok5} is automatically satisfied.
\begin{proposition}\label{prop:gauskrok1}
Let $X$ be a Gaussian random field on $M$ with values in $\R^k$ such that
\begin{enumerate}
    \item $X\randin \mC^1(M,\R^k)$;
    \item Almost surely, $0$ is a regular value of $X$;
    \item For any $p\in M$ the Gaussian vector $X(p)\randin \R^k$ is non-degenerate: ${\det\EE\{X(p)X(p)^T\}\neq 0}$;
\end{enumerate}
Then $X$ is \zkrok.
\end{proposition}
\begin{proof}
In \cite[Section 9.1]{KRStec} the author uses \cite[Lemma 9.1]{KRStec} to prove the validity of \zkrok.\ref{itm:krok5}, in the equivalent form reported in Proposition \ref{prop:tecnicalzkrok}, point (\ref{itm:KROK8}).
\end{proof}
Actually, the requirement that $0$ is almost surely a regular value is, in many cases, redundant. We already seen that when $X\in\mC^2$, one can use the generalized Bulinskaya Lemma, see Remark \ref{rem:bulinskaya}. However, in the Gaussian case, if the field is smooth\footnote{The requirement that $X\randin\mC^r$ for $r$ large enough would be sufficient, however, the authors do not know precisely \emph{how large} $r$ should be.} then by \cite[Theorem 7]{dtgrf} we have that \emph{(3)} implies \emph{(2)}. This can be thought as a manifestation of Sard's theorem (see \cite{Hirsch}), so that it should not be surprising that a regularity higher than $\mC^1$ is required\footnote{Sard's theorem \cite{sard} states that the set of critical values of a map $f\colon \R^m\to\R^k$ of class $\mC^r$ has measure zero, provided that $r\ge 1+\max\{0,m-k\}$.}.
\begin{proposition}\label{prop:Gausmthkrok}
Let Let $X$ be a Gaussian random field on $M$ with values in $\R^k$ such that
\begin{enumerate}
    \item $X\randin \mC^\infty(M,\R^k)$;
    \item For any $p\in M$ the Gaussian vector $X(p)\randin \R^k$ is non-degenerate: ${\det\EE\{X(p)X(p)^T\}\neq 0}$;
    \end{enumerate}
Then $X$ is \zkrok.
\end{proposition}
\begin{proof}
Combine Proposition \ref{prop:gauskrok1} with \cite[Theorem 7]{dtgrf} as discussed above.
\end{proof}

\section{The zonoid section}\label{sec:zonoidsection}
We are now ready to define the main object of this paper. We recall, from \cref{sub:zonoidbundles} that a zonoid section $\zeta\in\Gamma(\ZZ(\Lambda^kT^*M))$ is the choice of a zonoid at each point $p$ of the manifold $M$ in the vector space $\Lambda^kT_p^*M$ such that this zonoid depends \emph{continuously} on the point $p$.
\begin{definition}\label{def:zonoidsec}
Let $X=(X^1,\ldots,X^k)\randin \mC^1(M,\R^k)$ be \emph{z-KRoK}. The associated \emph{zonoid section} $\zeta_X\in \Gamma\left(\ZZ(\Lambda^k T^*M)\right)$ is defined for every $p\in M$ by 
\begin{equation}
    \zeta_X(p):=\EE\kop\qwe0,\dd_pX^1\wedge\cdots\wedge\dd_p X^k\ewq\Big|X(p)=0\pok\rho_{X(p)}(0).
\end{equation}
\end{definition}
The fact that this definition is well posed, i.e. that the section $\zeta_X$ is indeed continuous, is a consequence of \cref{prop:hcontinuity} below.
This definition has to be intended in the following sense: let $[X|X(p)=0]=\mu(p,0)$ be the probability measure implied by the {\zkrok} condition and represented by a random map $(X|X(p)=0)\randin \mC^1(M,\R^k)$, as explained in \cref{sub:notakrok}. Then we consider the random covector $(\dd_pX^1\wedge\cdots\wedge\dd_p X^k|X(p)=0)=:Y\randin \Lambda^kT_p^*M$ and form the random segment $[0,Y]\subset \Lambda^kT^*_pM$. This is, in particular, a random zonoid and we can take its expectation as explained in \cref{sec:zonandrv} (we will see that $\EE\|Y\|<+\infty$ in a moment), and build the zonoid $\zeta_X(p)\subset \Lambda^kT^*_pM$ having support function $
h_{\zeta_{X}(p)}\colon \Lambda^kT_pM\to \R$ given, for every $u\in\Lambda^kT_pM$, by
\be 
h_{\zeta_X(p)}(u)=\rho_{X(p)}(0) \EE\max\kop0,\langle Y,u\rangle \pok.
\ee 
We denote by $h_{\zeta_X}\colon \Lambda^kTM\to\R$ the function given by $(p,u)\mapsto h_{\zeta_X(p)}(u)$. The following property is a useful consequence of the {\zkrok} hypotheses.
(see equation \eqref{eq:suppofexpseg} and the precedent discussion.)
\begin{proposition}\label{prop:hcontinuity}
$h_{\zeta_X}\colon \Lambda^kTM\to \R$ is continuous.
\end{proposition}
\begin{proof}
Let $(p_n,u_n)\to (p_0,u_0)$ be a converging sequence in $\Lambda^kTM$. Define ${\beta_n \colon\mC^1(M,\R^k)\to \R}$ as
\be 
\beta_n(f):=\max\kop 0,\langle d_{p_n}f^1\wedge\dots\wedge d_{p_n}f^k,u_n\rangle\pok
\rho_{X(p_n)}(0). \ee
Clearly $\beta_n$ is continuous and, by \zkrok-\cref{itm:krok:3}, it converges: $\beta_n\to \beta_0$ in the compact-open topology of ${\mC\tyu\mC^1(M,\R^k);\R\uyt}$. Moreover, since $u_n$ converges, there exists a constant $C>0$ such that
\be 
\beta_n(f)\le \|d_{p_n}f^1\wedge\dots\wedge d_{p_n}f^k\|\|u_n\|\le C J_{p_n}f.
\ee
Applying \cref{prop:tecnicalzkrok}, with $x_n=0$, we obtain 
\bega
    \lim_{n\to+\infty}h_{\zeta_X(p_n)}(u_n)
    &=
    \lim_{n\to+\infty}\EE\kop\beta_n(X)\Big| X(p_n)=x_n\pok
    \\
    &=
    \EE\kop\beta_0(X)\Big| X(p_0)=0\pok
    \\
    &=
    h_{\zeta_X(p_0)}(u_0).
    \eega
\end{proof}
By \cref{lem:continuih} this ensures that the function $\zeta_X\colon M\to \ZZ(\Lambda^k T^*M)$ is indeed continuous and that \cref{def:zonoidsec} was well posed:  $\zeta_X\in \Gamma\left(\ZZ(\Lambda^k T^*M)\right)$.
\subsection{The Pull-back property}
We now establish a simple and very useful criteria for building {\zkrok} maps out of others in a seemingly functorial way. This is also reminiscent of a property of the characteristic classes of vector bundles.
\begin{mainthm}\label{thm:pullback}
Let $X\randin C^1(M,\R^k)$ be \zkrok. Let $S$ be a smooth manifold and let $\f\colon S\to M$ be a smooth map such that $\f\transv X^{-1}(0)$ almost surely. Then $X\circ \f\randin C^1(N,\R^k)$ is {\zkrok} and 
\be\label{eq:pullback} 
\zeta_{X\circ \f}(q)=d_q\f^*\zeta_X(\f(q)), \quad \forall q\in S.
\ee
\end{mainthm}
\begin{proof}
Assuming the first part of the statement, the formula \eqref{eq:pullback} is obvious from the definition of $\zeta_X$. To prove the theorem we have to show that the random map $X\circ \f$ satisfies all the five properties of \cref{def:zkrok}, with respect to any Riemannian metric on $S$.
\begin{enumerate}[wide]
    \item $X\circ \f\randin \mC^1(M,\R^k)$, by definition.
    \item The fact that $0$ is a regular value of $X\circ\f$ is completely equivalent (under the condition that $0$ is a regular value of $X$) to the hypothesis $\f\transv X^{-1}(0)$.
    \item For $q\in S$, the probability $[(X\circ \f)(q)]=[X(\f(q))]$ on $\R^k$ has density $\rho_{(X\circ\f)(q)}(\cdot)\colon \R^k\to [0,+\infty]$, where $\rho_{(X\circ\f)(q)}(x):=\rho_{X(\f(q))}(x)$. 
    \item Since $\f$ is continuous and $\rho_X$ is continuous at $(p,0)$, it follows that $\rho_{X\circ\f}$ is continuous at $(q,0)$ for any $q\in S$.
    \item Let $\mu(p,x):=[X |X(p)=x]\in\mathscr{P}(\mC^1(M,\R^k))$ be the regular conditional probability on $\mC^1(M,\R^k)$ associated to the {\zkrok} random map $X$. By assumption, the function 
    \be 
    J_M\cdot \mu\colon M\times \R^k\to \Me^+\tyu \mC^1(M,\R^k)\uyt
    \ee
    is continuous at $(p,0)$. Let $\f^*\colon \mC^1(M,\R^k)\to \mC^1(S,\R^k)$ be the function given by $\f^*(f)\leonew{:}=f\circ \f$. This is continuous with respect to the $\mC^1$ topologies and we define $\nu(q,x):=\f^*_{\#}\mu(\f(q),x)$ to be the push-forward of $\mu(\f(q),x)$ via  $\f^*$. So $\nu(q,x)$ is the probability measure such that for every measurable function $F\colon \mC^1(M,\R^k)\to [0,+\infty]$, we have
    \be 
    \int_{\mC^1(S,\R^k)}F(g)d\nu(q,x)(g)=\EE\kop F(\f^*(X)) |X(\f(p))=x\pok.
    \ee
    From this, one can see that $\nu(q,\cdot)(\cdot)$ is a regular conditional probability of $X\circ\f$ given $(X\circ \f)(q)$ (see \cref{sub:remarkrok}). Indeed, for every $B\in \mathcal{B}(\mC^1(M,\R^k))$, by taking $F\leonew{:}=1_{B}$, we see that 
    \be 
    \nu(q,x)(B)=\PP\kop X\circ \f\in B |X(\f(p))=x\pok
    \ee
    is Borel measurable with respect to $x\in\R^k$ and for any $V\in \mathcal{B}(\R^k)$, by taking $F(g)\leonew{:}=1_B(g)1_V(g(q))$ we obtain 
    \begin{multline}
    \PP\{X\circ\f\in B; (X\circ\f)(q)\in V\}
    =
    \EE\kop1_B(X\circ \f) 1_V(X(\f(q)))\pok
    \\
    =
    \int_{\R^k}
    \EE\kop 1_B(X\circ \f) 1_V(X(\f(q)))|X(\f(p))=x\pok d[X(\f(p))](x)=
    \\
    =
    \int_{\R^k}
    \nu(q,x)(B)d[(X\circ\f)(p)](x),
    \end{multline}
    so that Property \cref{itm:propA} is proven. Moreover, it is obvious by the construction that $\nu(q,x)$ is a Borel probability, indeed it follows by the measurability of the function $f^*$, thus Property \ref{itm:propB}. 
    
    At this point, we proved that that for any $q\in S$, we have the regular conditional probability $\nu(q,\cdot)(\cdot)$. To conclude the proof we have to show the continuity of ${J_q\cdot\nu(q,x)}$ at $(q,0)$. Let $\a\colon \mC^1(S,\R^k)\to [0,1]$ be continuous. Let $(q_n,x_n)\to (q,0)$ be a converging sequence in $S\times \R^k$. Then
    \be \label{eq:thiseqinthisproof}
    \int_{\mC^1(S,\R^k)}\a (g)(J_{q_n} g) d\nu(q_n,x_n)(g)=\EE\kop \a(X\circ \f)\tyu J_{q_n}(X\circ\f)\uyt|X(\f(q_n))=x_n\pok=\dots
    \ee
    Observe that the normal Jacobians satisfy the inequality
    \be 
    J_{q_n}(X\circ\f)\le J_{\f(q_n)}X\cdot J_{q_n}\f\le C\cdot J_{\f(q_n)}X,
    \ee
    where the last inequality is due to the facts that the sequence ${q_n}$ is contained in a compact subset of $S$  and that $J_q\f$ is continuous in $q$, because $\f\in\mC^1$. 

    It follows that we can apply \cref{prop:tecnicalzkrok} to the sequence of points $(p_n,x_n)\leonew{:}=(\f(q_n),x_n)$ and the continuous functions $\beta_n$ defined as
    \be 
    \beta_n(f):=\a\tyu f\circ\f\uyt J_{q_n}(f\circ\f)\to \a(f\circ\f)J_q(f\circ \f).
    \ee
    The above sequence converges in the compact-open topology of $\mC_b(\mC^1(M,\R^k);\R)$. Indeed, since $\mC^1(M,\R^k)$ is metrizable, this is equivalent to say that whenever $f_n\to f$ in $\mC^1(M,\R^k)$, then $\beta_n(f_n)\to\beta(f)$.
    Now, $f_n\to f$ converges in $\mC^1(M,\R^k)$ if and only if $j^1_{q_n}f_n\to j^1_qf$ in $J^1(M,\R^k)$ for every converging sequence $q_n\to q$, thus, in particular, $J_{q_n}\to J_qf$, since $J_qf$ depends continuously on $j^1_qf$. By \cref{prop:tecnicalzkrok} we get that \cref{eq:thiseqinthisproof} becomes
    \be 
    \dots =\EE\kop\beta_n(X)|X(p_n)=x_n\pok\to \EE\kop\a(X\circ \f)J_q(X\circ\f)|X\circ\f(q)=0\pok,
    \ee
    which proves the thesis. 
\end{enumerate}
\end{proof}

\subsection{Independent intersection and wedge product}
If $X_1\randin C^1(M,\R^k)$ and $X_2\randin C^1(M,\R^l)$ are two {\zkrok} fields, one can build another random field $Y=(X_1,X_2)\randin C^1(M,\R^{k+l})$ whose zero set is the intersection of the previous two zero sets: $Y^{-1}(0)=X_1^{-1}(0)\cap X_2^{-1}(0).$ In the case where $X_1$ and $X_2$ are independent, we prove that the zonoid section of the new field is the wedge product of the previous zonoid sections.

\begin{mainthm}\label{thm:wedge}
Let $X_1\randin C^1(M,\R^k)$ and $X_2\randin C^1(M,\R^l)$ be independent  {\zkrok} fields. Then $Y:=(X_1,X_2)\randin C^1(M,\R^{k+l})$ is {\zkrok} and we have for all $p\in M$
\begin{equation}
    \zeta_{Y}(p)=\zeta_{X_1}(p)\wedge\zeta_{X_2}(p).
\end{equation}
\end{mainthm}
\begin{proof}
Conditions {\zkrok} \cref{itm:krok1} to \cref{itm:krok4} are immediately satisfied, note that since $X_1$ and $X_2$ are independent we have for all $x_1\in\R^k$, $x_2\in\R^l$ and all $p\in M$: $\rho_{Y(p)}(x_1,x_2)=\rho_{X_1(p)}(x_1)\rho_{X_2(p)}(x_2).$ To see that {\zkrok}\cref{itm:krok5} is satisfied it is enough to see that if $\mu_i(\cdot,\cdot)$ is a regular conditional probability for $X_i$ then $\mu(p,(x_1,x_2)):=\mu_1(p,x_1)\otimes\mu_2(p,x_2)$ is a regular conditional probability of $Y$ given $Y(p)$. 
With such choice of $\mu$, one can prove that $Y$ satisfies
{\zkrok}-\cref{itm:krok5}, by repeating the reasoning used in the proof of \cref{thm:pullback}.
In particular, in the notation introduced in \cref{sub:notakrok}, we have that for all $p\in M$, the random vectors $(X_1|X_1(p)=0)$ and $(X_2|X_2(p)=0)$ are independent.

Now it remains to observe that by definition of the field $Y$, we have for all $p\in M$:
\begin{equation}
    \dd_pY^1\wedge\cdots \dd_pY^{k+l}=(\dd_pX_1^1\wedge\cdots\wedge\dd_pX_1^k)\wedge( \dd_pX_2^1\wedge\cdots\wedge\dd_pX_2^l).
\end{equation}
Hence, using \cref{eq:wedgeandVit}, we have 
\begin{multline}
    \rho_{Y(p)}(0)[0,\dd_pY^1\wedge\cdots \dd_pY^{k+l}]=
    \\
    \left(\rho_{X_1(p)}(0)[0,\dd_pX_1^1\wedge\cdots\wedge\dd_pX_1^k]\right)\wedge\left( \rho_{X_2(p)}(0)[0,\dd_pX_2^1\wedge\cdots\wedge\dd_pX_2^l]\right).
\end{multline}
The result then follows by taking expectations on both sides and from the independence observed earlier.
\end{proof}

\subsection{Bernoulli combination and Minkowski sum}
Another simple operation on random fields allows to build the convex combination of the zonoid sections.

\begin{proposition}\label{prop:bern}
Let $X_0,X_1\randin C^1(M,\R^k)$ be {\zkrok} and let $\epsilon\randin\{0,1\}$ be a Bernoulli random variable of parameter $t\in[0,1]$ independent of $X_0$ and $X_1$, that is $\epsilon=0$ with probability $t$ and $\epsilon=1$ with probability $1-t$. Assume, in addition, that
\be\label{hyp:bern}
    (*)\ \text{ there is no point $p\in M$  such that $\rho_{X_i}(p,0)=0$ for both $i=0,1$.}
\ee
Let $X_t:=\epsilon X_0 + (1-\epsilon)X_1$. Then $X_t\randin C^1(M,\R^k)$ is {\zkrok} and we have for all $p\in M$
\begin{equation}
    \zeta_{X_t}(p)=(1-t)\zeta_{X_0}(p)+t\zeta_{X_1}(p).
\end{equation}
\end{proposition}

\begin{proof}
The properties {\zkrok} \cref{itm:krok1} to \cref{itm:krok4} are satisfied by $X_t$ and observe that for all $p\in M$, we have $\rho_{X_t(p)}=(1-t)\rho_{X_0(p)}+t\rho_{X_1(p)}$. Let $\mu_i(p,x)$ be a regular conditional probability for $X_i$ given $X_i(p)$, $i=0,1$. We prove that \begin{equation}\label{eq:mutbern}
    \mu_t(p,x):=\frac{(1-t)\rho_{X_0(p)}(x)\mu_0(p,x)+t\rho_{X_1(p)}(x)\mu_1(p,x)}{\rho_{X_t(p)}(x)}
\end{equation}
is a regular conditional probability for $X_t$ given $X_t(p)$. Indeed, let $B\subset C^1(M,\R^k)$ and $V\subset \R^k$ be Borel subsets, then, by definition of $X_t,$ we have for all $p\in M$, 
\begin{align}
    \PP\left(X_t\in B;\ X_t(p)\in V\right)    &=(1-t)\PP\left(X_0\in B;\ X_0(p)\in V\right) + t \PP\left(X_1\in B;\ X_1(p)\in V\right)  \\
    &=\int_V\left((1-t)\mu_0(p,x)(B)\rho_{X_0(p)}(x)+t\mu_1(p,x)(B)\rho_{X_1(p)}(x)\right) \dd x
\end{align}
where the first equality follows from the definition of $X_t$ and the second from the property of conditional probabilities given in~\eqref{eq:condproba}. And thus we obtain 
\begin{equation}
    \PP\left(X_t\in B;\ X_t(p)\in V\right)=\int_V\mu_t(p,x)(B)\rho_{X_t(p)}(x) \dd x.
\end{equation}
Moreover $\mu_t(p,x)$ is a probability measure for all $p\in M$, $x\in \R^k$ thus it is a regular conditional probability for $X_t$. The hypothesis $(*)$ guarantees that $\mu_t$ satisfies {\zkrok} \cref{itm:krok5}, since $\mu_0$ and $\mu_1$ do. 
Finally, the result follows from the fact that $\rho_{X_t(p)}(0)\mu_t (p,0)=(1-t)\rho_{X_0(p)}(0)\mu_0 (p,0)+t \rho_{X_1(p)}(0)\mu_1 (p,0)$ for all $p\in M$.
\end{proof}
\begin{remark}
The hypothesis $(*)$ in \cref{prop:bern} is what allows to avoid the difficulties coming from the denominator in \cref{eq:mutbern} when proving that $X_t$ satisfies \zkrok-\ref{itm:krok5}. It is not a necessary condition, although in general the field $X_t$ may fail to be \zkrok. 
\end{remark}
\begin{remark}\label{rem:weakzkrok}
We believe that the {\zkrok} Hypotheses, as stated in \cref{def:zkrok}, are a bit more restricting than necessary. Indeed, the continuity condition in  \ref{itm:krok5} could probably be replaced by the weaker conditions that the product $(p,x)\mapsto \rho_{X(p)}J_p\cdot\mu(p,x)$ is continuous at $(p,0)$ for all $p\in M$ and that $E(p,x)=\EE\{J_pX|X(p)=x\}$ is locally bounded, without affecting the results of the paper except for \cref{prop:bern}, in which the hypothesis $(*)$ could be dropped, and \cref{thm:abundance} which we will discuss in \cref{sec:ex} below.
\end{remark}
\section{The Alpha formula}
We will use the following version of Kac-Rice formula to deduce all our results. This is obtained as a particular case of \cite{KRStec}. See Appendix \ref{apx:comp} for a detailed comparison with the standard statements of Kac-Rice formula in \cite{AzaisWscheborbook} and \cite{AdlerTaylor}. The only differences are in the hypotheses, in particular the statement below is almost identical to \cite[Theorem 6.7]{AzaisWscheborbook}.
\begin{proposition}[$\alpha$-Kac-Rice formula]\label{thm:alphaKR}
Let $(M,g)$ be a Riemannian manifold of dimension $m\in\N$. Let $F\colon M\randto \R^m$ be a {\zkrok} random field. Let $\a\colon \mC^1(M,\R^m)\times M\to \R$ be a Borel measurable function. Then
\be\label{eq:aKR} 
\EE\kop\sum_{p\in F^{-1}(0)}\alpha(F,p)\pok=\int_M\delta^\a_F(p)dM(p).
\ee
Where
\be 
\delta^\a_F(p)=\EE\kop \a(F,p)J_pF \Big| F(p)=0\pok \rho_{F(p)}(0).
\ee
\end{proposition}
\begin{proof}
In the language of \cite[Theorem 4.1]{KRStec}, if $F$ is {\zkrok} with values in $\R^{\dim M}$, then the pair $(X,\{0\})$ is \emph{KROK}.
\end{proof}
The name Kac-Rice formula is often used to denote also a more general version of \cref{thm:alphaKR} which allows to deal with the case in which $X^{-1}(0)$ is not zero dimensional, see \cite[Theorem 6.8]{AzaisWscheborbook}. The additional flexibility provided by Theorem \ref{thm:Alphaca} below is crucial for us, since we want to be able to build a framework of calculus for intersections of random submanifolds $X^{-1}(0)$ of arbitrary codimension.
\begin{theorem}[Alpha Formula]\label{thm:Alphaca}
Let $k\le m\in \N$. Let $(M,g)$ be a Riemannian manifold of dimension $m$. Let $X\colon M\randto \R^k$ be a {\zkrok} random field and define the random submanifold $Z:=X^{-1}(0)$. Let $\a\colon \mC^1(M,\R^k)\times M\to \R$ be a Borel measurable function. Then
\be\label{aformula} 
\EE\kop\int_{ Z}\alpha(X,p)dZ(p)\pok=\int_M\delta^\a_X(p)dM(p).
\ee
Where
\be\label{eq:alphadensity}
\delta^\a_X(p):=\EE\kop \a(X,p)J_pX \Big| X(p)=0\pok \rho_{X(p)}(0).
\ee
\end{theorem}

The proof will be given later, in \cref{sec:proofAlpha}, after some preliminaries. In \cite[Theorem 6.10]{AzaisWscheborbook} the analogous statement for Gaussian fields is reported mentioning that the proof follows the same lines as in the case $m=k$. 
Here, to prove its validity under our {\zkrok} hypotheses, we are going to use a different strategy.
We are going to prove that, with little work and using the Pull-back property (\cref{thm:pullback}), Theorem \ref{thm:Alphaca} is a natural consequence of Theorem \ref{thm:alphaKR}. This method of proof is new and interesting in that it shows how it's always possible to reduce everything to the zero dimensional case using the construction, by Adler and Taylor \cite{AdlerTaylor}, of Gaussian fields that represent the Riemannian structure, see Section \ref{sec:ATfield}. Moreover, it is fully in the spirit of this work to investigate the relations between the various Kac-Rice formulas.
\subsection{The Adler-Taylor metric and normal fields}\label{sec:ATfield}
In \cite[Section 12]{AdlerTaylor} Adler and Taylor introduced and developed the concept of the Riemannian metric induced by a sufficiently regular random field $\y\colon M\randto \R$ on a smooth manifold:
\be\label{eq:ATmetric}
g^\y_{AT}(p)(v,w)=\EE\{d_p\y(v)\cdot d_p\y(w)\}.
\ee 
We will refer to $g_{AT}^\y$ as \emph{the Adler-Taylor metric} induced by $\y$. Given a Riemannian manifold $(M,g)$, it will be very useful for us to express $g$ as the Adler-Taylor metric induced by some smooth Gaussian field with unit variance.

\begin{definition}\label{def:ATfield}
Let $(M,g)$ be a Riemannian manifold and $\y\randin \mC^\infty(M)$ be a smooth Gaussian random field. 
We will say that $\y$ is a \emph{normal field} on $(M,g)$ if $\y(p)\sim \mathcal N(0,1)$ for every $p\in M$ and $g=g^\y_{AT}$.
In this case we will write $\y\sim \mathcal N(M,g)$.
\end{definition}
\begin{remark}
The law of the normal field $\y\sim \mathcal N(M,g)$ is not uniquely determined. It depends exactly on the choice of an isometric immersion of $(M,g)$ into the sphere of an Hilbert space. By  Nash's isometric embedding theorem, every smooth Riemannian manifold $(M,g)$ admits a normal field $\y$ with finite dimensional support $\text{supp}(\y)\subset\mC^\infty(M,\R)$. See also \cite{AdlerTaylor} and \cite{Nicolaescu2016}.
\end{remark}
By \cref{def:ATfield} it is clear that if $\y\sim \mathcal{N}(M,g)$ then for every smooth submanifold $Z\subset M$ with induced metric $g|_Z$ we have $\y|_Z\sim \mathcal{N}(Z,g|_Z)$. This property, together with the following lemma makes the normal field a very good tool to express integrals over the manifold. 
\begin{lemma}\label{lem:Montecarlo}
Let $(M,g)$ be a Riemannian manifold of dimension $m$, let $\y\sim\mathcal N(M,g)$ and let $Y^1,\ldots,Y^m,$ be i.i.d. copies of $\y$. Define the random discrete set $\Sigma:=\{Y^1=\dots =Y^m=0\}$. Let $\a\colon M\to \R$ be Borel with compact support. Then we have
\begin{equation} 
\int_M\a(p)=\frac{s_m}{2}\EE\kop\sum_{p\in\Sigma}\a(p)\pok
\end{equation}
where recall that $s_m:=\vol_m(S^m).$
\end{lemma}
\begin{proof}
Let $Y=(Y^1,\dots,Y^m)\colon M\randto \R^m$. First note that, since $Y(p)\sim N(0,\mathbbm{1}_m)$ for all $p\in M$, by differentiating $\EE\{|Y^i(p)|^2\}=1$ with respect to $p$ we see that the random vectors $Y(p)$ and $d_pY$ are independent.
By \cref{thm:alphaKR}, we have that
\be 
\EE\kop\sum_{p\in\Sigma}\a(p)\pok=\int_M\a(p)\EE\kop J_pY|Y(p)=0\pok\rho_{N(0,1)}(0)dM(p)=c(m)\int_M\a(p)dM(p)
\ee
where the last equality is due to two facts: the first is that for any fixed $p\in M$, the random vectors $Y(p)$ and $d_pY$ are independent; the second is that in an orthonormal frame, all the rows of $d_pY$ are identically distributed standard Gaussian vectors in $\R^m$. 

Now, the constant $c(m)$ can be computed by writing more carefully the formula, but there is a quicker way. The above identity should be true in the case when $M=S^m$, $\a=1$ and $\y\colon S^m\randto \R$ is a normal field on $S^m$ defined as $\y(p)=\langle\gamma, p\rangle$ for $\gamma\sim N(0,\mathbbm{1}_m)$. In such case $\Sigma$ is 
almost surely a pair of antipodal points, thus
\be 
c(m)=\frac{1}{s_m}\EE\#\Sigma=\frac{2}{s_m}.
\ee
\end{proof}

\subsection{Proof of the Alpha Formula (\cref{thm:Alphaca})}\label{sec:proofAlpha}

Let $k\le m\in \N$ and let 
\begin{equation}
    X\colon M\randto \R^k
\end{equation}
be a \zkrok field. Let $d:=m-k$, let $Y^1,\ldots, Y^d\sim \mathcal{N}(M,g)$ be i.i.d. normal fields independent of $X$ and let $Y:=(Y^1,\dots,Y^d)\colon M\randto \R^d$. We write $Z:=X^{-1}(0)$ and $\Sigma:=Y^{-1}(0)$ and we let $F:=(X,Y)\colon M\randto \R^m$.

\subsubsection{Intersection with a normal field}
By \cref{thm:wedge}, $F$ is {\zkrok}. 
By integrating first with respect to $Y$, using the independence of $X$ and $Y$, we deduce the following identity from \cref{thm:alphaKR}:
\be\label{eq:aluno}
\EE\int_Z\a(X,p)=\frac{s_d}{2}\EE\kop\sum_{p\in \Sigma\cap Z}\a(X,p)\pok=\dots
\ee
Now, we apply \cref{thm:alphaKR} with $\alpha(F,p):=\alpha(X,p)$ depending only on the first factor and \cref{eq:aluno} becomes
\be\label{eq:aldue}
\dots=\frac{s_d}{2}\int_M\delta_F^\a(p)dM(p).
\ee
It remains only to show that $\frac{s_d}{2}\delta_F^\a=\delta_X^\a$. 
\subsubsection{The constant doesn't matter}
Once again, we don't need to keep track of the constants as long as they depend only on $k$ and $m$. Indeed, we can argue as in the proof of \cref{lem:Montecarlo} and observe that if the identity
\be\label{eq:alphacmk}
\EE \int_Z \a(X,p)=c(m,k)\int_M\delta^\a_X(p)dM(p)
\ee
holds under the hypotheses of \cref{thm:Alphaca}, then we can check the constant in the case when $M=S^m$, $\a=1$ and $X=(X^1,\dots,X^k)$ is such that $X^i(p)=\langle \gamma_i,p\rangle$ for a family of $k$ i.i.d. standard Gaussian vectors $\gamma_i\sim N(0,\mathbbm{1}_{m+1})$. Such random field is invariant under orthogonal transformations, therefore $\delta_X^1$ is a constant, hence we can compute it at $p=e_0$ the first vector of the canonical basis of $\R^{m+1}$. Since, in this case, $Z$ is almost surely a unit sphere of dimension $d$, we obtain the identity
\be 
s_d=c(m,k)s_m\delta_X^1(e_0)=c(m,k)s_m\EE\kop|J_0\begin{pmatrix}
\gamma_1 & \dots & \gamma_k
\end{pmatrix}|\Big| \gamma_i^0=0\pok\frac{1}{(2\pi)^{\frac k2}},
\ee
from which we deduce, using \cref{lem:hate} below, that
\be 
c(m,k)^{-1}=\frac{s_m}{s_d}\EE\kop\|\xi_1\wedge\cdots\wedge\xi_k\|\pok\frac{1}{(2\pi)^{\frac k2}}=1
\ee
where $\xi_1,\ldots,\xi_k\sim N(0,\mathbbm{1}_m)$ are i.i.d.
\begin{lemma}\label{lem:hate}
Let $\xi_1,\ldots,\xi_k\randin\R^m$ be i.i.d. standard Gaussian vectors. We have:
\be 
\EE\|\xi_1\wedge\cdots\wedge\xi_k\|=\frac{m!\uball_m}{(2\pi)^{\frac{k}{2}}(m-k)!\uball_{m-k}}=(2\pi)^{\frac k2}\frac{s_{m-k}}{s_m}
\ee
\end{lemma}
\begin{proof}
We will prove the lemma using zonoid calculus, as discussed in \cref{sec:zonoids}. First, by \cref{eg:gausszonisball}, we have that $\EE\seg{\xi_i}=(2\pi)^{-\tfrac12} B^m$ for all $i=1,\ldots,m$. It follows then from \cref{def:zlength} that 
\bega\label{eqthiseq2}
\EE\|\xi_1\wedge\cdots\wedge\xi_k\|=(2\pi)^{-\frac k2}\ell\tyu(B^m)^{\wedge k}\uyt=\dots.
\eega
Observe that $B^m$ is a Grassmannian zonoid, hence, by using first  \cref{lem:lengthwithballs} and then \cref{prop:kintvolpow}, \cref{eqthiseq2} becomes
\bega 
\dots=(2\pi)^{-\frac k2}\frac{1}{(m-k)!\uball_{m-k}}\ell\tyu(B^m)^{\wedge m}\uyt=(2\pi)^{-\frac k2}\frac{1}{(m-k)!\uball_{m-k}}m!\uball_m
\eega
which gives the first equality we wanted. The second follows from the identity $d!\uball_d=(2\pi)^ds_d$.
\end{proof}

\begin{remark} \cref{prop:kintvolpow} implies that in the setting of \cref{lem:hate} above we have
\be\label{eq:plugin} 
\EE\|\xi_1\wedge\cdots\wedge\xi_k\|=(2\pi)^{-\frac k2}k!\Vint_k(B^m).
\ee
\end{remark}
\subsubsection{Computing the density}
In virtue of the identities \eqref{eq:aluno} and \eqref{eq:aldue}, to prove the identity \eqref{eq:alphacmk}, it is sufficient to show that 
\be 
\delta^\a_F(p)=c(m,k)\delta^\a_X(p),
\ee
for some constant $c(m,k)$ depending only on $m$ and $k$. (Since we already showed that the constant doesn't matter, we will keep calling it with the same letter $c(m,k)$ even though its value changes from line to line.) 
Since $X$ and $Y$ are independent, we have that $\rho_{F(p)}(0)=\rho_{X(p)}(0)\rho_{Y(p)}(0)=c(m,k)\rho_{X(p)}(0)$. Moreover, observe that $d_pY$ and $Y(p)$ are independent. Therefore
\bega \label{eq:thiseq3}
\delta^\a_F(p)&
=\EE\kop\a(X,p)J_pF|F(p)=0\pok\rho_{F(p)}(0)
\\
&=
c(m,k)\EE\kop\a(X,p)\|d_pX^1\wedge \dots \wedge d_pX^k\wedge d_pY^1\wedge \dots \wedge d_pY^d \|\Big| X(p)=0\pok\rho_{X(p)}(0)
=\dots
\eega 
Recall that taking coordinates with respect to an orthonormal basis of $T^*_pM$, we have that $d_pY^1,\dots,d_pY^d$ become i.i.d. standard Gaussian vectors in $\R^m$, so that, by integrating first with respect to $Y$ and using \cref{lem:hate2} below, we obtain that \cref{eq:thiseq3} becomes
\bega
\dots&=c(m,k)\EE_{X}\kop\a(X,p)\EE_Y\kop\|d_pX^1\wedge \dots \wedge d_pX^k\wedge d_pY^1\wedge \dots \wedge d_pY^d\|\pok\Big| X(p)=0\pok\rho_{X(p)}(0)
\\
&=
c(m,k)\EE\kop\a(X,p)\|d_pX^1\wedge \dots \wedge d_pX^k\|\Big| X(p)=0\pok\rho_{X(p)}(0)
=
\delta^\a_X(p)
\eega 
which is what we wanted.
\begin{lemma}\label{lem:hate2}
Let $\xi_1,\dots,\xi_d\randin \R^m$ be i.i.d. standard Gaussian vectors and let $v_1,\dots,v_k\in\R^m$. Then there exists a constant $c(m,k)>0$ s.t.
\be 
\EE\|v_1\wedge \dots \wedge v_k\wedge \xi_1\wedge \dots \wedge \xi_d\|=c(m,k)\|v_1\wedge \dots \wedge v_k\|
\ee
\end{lemma}
\begin{proof}
Let $e_1,\dots, e_m$ be an orthonormal basis. We can assume that $v_1,\dots,v_k$ belong to the space generated by $e_1,\dots,e_k$. Let us denote by $\pi\colon \R^m\to \R^m$, the orthogonal projection onto the space spanned by $e_{k+1},\dots,e_{m}$. Then 
\bega 
\EE\|v_1\wedge \dots \wedge v_k\wedge \xi_1\wedge \dots \wedge \xi_d\|
&=
\EE\|v_1\wedge \dots \wedge v_k\wedge \pi(\xi_1)\wedge \dots \wedge \pi(\xi_d)\|
\\
&=
\|v_1\wedge \dots \wedge v_k\|\cdot \EE\| \pi(\xi_1)\wedge \dots \wedge \pi(\xi_d)\|.
\eega
This concludes the proof of the Lemma, because $\pi(\xi_i)$ are now independent standard Gaussian vectors in a space of dimension $m-k$.
\end{proof}

\section{Main results}\label{sec:mainres}

\subsection{The density of expected volume}
Taking $\a=1$ in \cref{thm:Alphaca}, we obtain the formula for the expected volume of a random submanifold $Z=X^{-1}(0)$.
In this case, abusing notation, we write
\begin{equation}
    \delta_Z(p):=\delta_X(p):=\delta^1_X(p),
\end{equation}
where $\delta^1_X(p)$ is defined by \eqref{eq:alphadensity}, with $\a\equiv 1$. 

\begin{theorem}[Expected volume]\label{thm:Evol}
Let $k\le m\in \N$. Let $(M,g)$ be a Riemannian manifold of dimension $m$. Let $X\colon M\randto \R^k$ be a {\zkrok} random field and define the random submanifold $Z:=X^{-1}(0)$. Let $A\subset M$ be a Borel subset. Then 
\begin{equation}\label{eq:maindensvol}
    \delta_Z(p)=\ell\tyu\zeta_X(p)\uyt
\end{equation}
 and thus
\be 
\EE\kop\vol_d(Z\cap A)\pok=\int_A\ell\tyu\zeta_X(p)\uyt dM(p).
\ee
\end{theorem}
\begin{proof}
By \cref{thm:Alphaca}, we have that 
\begin{equation}
    \delta_Z(p)=\EE\kop \|d_pX^1\wedge \dots \wedge d_pX^k\|\big|X(p)=0\pok\rho_{X(p)}(0)
\end{equation}
is the density of the measure $A\mapsto \EE\{\vol_d(Z\cap A)\}$.
By definition of the length (\cref{def:zlength}) and of the zonoid section (\cref{def:zonoidsec}), this is precisely equal to $\ell\tyu\zeta_X(p)\uyt$, which is what we wanted.
\end{proof}
Notice that, since $J_pX=\|d_pX^1\wedge \dots \wedge d_pX^k\|$, \eqref{eq:maindensvol} is the first of the two identities in \eqref{eq:krmain}.

Let us use the convention that $\vol_n(\emptyset):=0$ for all $n\in \mathbb{Z}$ and $\vol_n(Z)=+\infty$ if $Z\neq\emptyset$ and $n<0$.
Using the expression for independent intersection described in \cref{thm:wedge} we find the following.
\begin{corollary}\label{cor:Evolvol}
Let $X_1\randin C^1(M,\R^{k_1}),\ldots,X_n\randin C^1(M,\R^{k_n})$ be independent \zkrok fields, write $k:=k_1+\cdots+k_n$ and let $Z_i:=(X_i)^{-1}(0)$, $i=1,\ldots,n$. Then we have, for all $p\in M$,
\be\label{eq:1megacor} 
\delta_{Z_1\cap\cdots\cap Z_n}(p)=\ell(\zeta_{X_1}(p)\wedge\ldots\wedge \zeta_{X_n}(p)).
\ee
In other words, for all $U\subset M$ measurable we have
\be\label{eq:2megacor}
 \EE\vol_{m-k}\left(Z_1\cap\cdots\cap Z_n\cap U\right)=\int_U \ell(\zeta_{X_1}(p)\wedge\ldots \wedge\zeta_{X_n}(p)) dM (p)
\ee
In the case where $k_i=1$ for all $i=1,\dots, n$ and were $n=m=\dim M$, we have
\begin{equation}\label{eq:3megacor}
    \EE\#\left(Z_1\cap\cdots\cap Z_m\cap U\right)=m!\int_U \MV(\zeta_{X_1}(p),\ldots,\zeta_{X_m}(p)) dM (p),
\end{equation}
where $\MV$ denotes the mixed volume, see \cref{sec:MV}.
 In the case in which $k_i=1$ for all $i=1,\dots,n$ and all the fields are identically distributed, we have
\begin{equation}\label{eq:4megacor}
    \EE\vol_{m-n}\left(Z_1\cap\cdots\cap Z_n\cap U\right)=n!\int_U  \mathcal{V}_n(\zeta_{X_1}(p)) dM (p),
\end{equation}
where we recall that $\mathcal{V}_n$ denotes the $n^{th}$ intrinsic volume defined in \cref{eq:defvint};
if, in addition, $n=m=\dim M$, then
\begin{equation}\label{eq:5megacor}
    \EE\#\left(Z_1\cap\cdots\cap Z_m\cap U\right)=m!\int_U  \vol_m(\zeta_{X_1}(p)) dM (p).
\end{equation}
\end{corollary}

\begin{proof}
As we mentioned above, \eqref{eq:2megacor} follows by combining \cref{thm:Evol} with \cref{thm:2}.
    In the case where $k_i=1$ for all $i=1,\ldots,n$,  and where $n=m=\dim M$, we have $k=n$, so that by \cref{prop:MVandwedge}  \eqref{eq:2megacor} specializes to \eqref{eq:3megacor}.
If all the fields are identically distributed and scalar: $k_1=\dots=k_n=1$,  then their zonoid sections coincide and thus \eqref{eq:2megacor} becomes \eqref{eq:4megacor} by \cref{prop:kintvolpow}.
Finally, if $n=m$ we obtain \eqref{eq:5megacor} as a special case of either \eqref{eq:3megacor} or \eqref{eq:4megacor}.
\end{proof}

\subsection{Alexandrov-Fenchel and Brunn-Minkowsky inequalities for random submanifolds}\label{sub:AFBM}
Applying the inequalities (AF) and (BM) (\cref{prop:BMi,prop:AFi}) we obtain lower bounds for the densities.

\begin{mainthm}[KRAF]\label{cor:KRAF}
Let $Y_1,\ldots,Y_{m-2},X_1,X_1',X_2,X_2'\randin C^1(M,\R)$ be independent {\zkrok} fields, such that $X_1'\sim X_1$ and $X_2'\sim X_2$. Let  $\mathfrak{Z}:=(Y_1)^{-1}(0)\cap\ldots\cap(Y_{m-2})^{-1}(0)$, $Z_i:=(X_i)^{-1}(0)$ and $Z_i':=(X_i')^{-1}(0)$. Then we have for all $p\in M$
\begin{equation}
    \delta_{ Z_1\cap Z_2\cap \mathfrak{Z}}(p)\geq \sqrt{\delta_{ Z_1\cap Z'_1\cap \mathfrak{Z}}(p)\cdot \delta_{ Z_2\cap Z_2'\cap \mathfrak{Z}}(p)}.
\end{equation}
\end{mainthm}

\begin{remark}
Note that \cref{cor:KRAF} is an inequality on the densities and not directly on the number of points of intersection. In fact, by Hölder's inequality, we have that 
\begin{equation}
    \sqrt{\EE\#\left(Z_1\cap Z_1'\right)\cdot\EE\#\left(Z_2\cap Z_2'\right)}\geq \int_M \sqrt{\delta_{ Z_1\cap Z'_1\cap \mathfrak{Z}}(p)\cdot \delta_{ Z_2\cap Z_2'\cap \mathfrak{Z}}(p)}.
\end{equation}
\end{remark}

\begin{mainthm}[KRBM]\label{cor:KRBM}
Let $X_0, X_1\randin C^1(M,\R)$ be {\zkrok} fields, let $\epsilon\randin\{0,1\}$ be a Bernoulli of parameter $0\leq t\leq 1$ independents of $X_0,X_1$, i.e. $\epsilon=0$ with probability $(1-t)$and $1$ with probability $t$. Let $X_t:=(1-\epsilon_i) X_0+\epsilon X_1$ be {\zkrok}\footnote{For instance, this is true if the condition $(*)$ of \cref{prop:bern} holds.}. Finally, let $Z_1^{(i)},\ldots,Z_m^{(i)}$ be i.i.d. copies of $(X_i)^{-1}(0)$, $i=0,1,t$. We have for all $p\in M:$
\begin{equation}
    \delta_{Z_1^{(t)}\cap\cdots\cap Z_m^{(t)}}(p)\geq \left(\delta_{Z_1^{(0)}\cap\cdots\cap Z_m^{(0)}}(p)\right)^{(1-t)}\left( \delta_{Z_1^{(1)}\cap\cdots\cap Z_m^{(1)}}(p)\right)^t.
\end{equation}
\end{mainthm}

\subsection{The expected current}
Assume that $M$ is oriented. Then a {\zkrok} field $X\colon M\randto \R^k$ defines a random $(m-k)$-current, by integration over the random (co-oriented and thus oriented, see \cref{def:orientZ}) submanifold $Z=X^{-1}(0)$:
\be
\int_Z\colon \Omega_c^{(m-k)}(M)\to \R
\ee
where recall that $\Omega_c^{(m-k)}(M)$ is the space of smooth differential forms of degree $m-k$ with compact support.
\begin{definition}\label{def:orientZ}
The orientation of $Z=X^{-1}(0)$ is defined by declaring that if  $\lambda\in \Lambda^{m-k}T^*_pM$ is such that $\lambda\wedge d_pX^1\wedge\dots\wedge d_pX^k>0$, then $\lambda|_Z>0$.
\end{definition}

In this subsection, we will prove that the expectation of this random current is the current represented by the continuous $k$-form $e_X\in\Gamma(\Lambda^{k}T^*M)\subset \Omega_c^{m-k}(M)^*$, which is the \emph{nigiro} (see \cref{def:nigiro}) of the zonoid section:
\be 
e_X(p)=\EE\kop d_pX^1\wedge \dots \wedge d_pX^k\big|X(p)=0\pok\rho_{X(p)}(0)=e(\zeta_X).
\ee
\begin{proposition}\label{prop:Econt}
$e_X$ is a continuous $k$-form: $e_X\in \Gamma(\Lambda^kT^*M)$.
\end{proposition}
\begin{proof}
Given a zonoid $\zeta$ in a fixed vector space $V$, its nigiro  $e(\zeta)$ can be expressed as
\be 
e(\zeta)=\sum_{i=1}^m \frac{h_{\zeta}(v_i)-h_{\zeta}(-v_i)}{2} v^i,
\ee
where $v_1,\dots,v_m$ is a basis of $V$ and $v^1,\dots,v^m$ is the dual basis.
Indeed, one can check that this formula is true for segments and is linear and continuous in $h_{\zeta}$.
Hence, $e(\zeta)$ depends continuously on the support function $h_{\zeta}$. 
Thus, the thesis follows from \cref{prop:hcontinuity}.
\end{proof}
\begin{theorem}[Expected current]\label{thm:Ecurrent}
Let $k\le m\in \N$. Let $(M,g)$ be an oriented Riemannian manifold of dimension $m$. Let $X\colon M\randto \R^k$ be a {\zkrok} random field and consider the random submanifold $Z:=X^{-1}(0)$, oriented according to \cref{def:orientZ}. Let $\omega\in\Omega_c^{(m-k)}(M)$. Then
\be 
\EE\kop\int_Z \omega|_Z\pok=\int_M  \omega \wedge e_X.
\ee
\end{theorem}
\begin{proof}
Let $d=m-k$. Let us define $\a\colon \mC^1(M,\R^k)\times M\to\R$ as follows: if $f(p)\neq 0$ or if $p$ is a critical point of $f$, then $\a(f,p)=0$; otherwise we set
\be\label{eq:alphacurrent} 
\a(f,p):=\langle\omega(p), e_1\wedge\dots\wedge e_d\rangle,
\ee
where $e_1,\dots,e_d$ is a positive orthonormal basis of $T_p(f^{-1}(0))=\ker d_pf$. 
Let $\Omega_M$ be the positive volume m-form of $M$, so that $\int_Mh\Omega_M=\int_M hdM$, for any integrable function $h\colon M\to \R$. An equivalent expression defining $\a$ is:
\be\label{eq:alphacurrent2}
\a(f,p)J_pf\Omega_M(p)=\omega(p)\wedge  d_pf^1\wedge\dots \wedge d_pf^k.
\ee
We conclude by applying \cref{thm:Alphaca} as follows.
\bega
\EE\kop\int_Z\omega|_{Z}\pok
&= 
\EE\kop\int_Z\a(X,p)dZ(p)\pok
\\
&=
\int_M\EE\kop \a(X,p)J_pX \Big| X(p)=0\pok \rho_{X(p)}(0)\Omega_M(p)
\\
&=
\int_M\EE\kop \omega\wedge dX^1\wedge \dots \wedge dX^k\big|X(p)=0\pok\rho_{X(p)}(0)
\\
&=
\int_M \omega\wedge e_X.
\eega
\end{proof}
Together, \cref{thm:Evol} and \cref{thm:Ecurrent} form the statement of \cref{thm:1}, whose proof is thus now complete.
\subsection{What does the Zonoid section know?}\label{sub:whatdoesitknow}
We have seen two cases of the Alpha formula (\cref{thm:Alphaca}) where the density $\delta^\alpha_X$ was a function of the zonoid section $\zeta_X$. 

We can ask what are the conditions on the function $\alpha$ for this to be the case.

\begin{proposition}\label{prop:whatknows}
    Let $\alpha:C^1(M,\R^k)\times M\to \R$ be a measurable function that is given for every $(f,p)\in C^1(M,\R^k)\times M$ by $0$ if $J_p\varphi=0$ and else by: 
    \begin{equation}
        \alpha(f,p)= (J_pf)^{-1}T(\dd_pf^1\wedge \cdots\wedge \dd_pf^k)+(J_pf)^{-1}F(\dd_pf^1\wedge \cdots\wedge \dd_pf^k)
    \end{equation}
    where $T:\Lambda^kT^*M\to \R$ is linear on the fibers and $F:\Lambda^kT^*M\to \R$ is positively homogeneous on the fibers. Then for every \zkrok field $X\randin C^1(M,\R^k)$ and every $p\in M$, the density $\delta_X^\alpha(p)$ is a function of the zonoid $\zeta_X(p)$.
\end{proposition}
\begin{proof}
    Let $X\randin C^1(M,\R^k)$ be \zkrok and let $p\in M$. By definition, see \cref{eq:alphadensity}, the density is given by 
    \begin{align}
        \delta_X^\alpha(p)=\rho_{X(p)}(0)\, &\EE\left[T\left(\dd_pX^1\wedge \cdots\wedge \dd_pX^k\right)|X(p)=0\right]   \\
        &+\rho_{X(p)}(0)\, \EE\left[F\left(\dd_pX^1\wedge \cdots\wedge \dd_pX^k\right)|X(p)=0\right].
    \end{align}
    The first summand gives
    \begin{align}
        \rho_{X(p)}(0)\EE\left[T(\dd_pX^1\wedge \cdots\wedge \dd_pX^k)|X(p)=0\right]&=T\left(\rho_{X(p)}(0)\EE\left[\dd_pX^1\wedge \cdots\wedge \dd_pX^k|X(p)=0\right]\right)  \\
        &=T(e_X(p)).
    \end{align}
    For the second term, if we call $Y:=\rho_{X(p)}(0)(\dd_pX^1\wedge \cdots\wedge \dd_pX^k|X(p)=0)$ then we have tautologically 
    \begin{equation}
        \rho_{X(p)}(0)\EE\left[F(\dd_pX^1\wedge \cdots\wedge \dd_pX^k)|X(p)=0\right]=\EE\left[F(Y)\right]
    \end{equation}
    But since $F$ is positively homogeneous, by \cref{prop:zoneq}, this does not depend on the random vector $Y$ but this is a function of the zonoid $\EE\seg{Y}=\seg{\zeta_X(p)}$ which is the centered version of $\zeta_X(p)$ (see \cref{def:nigiro}) and this concludes the proof.
\end{proof}
\begin{remark}
In particular, the above proof shows that if $F\equiv 0$, then $\delta^\a=T(e(\zeta_X))$, while if $T\equiv 0$, then $\delta^\a$ depends on $\zeta_X(p)$ only up to translations, i.e., on $\underline{\zeta_X(p)}$ (see \cref{def:nigiro}).
\end{remark}
In the case of the density of expected volume ( \cref{thm:Evol}) we have that $T\equiv 0$ and $F=\|\cdot\|$ is the norm (given by the Riemannian structure). 

In the case of the expected current (\cref{thm:Ecurrent}) we see from Equation \eqref{eq:alphacurrent2} that $\a$ is given pointwise by a linear function evaluated on $(J_pf)^{-1}(\dd_pf^1\wedge \cdots\wedge \dd_pf^k)$. 
Since $J_pf=\|\dd_pf^1\wedge \cdots\wedge \dd_pf^k\|$, the latter is a unit simple vector. 
Let us consider the bundle $G_+(k,T^*M)\to M$ whose fiber over $p\in M$ is the Grassmannian of oriented $k$-dimensional vector subspaces of $T^*_pM$. The set of unit simple vector in $\Lambda^k T^*M$ is identified with  $G_+(k,T^*M)$ via the Pl\"ucker embedding:
\begin{equation}
\Pi\colon G_+(k,T^*_pM)\xrightarrow{\sim} \kop v_1\wedge\dots\wedge v_k\in \Lambda^k T^*_pM|\, \|v_1\wedge\dots\wedge v_k\|=1\pok,
\end{equation}
\be 
(V,[v_1\wedge \dots \wedge v_k])\mapsto \frac{v_1\wedge\dots\wedge v_k}{\|v_1\wedge\dots\wedge v_k\|}
\ee
where $[v_1\wedge \dots \wedge v_k]$ denotes the orientation of $V$ induced by the basis $v_1\wedge \dots \wedge v_k$.
We recall that, by \cref{lem:GZchar}.(iii), we have that a centered Grassmannian zonoid $K$ in $\Lambda^kT^*_pM$ is associated, via a one to one correspondence, with a positive measure $\mu_K$ on $G(k,T_pM)$, given by \eqref{eq:costrans}.

Let us call \emph{linear} those functions $\theta_T:G_+(k,T^*M)\to \R$ such that if $v_1,\dots,v_k$ is an orthonormal basis of $V\subset T^*_pM$ then
\be 
\theta_T(V,[v_1\wedge \dots \wedge v_k])=T\tyu v_1\wedge \dots \wedge v_k\uyt
\ee
for some linear function $T:\Lambda^k T^*M\to \R$. Then, we can rewrite \cref{prop:whatknows} in the following way. 
\begin{proposition}\label{prop:introwhatknows}
    Let $\theta_T:G_+(k,T^*M)\to \R$ be a linear function and let $F:G(k,TM)\to \R$ be measurable. Then for every \zkrok random field $X:M\to\R^k$, we have 
    \be\label{eq:introwhatknows} 
    \EE\kop\int_Z \theta_T\tyu N_pZ, [\dd_pX^1,\dots , \dd_pX^k]\uyt
    +F\tyu N_pZ\uyt dZ\pok=\int_M \tyu T(e(\zeta_X))+\delta^F \uyt dM,
    \ee
    where $\delta^F:M\to \R$ is a function whose value at any $p\in M$ depends only on $F$ and on  $K:=\underline{\zeta_X(p)}$, and is given by
    \be\label{eq:deltasexy} 
    \delta^F(p)=\int_{G(k,T_pM)}Fd\mu_{K}.
    \ee
\end{proposition}
\begin{proof}
 The only thing that does not directly derive from \cref{prop:whatknows} is the formula \eqref{eq:deltasexy} for $\delta^F$. By \cref{thm:Alphaca}, $\delta^F$ is given by
 \be 
 \delta^F(p)=\rho_{X(p)}(0)\, \EE\left[F\left(\dd_pX^1\wedge \dots \wedge \dd_pX^k\right)|X(p)=0\right],
 \ee
 where we still denote by $F$ the even and homogeneous extension to the cone of simple vectors in $\Lambda^kTM$.
 \eqref{eq:deltasexy} now follows from \eqref{eq:intwrtgenmeasgrass}.
 
\end{proof}
\subsubsection{The zonoid section as a varifold}\label{subsub:varifold}
Let $\Gamma(\ZZ_0(\Lambda^kT^*M))$ denote the subspace of the space of zonoid sections $\Gamma(\ZZ(\Lambda^kT^*M))$, as defined in \cref{sub:zonoidbundles}, consisting of centered ones.
For instance, the centered zonoid section $ \seg{\zeta_X}$ of a \zkrok field $X\randin \mC^1(M,\R^k)$ is an element of $\Gamma(\ZZ_0(\Lambda^kT^*M))$.

 Let  $\zeta\in \Gamma(\ZZ_0(\Lambda^kT^*M))$. Following the discussion preceding \cref{prop:introwhatknows}, for every $p\in M$, the centered zonoid $\zeta(p)$  has an associated measure $\mu_{\zeta(p)}$ on the Grassmannian $G(k,T_p^*M)$,  which we identify with $G(m-k,T_pM)$. Recall that $h_K\colon \Lambda^kT_pM\to \R$ denotes the support function of the zonoid $K\subset \Lambda^kT_p^*M$, in particular, $h_{[0,v_1\wedge\dots \wedge v_k]}(x)=|\langle x,v_1\wedge\dots \wedge v_k \rangle|,$ for $v_1,\dots, v_n\in T_pM$ and $x\in \Lambda^kT_pM$. Because of \cref{lem:GZchar}.(iii), the measure $\mu_{\zeta(p)}$ is defined as the unique measure on  $G(m-k,T_pM)$ such that:
\bega  
h_{\zeta(p)}(x)&=\int_{G(m-k,T_pM)}h_{[0,v_1\wedge\dots \wedge v_k]}(x)d\mu_{\zeta(p)}(V)
\\
&=
\int_{G(m-k,T_pM)}\| v_{k+1}\wedge\dots \wedge v_{m}\wedge x\|d\mu_{\zeta(p)}(V)
, \quad \text{ for all $x\in \Lambda^kT_pM,$}
\eega
where for any $V$, we have chosen an orthonormal basis $v_1,\dots, v_m$ of $T_pM$ such that $v_1,\dots,v_k$ is a basis of $V^\perp$. 
Let $d=m-k$. We can put together such family of measures, to define a \emph{$d$-varifold} on $M$, that is, a positive measure on the total space of the Grassmann bundle $G(d,TM),$ see \cite{Allard1972Varifold}.
\begin{definition}
Let $\zeta\in \Gamma(\ZZ_0(\Lambda^kT^*M)).$ and let $d=m-k$. We define the $d$-varifold $V_{\zeta}$ as the positive measure on $G(d,TM)$ such that
\be\label{eq:defimuX}
V_{\zeta}(A):=\int_M\mu_{\zeta(p)}\tyu A\cap G(d,T_pM)\uyt dM(p).
\ee
\end{definition}
Notice that the underlying measure on $M$, usually denoted $\|V_{\zeta}\|$ in varifold theory (see \cite{Allard1972Varifold}),  is absolutely continuous with density $\ell(\zeta)$, see \cref{prop:fromrvtomeas}. and that, since $\zeta$ is continuous, $V_\zeta$ is always a Radon measure. 
\begin{lemma}\label{lem:zonoidvarifold}
The function $\zeta\mapsto V_\zeta,$ defined for all $\zeta\in\Gamma(\ZZ_0(\Lambda^kT^*M)),$  is injective.
\end{lemma} 
\begin{proof}\eqref{eq:defimuX} determines $\mu_{\zeta(p)}$, and thus $\zeta(p)$, for almost every $p\in M$; by the continuity of the latter, this determines  $\zeta$. 
\end{proof}
There exists another natural way of constructing a $d$-varifold. Let $Z\subset M$ be $\mC^1$ submanifold of dimension $d$, then $TZ\subset G(d,TM)$ is a subset of the Grassmann bundle\footnote{If $Z$ is of class $\mC^2$, then $TZ\subset G(d,TM)$ is a $\mC^1$ submanifold of dimension $d.$}
\begin{definition}
Let Let $Z\subset M$ be $\mC^1$ submanifold of dimension $d$. We define the $d$-varifold $V_{Z}$ as the positive measure on $G(d,TM),$ supported on $TZ,$ such that
\be
V_{Z}(A):=\int_M 1_A(T_pZ)dM(p).
\ee
\end{definition}
The reason why we talk about varifolds is that they are the proper language to understand \cref{thm:E} and also the title of the paper. Indeed, we have the following theorem.
\begin{mainthm}\label{thm:title}
Let $X\randin \mC^1(M,\R^k)$ be a \zkrok random field, and let $d=m-k$ be the dimension of the random submanifold $Z:=X^{-1}(0).$ Then
\be 
\EE V_Z=V_{\zeta_X}
\ee
\end{mainthm}
\begin{proof}
Let $F\colon G(d,TM)\cong G(k,TM)\to \R$ be a bounded continuous function. Then  \cref{prop:introwhatknows} yields the thesis as follows:
\bega 
\EE\kop V_Z(F)\pok &=\EE\kop \int_M F(T_pZ)dM(p)\pok
\\
&= \int_M \delta^F(p) dM(p)
\\
&= \int_M \int_{G(d,T_pM)}Fd\mu_{\zeta(p)} dM(p)=V_{\zeta}(F)
\eega
\end{proof}
\subsubsection{The zonoid section does not know the random field}
The previous observations, combined with \cref{prop:introwhatknows}, yields that the zonoid section depends only on the law of the zero set $X^{-1}(0)$. In more technical terms, we have the following.

\begin{proposition}\label{prop:onlydeponrandsubm}
    Let $X_1,X_2\randin\mC^1(M,\R^k)$ be \zkrok random fields and let $Z_i=X_i^{-1}(0)$, for $i=1,2.$ 
    Assume that
    \be\label{eq:P1ZZA} 
    \PP\kop Z_1\in W\pok=\PP\kop Z_2\in W\pok
    \ee 
    for any family $W$ of submanifolds of $M$ such that the set $\kop f\in\Omega:f^{-1}(0)\in W\pok$ is Borel  in $\mC^1(M,\R^k)$, where $\Omega\subset \mC^1(M,\R^k)$ is the subset of functions for which $0$ is a regular value\footnote{If $M$ is compact, $\Omega$ is open. In general, it can be expressed as a countable intersection of open and dense sets, thus it is always a Borel set, see \cite{Hirsch}.}. Then $\seg{\zeta_{X_1}}=\seg{\zeta_{X_2}}$.
\end{proposition}

\begin{proof}
    Since $X_1$ and $X_2$ satisfy \zkrok-\ref{itm:krok2}, we consider them as random elements of $\Omega$.
    For a family $W$ of submanifolds of $M$, we write $A_W:=\kop f\in\Omega:f^{-1}(0)\in W\pok\subset \Omega$. Let $\mathcal{A}$ be the $\sigma$-algebra on $\Omega$ consisting of all Borel subsets of the form $A_W$. By definition, $\mathcal{A}$ is contained in the Borel $\sigma$-algebra of $\Omega$ and a Borel function is measurable for $\mathcal{A}$ if and only if it depends only on the zero set. In particular, for any $F\colon G(k,TM)\to \R$ measurable, the function $I_F:f\mapsto \int_{f^{-1}(0)}F$ is measurable with respect to $\mathcal{A}$.
    
    Let us now consider the probability measure $\PP_1$, respectively $\PP_2$, on the measurable space $(\Omega,\mathcal{A})$ obtained by restricting the laws of $X_1,$ respectively $X_2$, to the $\sigma$-algebra $\mathcal{A}$, respectively. By hypothesis we have that 
    $\PP_1=\PP_2.$ Therefore
    \be\label{eq:EIF}
    \EE\kop I_F(X_1)\pok=\EE_1\kop I_F\pok=\EE_2\kop I_F\pok=\EE\kop I_F(X_2)\pok.
    \ee
    where $\EE_i$ denotes the integral with respect to the measure $\PP_i,$ for $i=1,2.$  \cref{prop:introwhatknows} implies that if \eqref{eq:EIF} holds for every $F,$ then $\mu_{X_1}=\mu_{X_2}$ and hence $\seg{\zeta_{X_1}(p)}=\seg{\zeta_{X_2}(p)},$ which is what we wanted.
    
\end{proof}
The nigiro $e(\zeta_X)$ of the zonoid section does not depend only on the law of the random submanifold $Z=X^{-1}(0)$, but also on the orientation of its normal bundle $NZ$ induced by the isomorphism given by $d_pX:T_pN\to \R^k$, for all $p\in M$.

A pair $(Z,o)$, where $Z$ is a submanifold (of $M$) and $o$ is an orientation of $NZ$ is called a \emph{cooriented} submanifold (of $M$). By considering also the case $F=0$ in \cref{prop:introwhatknows} and reasoning as in the proof of \cref{prop:onlydeponrandsubm} above, we get the following.
\begin{proposition}\label{prop:onlydeponrandsubm2}
    Let $X_1,X_2\randin\mC^1(M,\R^k)$ be \zkrok random fields and let $Z_i=X_i^{-1}(0)$, for $i=1,2.$ Let us denote by $o_{X_i}$ the orientation of $NZ_i$ induced by $dX_i$, for $i=1,2.$ 
    Assume that
    \be\label{eq:P1ZZA2} 
    \PP\kop (Z_1,o_{X_1})\in W\pok=\PP\kop (Z_2,o_{X_2})\in W\pok
    \ee 
    for any family $W$ of cooriented submanifolds of $M$ such that the set $\kop f\in\Omega:(f^{-1}(0),o_f)\in W\pok$ is Borel in $\mC^1(M,\R^k)$. Then $\zeta_{X_1}=\zeta_{X_2}$.
\end{proposition}

\section{Vector bundles}\label{sec:vectbun}
The results of the previous section can be extended to the setting of random sections of vector bundles. 
\begin{definition}\label{de:loczkrok}
Let $\pi:E\to M$ be a smooth vector bundles of rank $k$ and let $X\randin \mC^1(M|E)$ be a random section. We say that $X$ is \emph{locally \zkrok} if for every open set $U\subset M$ on which there is a trivialization $E|_U\cong U\times \R^k$, the local random field $X|_U\randin \mC^1(U,\R^k)$ is \zkrok.
\end{definition}
We denote the zero section of the vector bundle $E\to M$ by $0_M\subset E$. By applying \cref{thm:Alphaca} locally we get the following.

\begin{theorem}[Alpha Formula for vector bundles]\label{thm:AlphacaVB}
Let $k\le m\in \N$. Let $(M,g)$ be a Riemannian manifold of dimension $m$. Let $E\to M$ be a $\mC^1$ real vector bundle of rank $k$, endowed with a metric. Let $\nabla$ be any connection on $E$. Let $X\colon M\randto E$ be a locally {\zkrok} random section and define the random submanifold $Z:=X^{-1}(0_M)$. Let $\a\colon \mC^1(M|E)\times M\to \R$ be a Borel measurable function. Then
\be 
\EE\kop\int_{ Z}\alpha(X,p)dZ(p)\pok=\int_M\delta^\a_X(p)dM(p).
\ee
Where
\be\label{eq:alphadensitybundle}
\delta^\a_X(p)=\EE\kop \a(X,p)\left\|\frac{(\nabla X)_p^{\wedge k}}{k!}\right\| \Big| X(p)=0\pok \rho_{X(p)}(0),
\ee
\end{theorem}
\begin{remark}
The value of $(\nabla X)_p$ at a point $p$ such that $X(p)=0$ doesn't depend on the choice of the connection (see also \cref{lem:stupidconditioning}). It is a linear map $(\nabla X)_p\colon T_pM\to E_p$ between two Euclidean spaces, thus it has a well defined Jacobian determinant $J(\nabla X)_p=:J_pX$, which we wrote in a more fancy way, using the language of double forms, for which we refer to \cite{AdlerTaylor}.  
This language defines the linear map $(\nabla X)^{\wedge k}_p\colon \Lambda^kT_pM\to \Lambda^kE_p=:\det E_p$, such that
\be\label{eq:wedgpow}
v_1\wedge\dots \wedge v_k\mapsto k! \nabla X_p(v_1)\wedge \dots \wedge \nabla X_p(v_k),
\ee  
where the codomain can be identified as $\Lambda^kE_p=e_1\wedge\dots\wedge e_k \R$, for an orthonormal basis $e_1,\dots,e_k$ of $E_p$. We interpret $(\nabla X)^{\wedge k}_p$ as an element of $\Lambda^kT^*_pM\otimes \det E$.
Thus, choosing $v_1,\dots,v_k$ to be a orthonormal basis of $\left(\ker(\nabla X)_p\right)^\perp$ we have the equality:
\be\label{eq:jacob}
    \left\|\frac{(\nabla X)^{\wedge k}_p}{k!}
    \right\|=
    \left|\det\left(\Big\langle(\nabla X)_p(v_i),e_j\Big\rangle\right)_{1\le i,j\le k}\right|
    =
    J_pX.
\ee
\end{remark}
\begin{remark}
The function $\rho_{X(p)}\colon E_p\to [0,+\infty)$ is the density of $[X(p)]$ with respect to the Euclidean metric on the fiber $E_p$. This term depends on the choice of the metric as well as the Jacobian of $X$ (see \eqref{eq:jacob}), but the product of the two does not, so that $\delta^\a_X$ is independent on the choice of a metric on $E$.
\end{remark}
\cref{def:zonoidsec} can be extended to define the \emph{zonoid section} in this setting.
\begin{definition}\label{def:zonoidsecvb}
Let $k\le m\in \N$. Let $(M,g)$ be a Riemannian manifold of dimension $m$. Let $E\to M$ be a $\mC^1$ real vector bundle of rank $k$, endowed with a metric. Let $X\randin \mC^1(M|E)$ be locally z-KRoK. The associated \emph{zonoid section} $\zeta_X\in \Gamma(\ZZ(\Lambda^k T^*M\otimes \det E))$ is defined for every $p\in M$ by 
\begin{equation}
    \zeta_X(p):=\EE\kop\qwe0,\frac{(\nabla X)^{\wedge k}_p}{k!}\ewq\Big|X(p)=0\pok\rho_{X(p)}(0).
\end{equation}
\end{definition}
We recall that an orientation of $E$ corresponds to a trivialization of $\det E$. In general, the support function of $\zeta_X$ is a continuous function $h_{\zeta_X}\colon \Lambda^k TM\otimes \det E^*\to \R$ and the nigiro $e_X=e(\zeta_X)$ is a continuous section of $\Lambda^k T^*M\otimes \det E$. Moreover, we compute the length $\ell(\zeta_X)$ and the other intrinsic volumes of $\zeta_X$ in terms of the metric on $\Lambda^kT^*M\otimes \det E$ induced by the Riemannian metric and the metric on $E$, hence they define continuous functions on $M$.

By applying locally the results of Section \ref{sec:mainres} we extend them to the setting of vector bundles. In particular, \cref{prop:hcontinuity}, \cref{thm:pullback}, \cref{thm:wedge}, \cref{thm:Evol}, \cref{prop:Econt}, \cref{thm:Ecurrent} hold with the obvious modifications of the statements. 

\begin{theorem}\label{thm:megathmvb}
Let $k\le m\in \N$. Let $(M,g)$ be a Riemannian manifold of dimension $m$. Let $E\to M$ be a $\mC^1$ real vector bundle of rank $k$, endowed with a metric $h$. Let $X\colon M\randto E$ be a locally {\zkrok} random section and define the random submanifold $Z:=X^{-1}(0_M)$.
\begin{enumerate}
\item \emph{(Pull-back property)} Let $\f\colon S\to M$ be a $\mC^1$ map such that $\f\transv Z$ almost surely. Then $X\circ \f$ is a locally {\zkrok}  random section of the pull-back bundle $\f^*E\to S$ and
\be 
\zeta_{X\circ \f}(p)=(d_p\f^*\otimes \mathrm{id}_{\det E})\zeta_X(p).
\ee
    \item 
    If $X_1,X_2$ are independent locally {\zkrok} random sections of two vector bundles $E_1,E_2$ over $M$, then $Z_1\cap Z_2$ is the zero set of $X_1\oplus X_2\colon M\randto E_1\oplus E_2$, there is a canonical identification $\det(E_1\oplus E_2)=\det E_1\otimes \det E_2$ and
    \be 
    \zeta_{X_1\oplus X_2}=\zeta_{X_1}\wedge \zeta_{X_2},
    \ee
    where this wedge operation is meant as a bilinear map $\Lambda^{k_1}T^*M\otimes \det E_1\times \Lambda^{k_2}T^*M\otimes \det E_2\to \Lambda^{k_1+k_2}T^*M\otimes \det E_1\otimes \det E_2$.
        \item For any Borel $A\subset M$ Borel set, we have
        \be  \EE\{\vol_{m-k}(Z\cap A)\}=\int_A\ell(\zeta_X)dM
        \ee.
 \item       
If $E$ and $M$ are oriented, then we identify $\det E=\R$ and $Z$ is oriented according to \cref{def:orientZ}, then we have the equality of currents:
\be 
    \EE\int_Z=\int_M\wedge e_X \in \Omega_c^{(m-k)}(M)^*.
    \ee
    \end{enumerate}
\end{theorem}
\begin{theorem}\label{thm:GBCbody}
Under the hypotheses of \cref{thm:megathmvb} and assuming that $E$ and $M$ are oriented, if moreover $e_X$ is smooth, then it is closed and the class $[e_X]\in H_{DR}^k(M)$ is the (De Rham) Euler class of the vector bundle $E$.
\end{theorem}
\begin{proof}
Observe that $d \int_Z=0$ in the sense of currents, that is, $\int_Z \omega=0$ for every $\omega$ closed. By linearity, the same holds for the current $\EE\int_Z$. If $e_X$ is smooth, point (4) above implies that then $de_X=0$. Let $\eta\in \Omega^k(M)$ be a De Rham representative of the Euler class of $E$. It is proved in \cite[Chapter 12]{botttu} that if $\omega$ is a closed form, then $\int_Z\omega|_Z=\int_M\omega\wedge \eta$ holds for all $X\transv 0_M$. By taking the expectation on both sides and using point (4) we obtain the identity:
\be 
Q([\omega], [\eta])=\int_M \omega\wedge \eta=\int_M\omega\wedge e_X=Q([\omega], [e_X]), \quad \forall [\omega]\in H_{DR}^{(m-k)}(M)
\ee
where $Q$ denotes the (De Rham) intersection form of $M$. Since the latter is nondegenerate by Poincaré duality (see \cite[Chapter 3]{botttu}) it follows that $[\eta]=[e_X]$.

\end{proof}

The latter statement can be expressed in a more general form using the language of twisted forms.
Given a real line bundle $L\to M$, a \emph{$i$-form with values in $L$} is a section of $\Lambda^iT^*M\otimes L\to M$ and the space of such objects is denoted as $\Omega^i(M,L)$. When $L$ is the orientation bundle of the manifold, that we will denote as $L_M$, the elements of $\Omega^m(M,L_M)$ are called \emph{densities} and there is a canonical integration operator $\int_M\colon \Omega^m_c(M,L_M)\to \R$, see \cite[Chapter 7]{botttu} or \cite[Appendix A]{KRStec}. Given $e\in \Omega^k(M,\det E)$ and $\omega\in \Omega^{(m-k)}(M,L_M\otimes \det E^*)$, their product $ \omega\wedge e$ can be canonically identified as a density, since $ L_M\otimes \det E^*\otimes \det E\cong L_M$ and therefore the number $\int_M \omega \wedge e$ is well defined, regardless of orientability.

On the other hand, once the vector bundle $E$ is endowed with a metric, if $X\transv 0_M$ then the orientation line bundle $L_Z$ of the submanifold $Z=X^{-1}(0_M)$ is isomorphic to $L_Z\cong L_M\otimes \det E|_Z$ (the isomorphism depends on the euclidean structure of $\det E$). Therefore, given $\omega \in \Omega^{(m-k)}(M,L_M\otimes \det E^*)$, its restriction $\omega|_Z$ can be seen as a density on $Z$ and thus the integral $\int_Z\omega|_Z$ is well defined.

\begin{corollary}\label{cor:nonorientEcurrent}
Let $k\le m\in \N$. Let $M$ be a smooth manifold of dimension $m$. Let $E\to M$ be a smooth real vector bundle of rank $k$, endowed with a metric. 
Let $X\colon M\randto E$ be a locally {\zkrok} random section and define the random submanifold $Z:=X^{-1}(0_M)$. Let $\omega\in\Omega_c^{(m-k)}(M,L_M\otimes \det E^*)$. Then
\be 
\EE\kop\int_Z \omega|_Z\pok=\int_M  \omega \wedge e_X.
\ee
\end{corollary}

\begin{remark}\label{rem:nonorient}
The language of twisted forms allows to define a twisted version of De Rham cohomology, see \cite{botttu}. In this sense, it is easy to see that again we have that $d e_X=0$ and $[e_X]\in H_{dR}^k(M,\det E)$ is the Euler class of the vector bundle $E$.
\end{remark}

\section{Crofton formula in Finsler manifolds}\label{sec:Finsler}

A \emph{Finsler structure} on a manifold $M$ is the choice of a norm $F_p$ on each tangent space $T_pM$ that depends continuously on the point $p\in M$. This gives a well defined notion of length of curves. Indeed, given $\gamma:[0,1]\to M$ a smooth curve, one defines
\begin{equation}\label{eq:lengthFinsler}
    \ell^F(\gamma):=\int_0^1F_{\gamma(t)}(\dot{\gamma}(t))\dd t. 
\end{equation}

The choice of a full dimensional convex body in each cotangent space induces a norm in the tangent space. Indeed, if $\zeta(p)\subset T^*_pM$ is a symmetric convex body containing the origin in its interior, then the support function $h_{\zeta(p)}:T_pM\to \R$ defines a norm. In our case, the (centered) zonoid section of a \zkrok scalar field is not always full dimensional and defines only a semi norm.
\begin{definition}
We call a \emph{semi Finsler structure} on $M$, the choice of a semi norm $F_p:T_pM\to \R$ for each $p\in M$ depending continuously on $p$. Equivalently, this is the choice of a continuous section $p\mapsto \zeta(p)\subset T^*_pM$ of centrally symmetric convex bodies containing the origin.
\end{definition}
\begin{remark}
A centrally symmetric convex body $\zeta(p)\subset T^*_pM$ is contained in a hyperplane $v^\perp$ with $v\in T_pM$ if and only if $h_{\zeta(p)}(v)=0$. For the semi Finsler structure, it means that traveling from $p$ along the direction $v$ is \emph{free} and curves that pass at $p$ tangent to $v$ have locally length zero.
\end{remark}

The zonoid section associated to a \zkrok scalar field (see \cref{def:zonoidsec}) provides then a semi Finsler structure.

\begin{definition}\label{def:finsler}
Let $X\randin C^1(M,\R)$ be a {\zkrok} field. We denote by $F^X$ the semi Finsler structure induced by $\seg{\zeta_X}(\cdot)$, where recall that $\seg{\zeta_X}(\cdot)$ is the centered zonoid of $\zeta_X(\cdot)$, i.e., for all $p\in M$ and all $v\in T_pM$
\begin{equation}\label{eqdefFX}
    F^X_p(v):=\frac{\rho_{X(p)}(0)}{2}\EE\kop \left|\dd_p X(v)\right| \Big| X(p)=0\pok.
\end{equation}
\end{definition}

Our previous results interpret in this context as follows.

\begin{proposition}[Crofton formula for curves]\label{propcroftoniftrs} Let $X\randin C^1(M,\R)$ be {\zkrok} and let $Z:=X^{-1}(0)$. Let $\gamma:[0,1]\to M$ be a smooth curve such that $\gamma\transv Z$ almost surely. Then 
\begin{equation}\label{eq:croftonformula}
   \EE\#(\gamma\cap Z)=2\,\ell^{F^X}(\gamma).
\end{equation}
\end{proposition}
\begin{proof}
Consider the random field $X\circ\gamma:[0,1]\to \R$ and apply the pull-back property \cref{thm:pullback}. By \cref{eq:pullback}, we have
\begin{equation}\label{eq:zetaXgamma}
    h_{\zeta_{X\circ\gamma}(t)}(\partial_t)=h_{\zeta_X(\gamma(t))}(\dot{\gamma}(t))
\end{equation}
Since $\zeta_{X\circ \gamma}(t)$ lives in a space of dimension $1$ (formally the tangent to $[0,1]$), its length is given by 
\begin{align}
    \ell(\zeta_{X\circ\gamma}(t)) &=h_{\zeta_{X\circ\gamma}(t)}(\partial_t)+h_{\zeta_{X\circ\gamma}(t)}(-\partial_t) \\
    &=h_{\zeta_X(\gamma(t))}(\dot{\gamma}(t))+h_{\zeta_X(\gamma(t))}(-\dot{\gamma}(t))\\
    &=2h_{\seg{\zeta_X(\gamma(t))}}(\dot{\gamma}(t))=2 F^X(\dot\gamma(t)).
\end{align}
Applying \cref{thm:Evol}, we obtain
\begin{equation}
    \EE\#(X\circ\gamma)^{-1}(0)=\int_0^1 \ell\left(\zeta_{X\circ\gamma}(t)\right)\, \dd t=2\int_0^1 F^X(\dot{\gamma}(t))\, \dd t
\end{equation}
We recognize on the right $2\ell^{F^X}(\gamma).$ To conclude, note that $(X\circ\gamma)^{-1}(0)=\gamma^{-1}(\gamma\cap Z)$ and thus $\#(X\circ\gamma)^{-1}(0)=\#(\gamma\cap Z)$.
\end{proof}

Formulas of the type of \cref{eq:croftonformula} are called \emph{Crofton formula} from the original Crofton formula with curves on the sphere and random hyperplanes.

Constructions of Finsler structures that admit a Crofton formula are known for random hyperplanes in projective space, see \cite{PaivaGelfCrof,BernigCrofton,SchneiderCrofton}. Moreover, a more general result very similar to \cref{propcroftoniftrs} can be found in \cite[Theorem~A]{paivageod}, although the {\zkrok} hypothesis is significantly more general and the construction of the metric $F^X$ explicit with \cref{eqdefFX}.

\begin{remark}
Note that the (semi) Finsler structure satisfying \cref{eq:croftonformula} is unique. Indeed\miknew{,} if $v\in T_pM$ is such that there exists a curve $\gamma:[0,1]\to M$ almost surely transversal to $Z$\miknew{,} such that $\gamma(0)=p$ and $\dot\gamma(0)=v$\miknew{,} then\miknew{,} by \cref{eq:lengthFinsler}\miknew{,} we have $\tfrac{1}{\eps}\ell^{F^X}(\gamma|_{[0,\eps]})\to F^X(v)$ as $\eps\to 0.$ Moreover, by
\cref{lem:almostallcurves} below, almost all $(p,v)\in TM$ admit such a curve.
\end{remark}

\begin{lemma}\label{lem:almostallcurves}
Let $X\randin \mC^1(M,\R)$ be \zkrok and $Z=X^{-1}(0)$. Then for almost every $(p,v)\in TM$ we have that $\PP\{p\in Z, d_pX(v)=0\}=0$.
\end{lemma}
\begin{proof}
In fact, we are only going to use the assumption that $0$ is a regular value of $X$ almost surely.
Let us consider the set:
\be 
A:=\kop \tyu(p,v),f\uyt\in TM\times \mC^1(M,\R)\colon f(p)=0, d_pf(v)=0\pok.
\ee
We need to show that $\PP\{\tyu(p,v),X\uyt\in A\}=0$ for almost every $(p,v)\in TM$. By Tonelli's theorem, since $A$ is measurable, this is equivalent to show that $A$ has measure zero and this can be proven by sectioning in the opposite way (i.e., exchanging the order of integration). Indeed, for each $f\in \mC^1(M,\R)$ such that $f\transv \{0\}$, we have that $A_f:=\kop(p,v)\in TM|\tyu(p,v),X\uyt\in A\pok$ corresponds exactly to $T(f^{-1}(0))$, hence $A_f$ has measure zero for $[X]$-almost every $f$, which by Tonelli implies that $A$ has measure zero.
\end{proof}

Unlike for the length, there are several definitions of volume in Finsler manifolds. One way to define $k$-dimensional volumes of submanifolds is to define a $k$-density, that is\miknew{,} a nonnegative homogeneous function $\varphi_k$ on the simple vectors of $\Lambda^k TM$. The $k$-densities satisfy a pull-back property and thus, given an embedded submanifold $\iota:S\hookrightarrow M $, $\iota^*\varphi_k$ defines a density (in the classical sense) and can be integrated. The $k$-volume of $S$ is then defined to be
\begin{equation}
    \vol_{\varphi_k}(S):=\int_S \iota^*\varphi_k.
\end{equation}
See \cite{volumesFinsler} for the possible choices of $k$-densities and more details. 
One of the most common choices is the \emph{Holmes-Thompson density}. To define it, it is convenient for us to fix a Riemannian metric on our manifold $M$.
\begin{definition}
Let $F$ be a semi Finsler structure on $M$ and let $\zeta(p)\subset T^*_pM$ be the convex body such that $F_p=h_{\zeta(p)}.$ The $k^{th}$ \emph{Holmes-Thompson} density $\varphi_k^{HT}$ is given for all $p\in M$, and all simple vectors $v=v_1\wedge \cdots\wedge v_k\in\Lambda^k T_pM$
\begin{equation}
    \varphi_k^{HT}(v_1\wedge\cdots\wedge v_k):=\frac{\|v_1\wedge\cdots\wedge v_k\|}{\omega_k} \vol_k(\pi_v(\zeta(p)))
\end{equation}
where $\|\cdot\|$ is the norm on $\Lambda^kT_pM$ induced by the Riemannian structure, $\pi_v$ is the orthogonal projection onto $Span(v_1,\ldots,v_k)$ (identifying the space and its dual) and $\vol_k$ is the $k$-dimensional volume in the Riemannian structure in $T_pM$.
\end{definition}

The reader can refer to \cite[p.19]{volumesFinsler}. One can also show that this definition doesn't depend on the choice of the Riemannian metric, however, in our case, this becomes clear with the next lemma.

\begin{lemma}\label{lem:HTandwedge}
Let $F$ be a semi Finsler structure on $M$ such that for each $p\in M$, there is a zonoid $\zeta(p)\in \ZZo(T^*_pM)$ such that $F_p=h_{\zeta(p)}$. Then, the Holmes--Thompson density is given by 
\begin{equation}
    \varphi_k^{HT}=\frac{2}{k!\uball_k}h_{\zeta(p)^{\wedge k}}.
\end{equation}
\end{lemma}
\begin{proof}
This is a consequence of the definition and \cref{lem:supofsimplek}.
\end{proof}

Now with a proof very similar to the proof of \cref{propcroftoniftrs} we obtain a Crofton formula for higher dimensional volumes.

\begin{theorem}[Crofton formula]\label{prop:FinsCroftonVol}
Let $1\leq k\leq m$, let $X_1,\ldots,X_k\randin C^1(M,\R)$ be i.i.d. {\zkrok} fields and let $Z^{(k)}:=(X_1)^{-1}(0)\cap\cdots\cap(X_k)^{-1}(0)$. Let $\iota:S\hookrightarrow M$ be an embedded submanifold of dimension $k$ such that $S\transv Z^{(k)}$ almost surely, then we have 
\begin{equation}
    \EE\#(S\cap Z^{(k)})=k! \uball_k \vol_k^{F^{X_1}}(S)
\end{equation}
where $\vol_k^{F^{X_1}}$ denotes the Holmes--Thompson volume for the semi Finsler structure defined by \cref{eqdefFX}.
\end{theorem}
\begin{proof}
The proof is almost identical to the proof of \cref{propcroftoniftrs} but let us repeat it, if only to compute the constant. Let $X^{(k)}:=(X_1,\ldots,X_k)\randin C^1(M,\R^k)$ and consider $X^{(k)}\circ \iota\randin C^1(S,\R^k)$. Since $S$ is almost surely transversal to $Z^{(k)}=(X^{(k)})^{-1}(0)$, by the pull-back property (\cref{thm:pullback}) it is \zkrok and we have for all $q\in S$ 
\begin{equation}
    \zeta_{X^{(k)}\circ \iota}(q)=\dd_q\iota^* \zeta_{X^{(k)}}(\iota(q))=\dd_q\iota^* \left((\zeta_{X_1}(\iota(q)))^{\wedge k}\right)=\left(\dd_q\iota^* \zeta_{X_1}(\iota(q))\right)^{\wedge k}.
\end{equation}
where the second equality holds because $X^{(k)}:=(X_1,\ldots,X_k)$ and $X_1,\ldots,X_k$ are i.i.d. and the third equality is by definition of the linear maps induced in the exterior algebra. We fix a Riemannian structure on $S$ such that $\iota$ is a Riemannian embedding and we let $\omega_q\in \Lambda^kT^qS$ be the choice of a volume form (if $S$ is not orientable we can work locally). Now we note that $\zeta_{X^{(k)}\circ \iota}(q)$ lives in the one dimensional space $\Lambda^kT_qS$ thus its length is given by:
\begin{align}
    \ell\left(\zeta_{X^{(k)}\circ\iota}(q)\right) &=h_{\zeta_{X^{(k)}\circ\iota}(q)}(\omega_q)+h_{\zeta_{X^{(k)}\circ\iota}(q)}(-\omega_q) \\
    &=h_{\zeta_{X^{(k)}}(\iota(q))}(\dd_q\iota(\omega_q))+h_{\zeta_{X^{(k)}}(\iota(q))}(\dd_q\iota(-\omega_q)) \\
    &=2h_{\seg{\zeta_{X^{(k)}}}(\iota(q))}(\dd_q\iota(\omega_q))    \\
    &=2h_{\seg{\zeta_{X_1}}(\iota(q))^{\wedge k}}(\dd_q\iota(\omega_q))=k!\kappa_k \varphi_k^{HT}(\dd_q\iota(\omega_q)).
\end{align}
Where here $\varphi_k^{HT}$ denotes the Holmes Thompson density for the semi Finsler structure defined by $\seg{\zeta_{X_1}}$. To conclude, we note that $\#(X^{(k)}\circ\iota)^{-1}(0)=\#(S\cap Z^{(k)})$ and thus applying \cref{cor:Evolvol} to the \zkrok field $(X^{(k)}\circ\iota)$ we get 
\begin{equation}
    \EE\#(S\cap X^{(k)})=\int_S\ell\left(\zeta_{f^{(k)}\circ\iota}(q)\right)\,\dd S(q)=k!\kappa_k\int_S\varphi_k^{HT}(\dd_q\iota(\omega_q))\,\dd S(q)
\end{equation}
which is what we wanted.
\end{proof}

If we consider the submanifold $S$ in \cref{prop:FinsCroftonVol} to be again random, given by \zkrok fields, we obtain the following funny formula.

\begin{corollary}
Let $X_1,\ldots,X_k,Y_1,\ldots,Y_{m-k}\randin C^1(M,\R)$ be independent \zkrok fields with $X_1,\ldots,X_k$, respectively $Y_1,\ldots,Y_{m-k},$ identically distributed. Consider $Z_X^{(k)}:=(X_1)^{-1}(0)\cap\cdots\cap (X_k)^{-1}(0)$ and $Z_Y^{(m-k)}:=(Y_1)^{-1}(0)\cap\cdots\cap (Y_{m-k})^{-1}(0)$. Then we have
\begin{equation}
    k!\uball_k\EE\left[ \vol_k^{F^X}\left(Z_Y^{(m-k)}\right)\right]=(m-k)!\uball_{m-k}\EE\left[ \vol_{m-k}^{F^Y}\left(Z_X^{(k)}\right)\right]
\end{equation}
where $\vol_k^{F^X}$, respectively $\vol_{m-k}^{F^Y}$, denotes the Holmes-Thompson volume for the semi Finsler structure defined by $\seg{\zeta_{X_1}}$, respectively by $\seg{\zeta_{Y_1}}$.
\end{corollary}

\begin{proof}
Applying the previous result \cref{prop:FinsCroftonVol} successively to $X_1,\ldots,X_k$, fixing $Z_Y^{(m-k)}$ and to $Y_1,\ldots,Y_{m-k}$ fixing $Z_X^{(k)}$, we get, using the independence assumption, that both sides are equal to $\EE\#(Z_X^{(k)}\cap Z_Y^{(m-k)})$.
\end{proof}

\section{Examples}\label{sec:ex}
\subsection{Abundance of {\zkrok} fields}\label{sub:abu}

The following result shows that \zkrok random fields are \emph{dense} in the family of smooth random fields with integrable $\mC^1$ norm.
\begin{theorem}\label{thm:abundance}
Let $Y\randin \mC^q(M,\R^k)$ be a random field, with $q\ge 1+\max\{m-k,0\}$\footnote{This is the minimal regularity required for Sard's theorem \cite{sard} to hold.}, such that $\EE\{J_pY\}$ is finite and continuous with respect to $p\in M$. Let $\lambda\randin \R^k$ be an independent random vector with a continuous nowhere vanishing bounded density $\rho_\lambda$.
Then $X:=Y-\lambda$ is \zkrok.
\end{theorem}
\begin{proof}
    Let us show the validity of the \zkrok hypotheses one by one.
    \begin{enumerate}[wide]
        \item Clearly $X\randin \mC^1$.
        \item Observe that $0$ is a critical value of $Y-x$ if and only if $x$ is a critical value of $Y.$ By Sard's theorem, the set of such points has Lebesgue measure zero and since the law of $\lambda$ is absolutely continuous with respect to Lebesgue,
        we obtain  \zkrok-\cref{itm:krok2} by integrating first with respect to $\lambda$ then with respect to $Y$. 
        \item We can express the density of the random vector $X(p)\randin \R^k$ as follows:
        \bega
        \rho_{X(p)}(x)
        \int_{\R^k}\rho_\lambda(t-x)d[Y(p)](t)=\EE\kop\rho_\lambda \tyu Y(p)-x \uyt\pok.
        \eega
        The latter expectation is taken with respect to the randomness of $Y$. Notice that $\rho_{X(p)}(x)>0$ for all $x\in \R^k$ because $\rho_\lambda$ is assumed to have the same property.
        \item 
        The continuity of $\rho_{X(p)}(x)$ can be shown using the Dominated Convergence Theorem since  $\rho_\lambda$ is uniformly bounded.
        \item Let $\mathbbm{1}_B$ be the characteristic function of a Borel set $B\subset \mC^1(M,\R^k)$. For any $(p,x)\in M\times \R^k$ we define the probability measure 
        \be\label{eq:abudefmu} 
        \mu(p,x)(B):=\frac{\EE\kop\mathbbm{1}_B\tyu Y-Y(p)+x\uyt \rho_{\lambda}\tyu Y(p)-x\uyt\pok}{\rho_{X(p)}(x)}
        \ee
        To see that $\mu(p,\cdot)(\cdot)$ is a regular conditional probability (see \cref{sub:remarkrok}) for $X$ given $X(p)$, let us take a Borel subset $V\subset \R^k$ and compute
        \bega 
        \PP\kop X\in B; X(p)\in V\pok &= \int_{\mC^1(M,\R^k)}\tyu\int_{\R^k}\mathbbm{1}_{B}(f-t)\mathbbm{1}_{V}(f(p)-t)\rho_\lambda(t)d\R^k(t)\uyt d[Y](f)
        \\
        &=
        \int_{\mC^1(M,\R^k)}\tyu\int_{V}\mathbbm{1}_{B}(f-f(p)+x)\rho_\lambda(f(p)-x)d\R^k(x)\uyt d[Y](f)
        \\
        &=
        \int_V \EE\kop\mathbbm{1}_B\tyu Y-Y(p)+x\uyt \rho_{\lambda}\tyu Y(p)-x\uyt\pok\frac{\rho_{X(p)}(x)}{\rho_{X(p)}(x)}
        d\R^k(x)
        \\
        &=
        \int_V\mu(p,x)(B)d[X(p)](x).
        \eega 
        Finally, we prove \zkrok-\ref{itm:krok5} by showing point (\ref{itm:tecnokrok2}) of \cref{prop:tecnicalzkrok}. Let $\a$ be a bounded and continuous functional on $\mC^1(M,\R^k)$ and let $(p_n,x_n)\to (p,0)$ in $M\times \R^k$. Then 
        \be 
        \EE\kop (J_{p_n}X)\a(X)|X(p_n)=x_n\pok=\frac{\EE\kop (J_{p_n}Y)\a(Y-Y(p_n)+x_n)\rho_{\lambda}\tyu Y(p_n)-x_n\uyt\pok}{\rho_{X(p_n)}(x_n)}.
        \ee
        We already proved that the denominator is continuous and never vanishing. The convergence of the numerator can be proved using the following version of the Dominated Convergence Theorem, which is a corollary of Fatou's lemma.
        \begin{lemma}\label{lem:dominco}
            Let $0\le f_n\le g_n$ be random variables such that $f_n\to f$ and $g_n\to g$ almost surely. Assume that $\EE\{g_n\}\to \EE\{g\}$, then $\EE\{f_n\}\to \EE\{f\}$.
        \end{lemma}
        To conclude, we apply \cref{lem:dominco} with $f_n=J_{p_n}Y\a(Y-Y(p_n)+x_n)\rho_{\lambda}(Y(p_n)-x_n)$ and $g_n=(J_{p_n}Y) C$, where $C>0$ is a constant such that $\a(f)\rho_{\lambda}(x)\le C$ for all $f$ and $x$. Here we are using the crucial hypothesis that $\EE\{J_{p_n}Y\}\to \EE\{J_{p}Y\}$.
    \end{enumerate}
\end{proof}

\begin{remark}
If we used the alternative weaker version of {\zkrok} hypotheses discussed in \cref{rem:weakzkrok}, the requirement that $\rho_\lambda$ is nonvanishing could be dropped.
\end{remark}
\subsection{Random level sets}\label{ex:randlev}
    Let $\varphi\in C^\infty(M,\R^k)$ be a fixed function and let $\lambda\randin \R^k$ be a random vector whose law admits a \emph{continuous} density $\rho_\lambda:\R^k\to\R$. Then the random field
    \begin{equation}
        X:=\varphi-\lambda\randin C^\infty(M,\R^k)
    \end{equation}
     is \zkrok.
    Indeed, this is a special case of \cref{thm:abundance} except for the fact that we don't need to assume nothing but the continuity of $\rho_\lambda$. So, \zkrok-\cref{itm:krok2} follows from Sard's theorem; $X(p)$ admits the continuous density given for every $x\in\R^k$ by $\rho_{X(p)}(x)=\rho_\lambda(\varphi(p)-x)$ and this gives \zkrok-\cref{itm:krok:3} and \zkrok-\cref{itm:krok4}.
    

    Finally, to prove \zkrok-\cref{itm:krok5}, we let $\mu(p,x)$ be the Dirac delta measure $\mu(p,x)=\delta_{\varphi-\varphi(p)+x}$, which corresponds to \eqref{eq:abudefmu} in this case. Reasoning as in the proof of \cref{thm:abundance}, one can check that this is a regular conditional probability for $X$ given $X(p)$, but this time it is automatic to see that $\mu$ satisfies \zkrok-\cref{itm:krok5}, even if $\rho_\lambda$ is not bounded or if it has zeroes.
    
    Note that in that case, we have 
    \begin{equation}\label{eq:"random"vectorRLS}
        (\dd_pX^1\wedge\cdots\wedge \dd_p X^k|X(p)=0)=\dd_p\varphi^1\wedge\cdots\wedge \dd_p \varphi^k
    \end{equation}
    almost surely. Thus, we obtain for all $p\in M$:
\begin{equation}
    \zeta_X(p)=\rho_{\lambda}(\varphi(p))[0,\dd_p\varphi^1\wedge\cdots\wedge \dd_p \varphi^k].
\end{equation}
    In particular, notice that the zonoid is $\{0\}$ at critical points of $\varphi$ and thus is $\{0\}$ everywhere if $\varphi$ is constant.
    
    In this setting, \cref{thm:Evol} translates into the coarea formula for the function $f(p)=J_p\f\cdot \rho_\lambda(\f(p))$, while \cref{thm:Ecurrent} yields:
    \be 
    \int_{\R^k}\rho_{\lambda}(t)\tyu\int_{\f^{-1}(t)}\omega|_{\ker d\f}\uyt d\R^k(t)
    =
    \int_M \rho_{\lambda}(\varphi(p))  \dd_p\varphi^1\wedge\cdots\wedge \dd_p \varphi^k\wedge
    \omega.
    \ee
    Moreover, in the case where $k=1$, the semi Finsler structure defined by $X$ (see \cref{sec:Finsler}) is given for all $v\in T_pM$ by
    \begin{equation}
        F^X_p(v)=\frac{\rho_\lambda(\varphi(p))}{2}|d_p\varphi(v)|.
    \end{equation}
    Then, if $\gamma:[0,1]\to M$ is a smooth curve that is transversal to $\varphi$, one can see that its length for this semi Finsler structure is given by $\ell^{F^X}(\gamma)=\frac{1}{2} \PP\left(\lambda\in [\varphi(\gamma(0)),\varphi(\gamma(1))]\right).$
\subsection{Finite dimensional fields}\label{sub:finifields}

Let us detail the case where the random field lives in a finite dimensional subspace of $C^\infty(M,\R^k)$. This example could help the reader to understand better the \zkrok conditions and the construction of the zonoid section.

\begin{proposition}\label{prop:finitezkrokcond}
     Let $\mathcal{F}\subset C^\infty(M,\R^k)$ be a subspace of dimension $n<\infty$ endowed with a scalar product and such that for all $p\in M,$ the map $ev_p:\F\to\R^k$, $\varphi\mapsto\varphi(p)$ is surjective. Let $X\randin \mathcal{F}$ be a random function whose law admits a \emph{continuous} density $\rho_X:\mathcal{F}\to \R$ such that $\rho_X(0)>0$ and such that when $\|\varphi\|\to\infty$, we have $\rho_X(\varphi)=O(\|\varphi\|^{-\alpha})$ for some $\alpha>n$. Then $X$ is \zkrok.
\end{proposition}
\begin{proof}
Let us detail the \zkrok conditions one by one.

For \zkrok.\cref{itm:krok2}, the trick is to use the \emph{parametric transversality theorem}, see \cite[Theorem~2.7]{Hirsch}. Indeed, consider the function $\Phi:  \F\times M\to \R$ given by $\Phi(\varphi,p)=\varphi(p)$. Then its differential at $(\varphi,p)$ is given by $ev_p\oplus \dd_p\varphi $. By assumption this is surjective and thus the map $\Phi$ is transversal to zero, i.e. $0$ is a regular value of $\Phi$. The parametric transversality theorem then tells us that for almost all $\varphi\in\F$, the map $\varphi\mapsto \varphi(p)$ is transversal to $0$, i.e. for almost all $\varphi\in\F$, $0$ is a regular value of $\varphi$ which is what we wanted.

The law of $X(p)$ is the push forward of the law of $X$ by the linear map $ev_p:\F\to \R^k$. Suppose $B\subset \R^k$ is a Borel subset of measure $0$. Then $\PP(X(p)\in B)=\PP(X\in ev_p^{-1}(B))$. Let us denote 
\begin{equation}
    \F_p:=\ker(ev_p)=\set{\varphi\in\F}{\varphi(p)=0}.
\end{equation}
Then the space $ev_p^{-1}(x)$ is an affine subspace parallel to $\F_p$ which, by the surjectivity of $ev_p$, is of dimension $n-k$. Thus $ev_p^{-1}(B)\cong B\times \F_p$ is of Lebesgue measure zero in $\F$. Since the law of $X$ is, by assumption, absolutely continuous with respect to the Lebesgue measure on $X$, we obtain that $\PP(X\in ev_p^{-1}(B))=0$ and thus $\PP(X(p)\in B)=0$. This proves that the law of $X(p)$ is absolutely continuous with respect to Lebesgue on $\R^k$ and thus admits a density $\rho_{X(p)}:\R^k\to\R$ and this proves the property \zkrok-\cref{itm:krok:3}.

We can compute this density, we have for all $p\in M$ and $x\in\R^k$:
    \begin{equation}\label{eq:rhof(p)falpha}
        \rho_{X(p)}(x)=\int_{ev_p^{-1}(x)}\rho_X (\varphi) \dd\varphi.
    \end{equation}
    To prove the continuity requirement \zkrok-\ref{itm:krok4}, we can use the assumption of the behavior at infinity of $\rho_X$ and dominated convergence. Indeed, with the Euclidean structure, we can assume $\F=\R^n$. Let $p\in M$, we can assume that $\F_p=\R^{n-k}\subset \R^n$ is the space spanned by the $n-k$ first coordinates. Then we write $\rho_X(y,x)$ with $y\in \R^{n-k}$ and $x\in \R^k$. Let now $p_j\to p$ and $x_j\to 0$, let $g_j\in O(n)$ be such that $g_j^{-1}(\F_{p_j})=\F_p=\R^{n-k}$ then we have 
    \begin{equation}
        \rho_{X(p_j)}(x_j)=\int_{\R^{n-k}}\rho_X (g_j(y),x_j) \dd y.
    \end{equation}
    On $\R^{n-k}$, the function $y\mapsto\|y\|^{-\alpha}$ is integrable at infinity if and only if $\alpha>n-k$. Thus under our assumption $y\mapsto \rho_X (g_j(y),x_j)$ is dominated by an integrable function uniformly on $j$ and by dominated convergence we get \zkrok-\cref{itm:krok4}.
    \item We define $\mu(p,x)$ to be the probability measure on $\F$ with support on the affine space $ev_p^{-1}(x)$ that admits the continuous density $\rho_{X,p,x}:ev_p^{-1}(x)\to \R$ that is $0$ if $\rho_{X(p)}(x)= 0$ and else is given by 
    \begin{equation}\label{eq:mupxforfinitetype}
        \rho_{X,p,x}:=\frac{1}{\rho_{X(p)}(x)} \rho_X|_{ev_p^{-1}(x)}.
    \end{equation}
    Then $\mu(p,x)$ defines a regular conditional probability for $X$ given $X(p)$. Now let us note that for all $p\in M$, there exists a constant $c=c(p)>0$ such that $J_p\varphi\leq c\|\varphi\|^k$. Thus the function $\varphi\mapsto J_p\varphi \rho_X(\varphi)$ is at infinity an $O\left(\|\varphi\|^{-(\alpha-k)}\right)$ and this is integrable on $ev_p^{-1}(x)\cong\R^{n-k}$ if and only if $\alpha>n$ which is precisely our assumption and this gives us the finiteness condition in \zkrok-\cref{itm:krok5}. To see the continuity, let $\Psi:\F\to \R$ be a bounded continuous function. Let $p_j\to p$ and $x_j\to 0$, we repeat the argument of the previous item to write
    \begin{equation}
       \langle J_p\cdot\mu(p_j,x_j),\Psi\rangle=\frac{1}{\rho_{X(p_j)}(x_j)}\int_{\R^{n-k}}\Psi(g_j (y),x_j)  J_p(g_j(y),x_j)\,\rho_X (g_j(y),x_j) \dd y
    \end{equation}
    for some sequence $g_j\in O(n)$ converging to $\Id$. Since $\rho_X(0)>0$ we get from \cref{eq:rhof(p)falpha} that $\rho_{X(p)}(0)>0$ for every $p\in M$ and we can argue similarly as before: this is dominated by a $O\left(\|\varphi\|^{-(\alpha-k)}\right)$ at infinity which is integrable and we conclude by dominated convergence to obtain \zkrok-\cref{itm:krok5}.
\end{proof}

In that case we can compute explicitly the zonoid section.

\begin{proposition}\label{prop:zonfinitetype}
     Let $X\randin \F\subset C^1(M,\R^k)$ be \zkrok and as in \cref{prop:finitezkrokcond}. For every $p\in M$ and every $w\in \Lambda^kT_pM$ we have 
     \begin{align}\label{eq:finitetypesuppfct}
         h_{\zeta_X(p)}(w)          &=\int_{\F_p} \max\left\{0,(\dd_p\varphi^1\wedge\cdots\wedge \dd_p\varphi^k)(w)\right\} \rho_X(\varphi)\dd\varphi  \label{eq:hzetafinite}\\
         e_X(p)  (w)                &=\int_{\F_p} (\dd_p\varphi^1\wedge\cdots\wedge \dd_p\varphi^k)(w) \rho_X(\varphi)\dd\varphi  \\
     \end{align}
     where recall that $\F_p=\ker(ev_p)=\set{\varphi\in\F}{\varphi(p)=0 }$ and $\rho_X:\F\to\R$ is the density of the law of $X\randin\F$.
\end{proposition}
\begin{proof}
    We already did all the work in the proof of \cref{prop:finitezkrokcond}. In particular we computed the measure $\mu(p,x)$ in \cref{eq:mupxforfinitetype}. Letting $x=0$ and multiplying by $\rho_{X(p)}(0)$ gives the result.
\end{proof}


\appendix
\section{Comparison with other typical sets of hypotheses}\label{apx:comp}
We compare the {\zkrok} hypotheses (\cref{def:zkrok}) and \cref{thm:alphaKR} with other versions of Kac-Rice formula reported in \cite[Sec. 11.2]{AdlerTaylor} and \cite[Sec. 6.1.2 ]{AzaisWscheborbook}.
In the textbooks, a more general type of \emph{weight} $\a$ is considered: when $\a=\a(F,Y,p)$ depends also on an additional random field $Y$ (in \cite{AzaisWscheborbook}, while it is called $g$ in \cite{AdlerTaylor}). Here, we will only discuss the case of Theorem \ref{eq:aKR}, see also Remark \ref{rem:tuttiboni}.
\begin{remark}
The passage from the simple Kac-Rice formula, with $\a=1$, to the case when $\a$ is just a measurable function $\a\colon M\to \R$ that does not depend on $F$, is automatic. This is explained in \cite[Remark 2.7]{KRStec}.
\end{remark}
\begin{remark}\label{rem:tuttiboni}
The more general frameworks, i.e. when $\a=\a(Y,F,p)$ depends on an additional random field, can be all covered by assuming that $\a\colon \mC^1(M,\R^k)\times M\randto\R$ is random. Under this perspective, the hypotheses on the additional field $Y$ (in \cite[Theorem 11.2.1]{AdlerTaylor} and in \cite[Theorem 6.10]{AzaisWscheborbook}) can be viewed (and perhaps simplified) as the conditions under which it is possible to separate the randomness of $\a$ and that of $X$, by conditioning on the former and to make rigorous the following line of identities:
\bega  
\EE\kop\sum_{p\in F^{-1}(0)}\alpha(F,p)\pok&=\EE_{\a}\EE_{(X|\a=a)}\kop\sum_{p\in F^{-1}(0)}a(F,p)\pok
\\
&=\EE_\a\int_M\EE\kop a(F,p)J_pF \Big| F(p)=0,\a=a\pok \rho_{F(p)|\a=a}(0)
\\
&=\int_M\EE\kop a(F,p)J_pF \Big| F(p)=0\pok \rho_{F(p)}(0).
\eega 
and to apply \eqref{eq:aKR} in the inmost expectation, thinking of $\a$ as fixed.
\end{remark}
\subsubsection{Adler and Taylor's Expectation Metatheorem}
We compare the hypotheses (a), (b), (c), (d), (e), (f) and (g) in \cite[Theorem 11.2.1]{AdlerTaylor} to the \zkrok conditions.
\begin{enumerate}[{},wide] 
    \item (a) is equivalent to \zkrok-\ref{itm:krok1}
    \item (b) is implied by \zkrok-\ref{itm:krok:3} and \zkrok-\ref{itm:krok4}, together. In the opposite direction, \zkrok-\ref{itm:krok4} requires continuity also with respect to the spacial variable $p\in M$, which corresponds to $t\in T$ in \cite{AdlerTaylor}. Let us call (b+), this slightly stronger version of hypothesis (b).
    \item We will only consider the case in which $g\equiv 1$, thus (e) is always satisfied, while (c) reduces to the condition that the conditional density $p_t(x|\nabla f(t))$ of $f(t)$ given $\nabla f(t)$ exists, it is bounded, and it is continuous at $x=0$, uniformly in $t$. There is no such requirement among the \zkrok conditions.
    \item  Under finiteness of moments (f), condition (d) is comparable to   \zkrok-\ref{itm:krok5}, though none of the two possible implications hold. Indeed, condition (d) concerns only the pointwise distributions of the jet $j^1_pX=(p,X(p),dX(p)$, while \zkrok-\ref{itm:krok5} concerns the distribution of the pairs $(X,X(p))$, but it does not require the existence of the conditional density of $\det dX(p)$ conditioned to $X(p)$. Moreover, it is shown in \cite[Lemma 11.2.11]{AdlerTaylor} that (a), (b) and (d) together imply \zkrok-\ref{itm:krok2}.\footnote{In the in the current version of the book \cite{AdlerTaylor}, the statement of Lemma 11.2.11 includes the hypothesis (g) from Theorem 11.2.1. However, in the document \emph{Correction and Commentary} (downloadable on the book's first author's website) this hypothesis is said to be removable.  
    }
    \item 
    Thus, (a),(b+),(d),(f) presumably imply \zkrok-\ref{itm:krok1}-\ref{itm:krok5}.
    \item We don't see the role of hypothesis (g) (which can be roughly thought as the requirement that $dX$ is Holder-continuous in probability). 
    Indeed, it does not appear in the version of \cite{AzaisWscheborbook}. This might be due to the different argument used in the proof to prove the inequality ``$\ge$''. This is the difficult step in all versions of the proof of Kac-Rice formula (the other inequality can be deduced via the coarea formula and Fatou Lemma). 
\end{enumerate}
\subsubsection{Azais and Wschebor's version of Rice's formula} In the case of zero dimensional submanifolds $k=m$ and $M\subset \R^k$ is an open subset, the {\zkrok} hypotheses (\cref{def:zkrok}) are almost identical to the hypotheses of \cite[Theorem 6.7]{AzaisWscheborbook} for the level $u=0$. 
\begin{enumerate}[(i)]
    \item is equivalent to \zkrok-\ref{itm:krok1}.
    \item is equivalent to the combination of \zkrok-\ref{itm:krok:3} and \zkrok-\ref{itm:krok4}.
    \item is to be compared with the formulation of \zkrok-\ref{itm:krok5} that is given in point (\ref{itm:tecnokrok4}) of \cref{prop:tecnicalzkrok}. In the language of the latter, (iii) says that:
    
    \newtheorem*{aza}{Hypothesis (iii) of \cite[Theorem 6.7]{AzaisWscheborbook}}
    \begin{aza}
    There is a regular conditional probability of $X$ given $X(p)$ such that for any continuous function $\beta\in \mC(\mC^1(M,\R^k); \R)$ and any converging sequence $(p_n,x_n)\to (p_0,x_0)$ in a neighborhood of $M\times \{0\}$ in $M\times \R^k$, we have that
    \be\label{eq:apxlimbeta}
    \EE\kop\beta(X)\Big| X(p_n)=x_n\pok \to \EE\kop\beta(X)\Big|X(p_0)=x_0\pok.
    \ee
    \end{aza}
    The differences between the condition above and ours are three: 
    \begin{enumerate}[1.]
        \item In \cref{prop:tecnicalzkrok}.(\ref{itm:tecnokrok4}) the property should be valid for all sequences $\beta_n\to\beta_0$. From point (\ref{itm:tecnokrok2}) of \cref{prop:tecnicalzkrok} it is clear that this difference is irrelevant.
        \item For Condition (\ref{itm:tecnokrok4}) of \cref{prop:tecnicalzkrok} to be true it is sufficient to verify \eqref{eq:apxlimbeta} when $x_0=0$.
        \item In \cref{prop:tecnicalzkrok}.(\ref{itm:tecnokrok4}) a bound is assumed: $\beta(f)\le CJ_{p_n}f$, while in (iii) there is no restriction on the class of functions $\beta$ for which \eqref{eq:apxlimbeta} should hold. Because of this, condition (iii) seems ill posed in that the expression \eqref{eq:apxlimbeta} may take infinite values even for Gaussian fields, for instance with $\beta(f)=\exp({|f(p)|^3})$, where $p\in M$ is a fixed point.
    \end{enumerate}
    \item is equivalent to \zkrok.
\end{enumerate}
    In conclusion, we can say that \zkrok-\ref{itm:krok5} is a weaker assumption than (iii), while all other hypotheses are equivalent, thus \cref{thm:alphaKR} implies \cite[Theorem 6.7]{AzaisWscheborbook}.
       \section{Source code for symbols}
    The symbol $\randin$ used in this article was made out of two symbols $\subset$ combined into the command \texttt{\textbackslash randin} with the following code
    \begin{Verbatim}[fontsize=\tiny]
        \def\randin{%
  \mathchoice%
    {\raisebox{-.35ex}{$\displaystyle{^\subset}$}\mkern-11.5mu\raisebox{+.45ex}{$\displaystyle{_\subset}$}}
    {\mkern+1mu\raisebox{-.27ex}{$\textstyle{^\subset}$}\mkern-11.7mu\raisebox{+.45ex}{$\textstyle{_\subset}$}}
    {\raisebox{.35ex}{$\scriptstyle\subset$}\mkern-14mu\raisebox{-.15ex}{$\scriptstyle\subset$}}
    {\raisebox{.3ex}{$\scriptscriptstyle\subset$}\mkern-13.5mu\raisebox{-.10ex}{$\scriptscriptstyle\subset$}}
}
    \end{Verbatim}
    
    The symbol $\randto$ is the command \texttt{\textbackslash randto} defined with the following code.
   \begin{Verbatim}[fontsize=\tiny]
\newcommand{\maschera}{\textcolor{white}{\scalebox{0.3}{$\blacktriangle$}}}
\def\FlatOmega{%
  \mathchoice{%
    \displaystyle{\Omega}\mkern-14mu\raisebox{+.166ex}{$\displaystyle{\maschera}$}
     \mkern+7mu\raisebox{+.166ex}{$\displaystyle{\maschera}$}}{% 
    \hbox{$\textstyle{\Omega}$}\mkern-14mu\raisebox{+.166ex}{\hbox{$\textstyle{\maschera}$}}
     \mkern+7mu\raisebox{+.166ex}{\hbox{$\textstyle{\maschera}$}}}{
   \scriptstyle{\Omega}\mkern-14mu\raisebox{+.13ex}{$\scriptstyle{\maschera}$}
    \mkern+5mu\raisebox{+.13ex}{$\scriptstyle{\maschera}$}}{%
   \scriptscriptstyle{\Omega}\mkern-14mu\raisebox{+.13ex}{$\scriptscriptstyle{\maschera}$}
    \mkern+5mu\raisebox{+.13ex}{$\scriptscriptstyle{\maschera}$}}}
\newcommand{\scaledFlatOmega}{{\scalebox{0.8}{$_{\FlatOmega}$}}}
\def\randto{%
  \mathchoice{%
    \raisebox{-.101ex}{$\displaystyle{-}$}\mkern-4.4mu\raisebox{.729ex}{$\displaystyle{\scaledFlatOmega}$}
     \mkern-5.2mu\raisebox{-.101ex}{$\displaystyle\to$}}{%
    \raisebox{-.101ex}{\hbox{$\textstyle{-}$}}\mkern-4.4mu\raisebox{.729ex}{\hbox{$\textstyle{\scaledFlatOmega}$}}
     \mkern-5.2mu\raisebox{-.101ex}{\hbox{$\textstyle\to$}}}{%
    \raisebox{-.101ex}{$\scriptstyle{-}$}\mkern-4.4mu\raisebox{.729ex}{$\scriptstyle{\scaledFlatOmega}$}
     \mkern-5.2mu\raisebox{-.101ex}{$\scriptstyle\to$}}{%
    \raisebox{-.101ex}{$\scriptscriptstyle{-}$}\mkern-4.4mu\raisebox{.729ex}{$\scriptscriptstyle{\scaledFlatOmega}$}
     \mkern-5.2mu\raisebox{-.101ex}{$\scriptscriptstyle\to$}}}
    \end{Verbatim}
    
\bibliographystyle{alpha}
\bibliography{Zonoid_section.bib}

\begin{thebibliography}{MPRW15}

\bibitem[AK18]{KazaAveragezeroes}
Dmitri Akhiezer and Boris Kazarnovskii.
\newblock Average number of zeros and mixed symplectic volume of finsler sets.
\newblock {\em Geometric and Functional Analysis}, 28(6):1517--1547, Dec 2018.

\bibitem[All72]{Allard1972Varifold}
William~K. Allard.
\newblock On the first variation of a varifold.
\newblock {\em Annals of Mathematics}, 95(3):417--491, 1972.

\bibitem[APB10]{paivageod}
Juan-Carlos \`Alvarez-Paiva and Gautier Berck.
\newblock Finsler surfaces with prescribed geodesics, 2010.

\bibitem[APT04]{volumesFinsler}
J.~C. \'{A}lvarez Paiva and A.~C. Thompson.
\newblock Volumes on normed and {F}insler spaces.
\newblock In {\em A sampler of {R}iemann-{F}insler geometry}, volume~50 of {\em
  Math. Sci. Res. Inst. Publ.}, pages 1--48. Cambridge Univ. Press, Cambridge,
  2004.

\bibitem[AT07]{AdlerTaylor}
R.~J. Adler and J.~E. Taylor.
\newblock {\em Random fields and geometry}.
\newblock Springer Monographs in Mathematics. Springer, New York, 2007.

\bibitem[AV75]{ArsVit}
Zvi Artstein and Richard~A. Vitale.
\newblock {A Strong Law of Large Numbers for Random Compact Sets}.
\newblock {\em The Annals of Probability}, 3(5):879 -- 882, 1975.

\bibitem[AW09]{AzaisWscheborbook}
Jean-Marc Azais and Mario Wschebor.
\newblock {\em Level sets and extrema of random processes and fields}.
\newblock John Wiley \& Sons, Inc., Hoboken, NJ, 2009.

\bibitem[BBLM22]{ZA}
Paul Breiding, Peter Bürgisser, Antonio Lerario, and Léo Mathis.
\newblock The zonoid algebra, generalized mixed volumes, and random
  determinants.
\newblock {\em Advances in Mathematics}, 402:108361, 2022.

\bibitem[BCS00]{introRiemannFinsler}
D.~Bao, S.-S. Chern, and Z.~Shen.
\newblock {\em An introduction to {R}iemann-{F}insler geometry}, volume 200 of
  {\em Graduate Texts in Mathematics}.
\newblock Springer-Verlag, New York, 2000.

\bibitem[Ber77]{Berry_1977}
M~V Berry.
\newblock Regular and irregular semiclassical wavefunctions.
\newblock {\em Journal of Physics A: Mathematical and General},
  10(12):2083--2091, dec 1977.

\bibitem[Ber07]{BernigCrofton}
Andreas Bernig.
\newblock Valuations with crofton formula and finsler geometry.
\newblock {\em Advances in Mathematics}, 210(2):733--753, 2007.

\bibitem[BFS14]{bernig2014integral}
Andreas Bernig, Joseph~HG Fu, and Gil Solanes.
\newblock Integral geometry of complex space forms.
\newblock {\em Geometric and Functional Analysis}, 24(2):403--492, 2014.

\bibitem[Bil99]{Billingsley}
Patrick Billingsley.
\newblock {\em Convergence of probability measures}.
\newblock Wiley Series in Probability and Statistics: Probability and
  Statistics. John Wiley \& Sons, Inc., New York, second edition, 1999.
\newblock A Wiley-Interscience Publication.

\bibitem[BKL18]{breiding2018geometry}
Paul Breiding, Khazhgali Kozhasov, and Antonio Lerario.
\newblock On the geometry of the set of symmetric matrices with repeated
  eigenvalues.
\newblock {\em Arnold Mathematical Journal}, 4(3):423--443, 2018.

\bibitem[BL16]{Brgisser2016ProbabilisticSC}
Peter B{\"u}rgisser and Antonio~Marcondes Ler{\'a}rio.
\newblock Probabilistic schubert calculus.
\newblock {\em Journal f{\"u}r die reine und angewandte Mathematik (Crelles
  Journal)}, 2020:1 -- 58, 2016.

\bibitem[BLLP19]{basu2019random}
Saugata Basu, Antonio Lerario, Erik Lundberg, and Chris Peterson.
\newblock Random fields and the enumerative geometry of lines on real and
  complex hypersurfaces.
\newblock {\em Mathematische Annalen}, 374(3):1773--1810, 2019.

\bibitem[Bre10]{brezis2010functional}
H.~Brezis.
\newblock {\em Functional Analysis, Sobolev Spaces and Partial Differential
  Equations}.
\newblock Universitext. Springer New York, 2010.

\bibitem[BS98]{bogachev}
V.I. Bogachev and American~Mathematical Society.
\newblock {\em Gaussian Measures}.
\newblock Mathematical surveys and monographs. American Mathematical Society,
  1998.

\bibitem[BT95]{botttu}
R.~Bott and L.W. Tu.
\newblock {\em Differential Forms in Algebraic Topology}.
\newblock Graduate Texts in Mathematics. Springer New York, 1995.

\bibitem[CCJ19]{canzani2019probabilistic}
Y.~Canzani, L.~Chen, and D.~Jakobson.
\newblock {\em Probabilistic Methods in Geometry, Topology and Spectral
  Theory}.
\newblock Contemporary Mathematics. American Mathematical Society, 2019.

\bibitem[CH20]{CaHa2020}
Yaiza Canzani and Boris Hanin.
\newblock Local universality for zeros and critical points of monochromatic
  random waves.
\newblock {\em Communications in Mathematical Physics}, 378(3):1677--1712,
  2020.

\bibitem[{\c{C}}{\i}n11]{Erhan}
E.~{\c{C}}{\i}nlar.
\newblock {\em Probability and Stochastics}.
\newblock Graduate Texts in Mathematics. Springer New York, 2011.

\bibitem[CM15]{CammarotaM2015}
V.~Cammarota and D.~Marinucci.
\newblock On the limiting behaviour of needlets polyspectra.
\newblock {\em Ann. Inst. Henri Poincar\'{e} Probab. Stat.}, 51(3):1159--1189,
  2015.

\bibitem[CM18]{CamMari2018Berry}
Valentina Cammarota and Domenico Marinucci.
\newblock {A quantitative central limit theorem for the Euler–Poincaré
  characteristic of random spherical eigenfunctions}.
\newblock {\em The Annals of Probability}, 46(6):3188 -- 3228, 2018.

\bibitem[DR18]{DangRiv2018}
Nguyen~Viet Dang and Gabriel Rivi{\`e}re.
\newblock Equidistribution of the conormal cycle of random nodal sets.
\newblock {\em Journal of the European Mathematical Society}, 2018.

\bibitem[Dud02]{dudley}
R.~M. Dudley.
\newblock {\em Real Analysis and Probability}.
\newblock Cambridge Studies in Advanced Mathematics. Cambridge University
  Press, 2 edition, 2002.

\bibitem[FLL15]{FyLeLu}
Yan~V. Fyodorov, Antonio Lerario, and Erik Lundberg.
\newblock On the number of connected components of random algebraic
  hypersurfaces.
\newblock {\em J. Geom. Phys.}, 95:1--20, 2015.

\bibitem[{Gas}20]{Gass2020}
Louis {Gass}.
\newblock {Almost sure asymptotics for Riemannian random waves}.
\newblock {\em arXiv e-prints}, page arXiv:2005.06389, May 2020.

\bibitem[GW14]{GaWe1}
Damien Gayet and Jean-Yves Welschinger.
\newblock Lower estimates for the expected {B}etti numbers of random real
  hypersurfaces.
\newblock {\em J. Lond. Math. Soc. (2)}, 90(1):105--120, 2014.

\bibitem[GW15]{GaWe3}
Damien Gayet and Jean-Yves Welschinger.
\newblock Expected topology of random real algebraic submanifolds.
\newblock {\em J. Inst. Math. Jussieu}, 14(4):673--702, 2015.

\bibitem[GW16]{GaWe2}
Damien Gayet and Jean-Yves Welschinger.
\newblock Betti numbers of random real hypersurfaces and determinants of random
  symmetric matrices.
\newblock {\em J. Eur. Math. Soc. (JEMS)}, 18(4):733--772, 2016.

\bibitem[Hau14]{hausdorff1914grundzüge}
F.~Hausdorff.
\newblock {\em Grundz{\"u}ge der mengenlehre}.
\newblock G{\"o}schens Lehrb{\"u}cherei/Gruppe I: Reine und Angewandte
  Mathematik Series. Von Veit, 1914.

\bibitem[Hir94]{Hirsch}
Morris~W. Hirsch.
\newblock {\em Differential topology}, volume~33 of {\em Graduate Texts in
  Mathematics}.
\newblock Springer-Verlag, New York, 1994.
\newblock Corrected reprint of the 1976 original.

\bibitem[Kaz20]{Kazarnovskii2020}
B.~Ya. Kazarnovskii.
\newblock Average number of roots of systems of equations.
\newblock {\em Functional Analysis and Its Applications}, 54(2):100--109, Apr
  2020.

\bibitem[KKW13]{Wig2013arith}
Manjunath Krishnapur, Par Kurlberg, and Igor Wigman.
\newblock Nodal length fluctuations for arithmetic random waves.
\newblock {\em Annals of Mathematics}, 177(2):699--737, 2013.

\bibitem[KL20]{kozhasov2020number}
Khazhgali Kozhasov and Antonio Lerario.
\newblock On the number of flats tangent to convex hypersurfaces in random
  position.
\newblock {\em Discrete \& Computational Geometry}, 63(1):229--254, 2020.

\bibitem[Kos93]{kostlan:93}
Eric Kostlan.
\newblock On the distribution of roots of random polynomials.
\newblock In {\em From {T}opology to {C}omputation: {P}roceedings of the
  {S}malefest ({B}erkeley, {CA}, 1990)}, pages 419--431. Springer, New York,
  1993.

\bibitem[KSW21]{KabluSartoWig2022}
Zakhar Kabluchko, Andrea Sartori, and Igor Wigman.
\newblock Expected nodal volume for non-gaussian random band-limited functions,
  2021.

\bibitem[KWY21]{WigYesha2021}
Par Kurlberg, Igor Wigman, and Nadav Yesha.
\newblock The defect of toral laplace eigenfunctions and arithmetic random
  waves: Toral defect.
\newblock {\em Nonlinearity}, July 2021.

\bibitem[Let16]{LETENDRE2016}
Thomas Letendre.
\newblock Expected volume and euler characteristic of random submanifolds.
\newblock {\em Journal of Functional Analysis}, 270(8):3047--3110, 2016.

\bibitem[LL16a]{LeLu:gap}
Antonio Lerario and Erik Lundberg.
\newblock Gap probabilities and {B}etti numbers of a random intersection of
  quadrics.
\newblock {\em Discrete Comput. Geom.}, 55(2):462--496, 2016.

\bibitem[LL16b]{Lerariolemniscate}
Antonio Lerario and Erik Lundberg.
\newblock On the geometry of random lemniscates.
\newblock {\em Proc. Lond. Math. Soc. (3)}, 113(5):649--673, 2016.

\bibitem[LM21]{lerario2021probabilistic}
Antonio Lerario and Leo Mathis.
\newblock Probabilistic schubert calculus: Asymptotics.
\newblock {\em Arnold Mathematical Journal}, 7(2):169--194, 2021.

\bibitem[LS19a]{dtgrf}
A.~Lerario and M.~Stecconi.
\newblock Differential topology of {G}aussian random fields.
\newblock {\em Preprint ArXiv:1902.03805}, 2019.

\bibitem[LS19b]{stec2019MaxTyp}
A.~Lerario and M.~Stecconi.
\newblock Maximal and typical topology of real polynomial singularities.
\newblock {\em Ann.Inst.Fourier, in press, arxiv:1906.04444}, 2019.

\bibitem[Maf17]{MAFFUCCI2017}
Riccardo~W. Maffucci.
\newblock Nodal intersections for random waves against a segment on the
  3-dimensional torus.
\newblock {\em Journal of Functional Analysis}, 272(12):5218--5254, 2017.

\bibitem[Mar21]{Ma2021surveyECS}
Domenico Marinucci.
\newblock Some recent developments on the geometry of random spherical
  eigenfunctions, 2021.

\bibitem[MP11]{marinucci_peccati_2011}
Domenico Marinucci and Giovanni Peccati.
\newblock {\em Random Fields on the Sphere: Representation, Limit Theorems and
  Cosmological Applications}.
\newblock London Mathematical Society Lecture Note Series. Cambridge University
  Press, 2011.

\bibitem[MPRW15]{MaPeRoWi2015}
Domenico Marinucci, Giovanni Peccati, Maurizia Rossi, and Igor Wigman.
\newblock Non-universality of nodal length distribution for arithmetic random
  waves.
\newblock {\em Geometric and Functional Analysis}, 26:926--960, 2015.

\bibitem[MRV21]{MarRossVidotto2021}
Domenico Marinucci, Maurizia Rossi, and Anna Vidotto.
\newblock {Non-universal fluctuations of the empirical measure for isotropic
  stationary fields on ${\mathbb{S}^{2}}\times \mathbb{R}$}.
\newblock {\em The Annals of Applied Probability}, 31(5):2311 -- 2349, 2021.

\bibitem[MRW20]{MaRoWi2020Berry}
Domenico Marinucci, Maurizia Rossi, and Igor Wigman.
\newblock {The asymptotic equivalence of the sample trispectrum and the nodal
  length for random spherical harmonics}.
\newblock {\em Annales de l'Institut Henri Poincaré, Probabilités et
  Statistiques}, 56(1):374 -- 390, 2020.

\bibitem[MSS14]{MSSzoneq}
Ilya Molchanov, Michael Schmutz, and Kaspar Stucki.
\newblock {Invariance properties of random vectors and stochastic processes
  based on the zonoid concept}.
\newblock {\em Bernoulli}, 20(3):1210 -- 1233, 2014.

\bibitem[MW11a]{MaWi2011defect}
Domenico Marinucci and Igor Wigman.
\newblock The defect variance of random spherical harmonics.
\newblock {\em Journal of Physics A: Mathematical and Theoretical},
  44(35):355206, aug 2011.

\bibitem[MW11b]{MaWi2011excarea}
Domenico Marinucci and Igor Wigman.
\newblock On the area of excursion sets of spherical gaussian eigenfunctions.
\newblock {\em Journal of Mathematical Physics}, 52(9):093301, 2011.

\bibitem[MW14]{MaWi2014}
Domenico Marinucci and Igor Wigman.
\newblock On nonlinear functionals of random spherical eigenfunctions.
\newblock {\em Communications in Mathematical Physics}, 327, 05 2014.

\bibitem[Nic16a]{Nicolaescu2016}
Liviu~I. Nicolaescu.
\newblock A stochastic {G}auss--{B}onnet--{C}hern formula.
\newblock {\em Probability Theory and Related Fields}, 165(1):235--265, Jun
  2016.

\bibitem[Nic16b]{Nicolaescu}
Liviu~I. Nicolaescu.
\newblock A stochastic {G}auss-{B}onnet-{C}hern formula.
\newblock {\em Probab. Theory Related Fields}, 165(1-2):235--265, 2016.

\bibitem[Nic20]{nicolaescu2020lectures}
L.I. Nicolaescu.
\newblock {\em Lectures On The Geometry Of Manifolds (Third Edition)}.
\newblock World Scientific Publishing Company, 2020.

\bibitem[Not21]{MaxNotarnickHermite}
Massimo Notarnicola.
\newblock Matrix hermite polynomials, random determinants and the geometry of
  gaussian fields, 2021.

\bibitem[NPR19]{NourdinPeccatiRossi2019}
I.~Nourdin, G.~Peccati, and M.~Rossi.
\newblock Nodal statistics of planar random waves.
\newblock {\em Comm. Math. Phys.}, 369(1):99--151, 2019.

\bibitem[NS09]{NazarovSodin2009}
F.~Nazarov and M.~Sodin.
\newblock On the number of nodal domains of random spherical harmonics.
\newblock {\em Amer. J. Math.}, 131(5):1337--1357, 2009.

\bibitem[NS16a]{NazarovSodin2016}
F.~Nazarov and M.~Sodin.
\newblock Asymptotic laws for the spatial distribution and the number of
  connected components of zero sets of {G}aussian random functions.
\newblock {\em Zh. Mat. Fiz. Anal. Geom.}, 12(3):205--278, 2016.

\bibitem[NS16b]{NicSavale}
Liviu~I. Nicolaescu and Nikhil Savale.
\newblock The {G}auss-{B}onnet-{C}hern theorem: a probabilistic perspective.
\newblock {\em Probab. Theory Related Fields}, 369(4):2951--2986, 2016.

\bibitem[Par05]{Parth}
K.R. Parthasarathy.
\newblock {\em Probability Measures on Metric Spaces}.
\newblock Ams Chelsea Publishing. Academic Press, 2005.

\bibitem[PF08]{PaivaGelfCrof}
J.~C.~{\'A}lvarez Paiva and E.~Fernandes.
\newblock Gelfand transforms and crofton formulas.
\newblock {\em Selecta Mathematica}, 13(3):369, Feb 2008.

\bibitem[RW16]{WigRud2016}
Zeev Rudnick and Igor Wigman.
\newblock Nodal intersections for random eigenfunctions on the torus.
\newblock {\em Amer. J. Math}, 138(6):1605--1644, December 2016.

\bibitem[Sar42]{sard}
Arthur Sard.
\newblock {The measure of the critical values of differentiable maps}.
\newblock {\em Bulletin of the American Mathematical Society}, 48(12):883 --
  890, 1942.

\bibitem[Sch01]{SchneiderCrofton}
R.~Schneider.
\newblock Crofton formulas in hypermetric projective finsler spaces.
\newblock {\em Archiv der Mathematik}, 77(1):85--97, Jul 2001.

\bibitem[Sch14]{bible}
Rolf Schneider.
\newblock {\em Convex bodies: the {B}runn-{M}inkowski theory}, volume 151 of
  {\em Encyclopedia of Mathematics and its Applications}.
\newblock Cambridge University Press, Cambridge, expanded edition, 2014.

\bibitem[Spi79]{spivak}
Michael Spivak.
\newblock {\em A comprehensive introduction to differential geometry. {V}ol.
  {I}}.
\newblock Publish or Perish, Inc., Wilmington, Del., second edition, 1979.

\bibitem[SS93a]{shsm}
M.~Shub and S.~Smale.
\newblock Complexity of {B}ezout's theorem. {II}. {V}olumes and probabilities.
\newblock In {\em Computational algebraic geometry ({N}ice, 1992)}, volume 109
  of {\em Progr. Math.}, pages 267--285. Birkh\"auser Boston, Boston, MA, 1993.

\bibitem[SS93b]{ShSm1}
Michael Shub and Steve Smale.
\newblock Complexity of {B}\'ezout's theorem. {I}. {G}eometric aspects.
\newblock {\em J. Amer. Math. Soc.}, 6(2):459--501, 1993.

\bibitem[SS93c]{ShSm3}
Michael Shub and Steve Smale.
\newblock Complexity of {B}ezout's theorem. {III}. {C}ondition number and
  packing.
\newblock {\em J. Complexity}, 9(1):4--14, 1993.
\newblock Festschrift for Joseph F. Traub, Part I.

\bibitem[Ste21]{stecconi2021isotropic}
M.~Stecconi.
\newblock Isotropic random spin weighted functions on {$\Bbb S^2$} vs isotropic
  random fields on {$\Bbb{S}^3$}.
\newblock {\em Th.Prob.Math.Stat., in press, arXiv:2108.00736}, 2021.

\bibitem[Ste22]{KRStec}
M.~Stecconi.
\newblock Kac-{R}ice formula for transverse intersections.
\newblock {\em Analysis and Mathematical Physics}, 12(2):44, 2022.

\bibitem[SW19]{SarnakWigman2019}
P.~Sarnak and I.~Wigman.
\newblock Topologies of nodal sets of random band-limited functions.
\newblock {\em Comm. Pure Appl. Math.}, 72(2):275--342, 2019.

\bibitem[Vit91]{Vitale}
Richard~A. Vitale.
\newblock Expected absolute random determinants and zonoids.
\newblock {\em Ann. Appl. Probab.}, 1(2):293--300, 1991.

\bibitem[Wig10]{Wigman_2010}
I.~Wigman.
\newblock Fluctuations of the nodal length of random spherical harmonics.
\newblock {\em Communications in Mathematical Physics}, 298(3):787–831, Jun
  2010.

\bibitem[Wig11]{WigSurvey2010}
Igor Wigman.
\newblock On the nodal lines of random and deterministic laplace
  eigenfunctions, 2011.

\bibitem[Wig22]{WigSurvey2022}
Igor Wigman.
\newblock On the nodal structures of random fields -- a decade of results,
  2022.

\bibitem[Zel09]{Zelditch2009RealAC}
Steve Zelditch.
\newblock Real and complex zeros of riemannian random waves.
\newblock {\em arXiv: Spectral Theory}, 2009.

\end{thebibliography}

\end{document}